\def\DateTime{06/December/2010, 8:00 (Kyoto)}
\def\Version{Version 1.80}
\def\yes{\if00}
\def\no{\if01}
\def\iftenpt{\no}
\def\ifelevenpt{\no}
\def\iftwelvept{\yes}

\def\ifusepdf{\no}
\def\ifpsfont{\no}
\def\ifpxfont{\yes}
\def\ifquery{\no}

\iftenpt
\documentclass[leqno]{amsart}
\else\fi
\ifelevenpt
\documentclass[leqno,11pt]{amsart}
\else\fi
\iftwelvept
\documentclass[leqno,12pt]{amsart}
\else\fi

\usepackage{amssymb}
\usepackage{amscd}
\usepackage{latexsym}
\usepackage{verbatim}
\usepackage{mathrsfs}
\usepackage[all]{xy}

\ifusepdf
\usepackage{hyperref}
\else\fi
\ifpsfont
\usepackage[T1]{fontenc}
\usepackage{times}
\else\fi
\ifpxfont
\usepackage{pxfonts}
\else\fi

\ifelevenpt
\else\fi

\iftwelvept
\else\fi

\theoremstyle{plain}
\newtheorem{Theorem}{Theorem}[section]
\newtheorem{Proposition}[Theorem]{Proposition}
\newtheorem{Lemma}[Theorem]{Lemma}
\newtheorem{PropDef}[Theorem]{Proposition-Definition}

\newtheorem{Claim}{Claim}[Theorem]

\theoremstyle{definition}

\newtheorem{Remark}[Theorem]{Remark}

\def\rom{\textup}
\newcommand{\ZZ}{{\mathbb{Z}}}
\newcommand{\QQ}{{\mathbb{Q}}}
\newcommand{\RR}{{\mathbb{R}}}

\newcommand{\CC}{{\mathbb{C}}}
\newcommand{\PP}{{\mathbb{P}}}

\newcommand{\OO}{{\mathscr{O}}}

\newcommand{\Proj}{\operatorname{Proj}}
\newcommand{\aPic}{\widehat{\operatorname{Pic}}}
\newcommand{\Rat}{\operatorname{Rat}}

\newcommand{\Hom}{\operatorname{Hom}}

\newcommand{\Spec}{\operatorname{Spec}}

\newcommand{\Supp}{\operatorname{Supp}}
\newcommand{\length}{\operatorname{length}}

\newcommand{\mult}{\operatorname{mult}}

\newcommand{\acherncl}{\widehat{{c}}}
\newcommand{\adeg}{\widehat{\operatorname{deg}}}
\newcommand{\zeros}{\operatorname{div}}

\newcommand{\Mv}{\operatorname{Mv}}
\newcommand{\Fx}{\operatorname{Fx}}

\newcommand{\Nef}{\operatorname{Nef}}
\newcommand{\aAmp}{\widehat{\operatorname{Amp}}}
\newcommand{\aNef}{\widehat{\operatorname{Nef}}}
\newcommand{\Num}{\operatorname{Num}}

\newcommand{\dist}{\operatorname{dist}}
\newcommand{\Div}{\operatorname{Div}}
\newcommand{\aDiv}{\widehat{\operatorname{Div}}}

\newcommand{\avol}{\widehat{\operatorname{vol}}}
\newcommand{\aH}{\hat{H}^0}
\newcommand{\ah}{\hat{h}^0}

\newcommand{\aew}{({\rm a.e.})}
\newcommand{\Tpsh}{\operatorname{PSH}}

\newcommand{\TT}{\mathscr{T}}
\newcommand{\esssup}{\operatorname{ess\, sup}}

\newcommand{\rest}[2]{\left.{#1}\right\vert_{{#2}}}
\ifquery
\def\query#1{\setlength\marginparwidth{65pt} 
\marginpar{\raggedright\fontsize{7.81}{9} 
\selectfont\upshape\hrule\smallskip 
#1\par\smallskip\hrule}} 

\else
\def\query#1{}
\fi

\begin{document}

\title[Zariski decompositions on arithmetic surfaces]%
{Zariski decompositions on arithmetic surfaces}
\author{Atsushi Moriwaki}
\address{Department of Mathematics, Faculty of Science,
Kyoto University, Kyoto, 606-8502, Japan}
\email{moriwaki@math.kyoto-u.ac.jp}
\date{\DateTime, (\Version)}
\begin{abstract}
In this paper, we establish the Zariski decompositions of arithmetic $\RR$-divisors of continuous type on arithmetic surfaces
and investigate several properties.
We also develop the general theory of arithmetic $\RR$-divisors on arithmetic varieties.
\end{abstract}


\maketitle

\setcounter{tocdepth}{1}
\tableofcontents


\section*{Introduction}
\renewcommand{\theTheorem}{\Alph{Theorem}}
Let $S$ be a non-singular projective surface over an algebraically closed field and let
$\Div(S)$ be the group of divisors on $S$.
An element of $\Div(S)\otimes_{\ZZ} \RR$
\query{$\otimes \RR$\quad
$\Longrightarrow$\quad $\otimes_{\ZZ} \RR$
(17/October/2010)}
is called an $\RR$-divisor on $S$.
In addition,
it is said to be effective if
it is a linear combination of curves with non-negative real numbers.
The problem of the Zariski decomposition for an effective $\RR$-divisor $D$ is to find a decomposition
$D = P + N$ with the following properties:
\begin{enumerate}
\renewcommand{\labelenumi}{(\arabic{enumi})}
\item
$P, N \in \Div(S)\otimes_{\ZZ} \RR$.
\query{$\otimes \RR$\quad
$\Longrightarrow$\quad $\otimes_{\ZZ} \RR$
(17/October/2010)}

\item
$P$ is nef, that is,
$(P \cdot C) \geq 0$ for all reduced and irreducible curves $C$ on $S$.

\item
$N$ is effective.

\item
Assuming $N \not= 0$,
let $N = c_1 C_1 + \cdots + c_l C_l$ be the decomposition such that
$c_1, \ldots, c_l \in \RR_{>0}$ and
$C_1, \ldots, C_l$ are reduced and irreducible curves on $S$.
Then the following (4.1) and (4.2) hold:
\begin{enumerate}
\renewcommand{\labelenumii}{(\arabic{enumi}.\arabic{enumii})}
\item $(P \cdot C_i) = 0$ for all $i$.

\item The $l \times l$ matrix given by $\left((C_i \cdot C_j)\right)_{\substack{1 \leq i \leq l \\ 1 \leq j \leq l}}$ is negative definite.
\end{enumerate}
\end{enumerate}
In 1962, Zariski \cite{Zariski} established the decomposition in the case where $D \in \Div(S)$.
By the recent work due to Bauer \cite{Bauer} (see also Section~\ref{sec:Zariski:decomp:vec:space}),
$P$ is characterized by the greatest element  in
\[
\{ M \in \Div(S)\otimes_{\ZZ} \RR \mid \text{$D - M$ is effective and $M$ is nef} \}.
\]
\query{$\otimes \RR$\quad
$\Longrightarrow$\quad $\otimes_{\ZZ} \RR$
(17/October/2010)}
In this paper, we would like to consider an arithmetic analogue of the above problem on
an arithmetic surface. In order to make the main theorem clear,
we need to introduce a lot of concepts and terminology.

\subsection*{$\bullet$ Green functions for $\RR$-divisors}
Let $V$ be an equidimensional smooth projective variety over $\CC$.
An element of $\Div(V)_{\RR} := \Div(V) \otimes_{\ZZ} \RR$
\query{$\otimes \RR$\quad
$\Longrightarrow$\quad $\otimes_{\ZZ} \RR$
(17/October/2010)}
is called an {\em $\RR$-divisor} on $V$.
For an $\RR$-divisor $D$ on $V$,
we would like to introduce several types of Green functions for $D$.
We set $D = a_1 D_1 + \cdots + a_l D_l$,
where $a_1, \ldots, a_l \in \RR$ and $D_i$'s are reduced and irreducible divisors on $V$.
Let $g : V \to \RR \cup \{ \pm \infty \}$ be a  locally integrable function on $V$.
We say $g$ is a {\em $D$-Green function of $C^{\infty}$-type} (resp
a {\em $D$-Green function of $C^0$-type}) on $V$ if, for each point $x \in V$,
there are a small open neighborhood $U_x$ of $x$, 
local equations  $f_1, \ldots, f_l$ of $D_1, \ldots, D_l$ over $U_x$ respectively and
a $C^{\infty}$-function (resp. continuous function) $u_x$ over $U_x$ such that
\[
g = u_x + \sum_{i=1}^l (-a_i) \log \vert f_i \vert^2\quad\aew
\]
holds on $U_x$.
These definitions are counterparts of $C^{\infty}$-metrics and
continuous metrics.
Besides them,
it is necessary to introduce
a degenerated version of semipositive metrics.
We say $g$ is a {\em $D$-Green function of $\Tpsh_{\RR}$-type} on $V$ if the above
$u_x$ is taken as a real valued plurisubharmonic function on $U_x$ (i.e., $u_x$ is a plurisubharmonic function
on $U_x$ and $u_x(y) \in \RR$ for all $y \in U_x$).
To say more generally, let $\mathcal{L}^1_{\rm loc}$ be the sheaf consisting of locally integrable functions, that is,
\[
\mathcal{L}^1_{\rm loc}(U) = \{ g : U \to \RR \cup \{\pm \infty \} \mid \text{$g$ is locally integrable} \}
\]
for an open set $U$ of $V$, and let us fix  a subsheaf $\TT$ of $\mathcal{L}^1_{\rm loc}$ satisfying the following conditions
(in the following (1), (2) and (3), $U$ is an arbitrary open set of $V$):
\begin{enumerate}
\renewcommand{\labelenumi}{(\arabic{enumi})}
\item
If $u, v \in \TT(U)$ and $a \in \RR_{\geq 0}$,
then $u + v \in \TT(U)$ and $a u \in \TT(U)$.

\item
If $u, v \in \TT(U)$ and $u \leq v$ almost everywhere, then $u \leq v$.

\item
If $\phi \in \OO_V^{\times}(U)$
\query{$\mathcal{O}$\quad
$\Longrightarrow$\quad $\mathscr{O}$ throughout the paper
(20/October/2010)}
(i.e., $\phi$ is a  nowhere vanishing holomorphic function on $U$),
then 
$\log \vert \phi  \vert^2 \in \TT(U)$.
\end{enumerate}
This subsheaf $\TT$ is called a {\em type for Green functions on $V$}.
Moreover, $\TT$ is said to be {\em real valued} if
$u(x) \in \RR$ for any open set $U$, $u \in \TT(U)$ and $x \in U$.
Using $\TT$,
we say $g$ is a {\em $D$-Green function of $\TT$-type} on $V$
if the above $u_x$ is an element of  $\TT(U_x)$ for each $x \in V$.
The set of all $D$-Green functions of $\TT$-type on $V$ is denoted by $G_{\TT}(V;D)$.
If $x \not\in \Supp(D)$, then, by using (2) and (3) in the properties of $\TT$, we can see that the value
\[
u_x(x) + \sum_{i=1}^l (-a_i) \log \vert f_i(x) \vert^2
\]
does not depend on the choice of the local expression
\[
g = u_x + \sum_{i=1}^l (-a_i) \log \vert f_i \vert^2\quad\aew
\]
of $g$, so that
$u_x(x) + \sum_{i=1}^l (-a_i) \log \vert f_i(x) \vert^2$ is called the {\em canonical value of $g$ at $x$} and
it is denoted by $g_{\rm can}(x)$. Note that $g_{\rm can} \in \TT(V \setminus \Supp(D))$ and
$g = g_{\rm can}\ \aew$ on $V \setminus \Supp(D)$. Further, if $\TT$ is real valued,
then $g_{\rm can}(x) \in \RR$.

\subsubsection*{$\star$ $H^0(V, D)$ for an $\RR$-divisor $D$ and its norm arising from a Green function}
Let $D$ be an $\RR$-divisor.
If $V$ is connected, then $H^0(V,D)$ is defined by
\[
H^0(V, D) := \left\{ \phi \ \left|\ 
\begin{array}{l} \text{$\phi$ is a non-zero rational function} \\
\text{on $V$ with $(\phi) + D \geq 0$} 
\end{array}
\right\}\right. 
\cup \{ 0 \}.
\]
In general, let $V = V_1 \cup \cdots \cup V_{r}$ be the decomposition into
connected components of $V$.
Then
\[
H^0(V, D) := \bigoplus_{i=1}^r H^0(V_i, \rest{D}{V_i}).
\]
Let $g$ be a $D$-Green function of $C^0$-type on $V$.
For $\phi \in H^0(V, D)$, it is easy to see that
$\vert \phi \vert_g := \exp(-g/2)\vert \phi \vert$ coincides with a continuous function almost everywhere on $V$, so that the supremum norm
$\Vert \phi \Vert_g$ of $\phi$ with respect to $g$ is defined by
\query{${\operatorname{ess\ sup}}$ $\Longrightarrow$ $\esssup$ (2010/12/06)}
\[
\Vert \phi \Vert_g :=\esssup \left\{ \vert \phi \vert_g(x) \mid x \in V \right\}.
\]

\subsection*{$\bullet$ Arithmetic $\RR$-divisors}
Let $X$ be a $d$-dimensional generically smooth normal projective arithmetic variety.
Let $\Div(X)$ be the group of Cartier divisors on $X$.
As before, an element of $\Div(X)_{\RR} := \Div(X) \otimes_{\ZZ} \RR$
\query{$\otimes \RR$\quad
$\Longrightarrow$\quad $\otimes_{\ZZ} \RR$
(17/October/2010)}
is called an {\em $\RR$-divisor} on $X$.
It is said to be {\em effective} if it is a linear combination of prime divisors with non-negative real numbers.
In addition, for $D, E \in \Div(X)_{\RR}$,
if $D - E$ is effective, then it is denoted by $D \geq E$ or $E \leq D$.

Let $D$ be an $\RR$-divisor on $X$ and let $g$ be a locally integrable function on $X(\CC)$.
A pair $\overline{D} = (D, g)$ is called an {\em arithmetic $\RR$-divisor on $X$}
\query{Add ``on $X$''
(31/October/2010)}
if $F_{\infty}^*(g) = g\ \aew$,
where $F_{\infty}$ is the complex conjugation map on $X(\CC)$.
Moreover, $\overline{D}$ is said to be {\em of
$C^{\infty}$-type} (resp. {\em of $C^0$-type}, {\em of $\Tpsh_{\RR}$-type})
if $g$ is a $D$-Green function of $C^{\infty}$-type (resp.
of $C^0$-type, of $\Tpsh_{\RR}$-type).
More generally, for a fixed type  $\TT$ for Green functions,
$\overline{D}$ is said to be {\em  of
$\TT$-type}
if $g$ is a $D$-Green function of $\TT$-type.
For arithmetic $\RR$-divisors $\overline{D}_1 = (D_1, g_1)$ and $\overline{D}_2 = (D_2, g_2)$,
we define $\overline{D}_1 = \overline{D}_2$ and $\overline{D}_1 \leq \overline{D}_2$ as follows:
\[
\begin{cases}
\overline{D}_1 = \overline{D}_2\quad\overset{\text{def}}{\Longleftrightarrow}\quad\text{$D_1 = D_2$ and $g_1 = g_2\ \aew$},\\
\overline{D}_1 \leq \overline{D}_2\quad\overset{\text{def}}{\Longleftrightarrow}\quad\text{$D_1 \leq D_2$ and $g_1 \leq g_2\ \aew$}.
\end{cases}
\]
If $\overline{D} \geq (0,0)$, then $\overline{D}$ is said to be {\em effective}.
Further, the set 
\[
\{ \overline{M} \mid \text{$\overline{M}$ is an arithmetic $\RR$-divisor on $X$ and $\overline{M} \leq \overline{D}$} \}
\]
is denoted by $(-\infty, \overline{D}]$.

\subsubsection*{$\star$ Volume of arithmetic $\RR$-divisors of $C^0$-type}
Let $\aDiv_{C^0}(X)_{\RR}$ be the group of arithmetic $\RR$-divisors of $C^0$-type on $X$.
For $\overline{D} \in \aDiv_{C^0}(X)_{\RR}$, we define
$H^0(X, D)$, $\aH(X, \overline{D})$, $\ah(X, \overline{D})$ and $\avol(\overline{D})$ as follows:
\[
\begin{cases}
H^0(X, D)  :=  
\left\{ \psi \ \left|\ 
\begin{array}{l} \text{$\psi$ is a non-zero rational function} \\
\text{on $X$ with $(\psi) + D \geq 0$} 
\end{array}
\right\}\right. 
\cup \{ 0 \}, \\
\aH(X, \overline{D})  := \{ \psi \in H^0(X, D) \mid \Vert \psi \Vert_g \leq 1 \}, \\
\ah(X, \overline{D})  := \log \# (\aH(X, \overline{D})),\\
{\displaystyle \avol(\overline{D})  := \limsup_{n\to\infty} \frac{\ah(X, n\overline{D})}{n^d/d!}}.
\end{cases}
\]
Note that
\[
\aH(X, \overline{D})  =  
\left\{ \psi \ \left|\ 
\begin{array}{l} \text{$\psi$ is a non-zero rational function} \\
\text{on $X$ with $\widehat{(\psi)} + \overline{D} \geq (0,0)$} 
\end{array}
\right\}\right. \cup \{ 0 \}.
\]
The continuity of 
\[
\avol : \aPic(X)_{\QQ}  \to \RR
\]
is proved in \cite{MoCont}, where $\aPic(X)_{\QQ} := \aPic(X) \otimes_{\ZZ} \QQ$.
\query{$\otimes \QQ$\quad
$\Longrightarrow$\quad $\otimes_{\ZZ} \QQ$
(17/October/2010)}
Moreover,
in \cite{MoContExt}, we introduce $\aPic_{C^0}(X)_{\RR}$ as a natural extension of
$\aPic(X)_{\QQ}$ (for details, see \cite{MoContExt} or Subsection~\ref{subsec:def:arithmetic:R:divisor}) 
\query{add a comment : ``(for details, see \cite{MoContExt} or Subsection~\ref{subsec:def:arithmetic:R:divisor})''
(17/October/2010)}
and prove that $\avol : \aPic(X) _{\QQ} \to \RR$ has the continuous
extension 
\[
\avol : \aPic_{C^0}(X)_{\RR} \to \RR.
\]
Theorem~\ref{thm:aDiv:aPic:R} shows that there is a natural surjective homomorphism
\[
\overline{\OO}_{\RR} : \aDiv_{C^0}(X)_{\RR} \to \aPic_{C^0}(X)_{\RR}
\]
such that $\avol(\overline{D}) = \avol(\overline{\OO}_{\RR}(\overline{D}))$ for all $\overline{D} \in \aDiv_{C^0}(X)_{\RR}$.
In particular, by using results in \cite{HChen}, \cite{HChenFujita}, \cite{MoCont}, \cite{MoContExt}, \cite{MoArLin} and \cite{YuanVol},
we have the following properties of $\avol : \aDiv_{C^0}(X)_{\RR} \to \RR$ 
(cf. Theorem~\ref{thm:aDiv:aPic:R} and Theorem~\ref{thm:gen:Hodge:index}):
\begin{enumerate}
\renewcommand{\labelenumi}{(\arabic{enumi})}
\item
$\avol : \aDiv_{C^0}(X)_{\RR} \to \RR$ is positively homogeneous of degree $d$, that is,
$\avol(a\overline{D}) = a^d \avol(\overline{D})$ for all $a \in \RR_{\geq 0}$ and $\overline{D} \in \aDiv_{C^0}(X)_{\RR}$.

\item
$\avol : \aDiv_{C^0}(X)_{\RR} \to \RR$ is continuous in the following sense:
\query{Change the statement of (2)
(11/November/2010)}
Let $\overline{D}_1, \ldots, \overline{D}_r$, $\overline{A}_1, \ldots, \overline{A}_{r'}$ be arithmetic $\RR$-divisors of $C^0$-type.
For a compact set $B$ in $\RR^r$ and a positive number $\epsilon$, there are positive numbers $\delta$ and $\delta'$ such that,
for all $a_1, \ldots, a_r, \delta_1, \ldots, \delta_{r'} \in \RR$ and $\phi \in C^0(X)$
with $(a_1, \ldots, a_r) \in B$, $\sum_{j=1}^{r'} \vert \delta_j \vert \leq \delta$ and $\Vert \phi \Vert_{\sup} \leq \delta'$,
we have
\[
\hskip3em
\left\vert \avol\left(\sum_{i=1}^r a_i \overline{D}_i + \sum_{j=1}^{r'} \delta_j \overline{A}_j + (0, \phi) \right) -
\avol\left(\sum_{i=1}^r a_i \overline{D}_i \right) \right\vert
\leq \epsilon.
\]
Moreover, if  $\overline{D}_1, \ldots, \overline{D}_r$, $\overline{A}_1, \ldots, \overline{A}_{r'}$ are $C^{\infty}$, then
there is a positive constant $C$ depending only on $X$ and
$
\overline{D}_1, \ldots, \overline{D}_r, \overline{A}_1, \ldots, \overline{A}_{r'}
$
such that
\begin{multline*}
\hskip3em
\left\vert \avol\left(\sum_{i=1}^r a_i \overline{D}_i + \sum_{j=1}^{r'} \delta_j \overline{A}_j + (0, \phi) \right) -
\avol\left(\sum_{i=1}^r a_i \overline{D}_i \right) \right\vert \\
\leq C\left( \sum_{i=1}^r \vert a_i \vert + \sum_{j=1}^{r'} \vert \delta_j \vert \right)^{d-1} \left(
\Vert \phi \Vert_{\sup} + \sum_{j=1}^{r'} \vert \delta_j \vert \right)
\end{multline*}
for all $a_1, \ldots, a_r, \delta_1, \ldots, \delta_{r'} \in \RR$ and $\phi \in C^0(X)$.

\item
$\avol(\overline{D})$ is given by ``$\lim$'', that is,
\[
\avol(\overline{D}) = \lim_{t \to\infty} \frac{\ah(t \overline{D})}{t^d/d!},
\]
where $\overline{D} \in \aDiv_{C^0}(X)_{\RR}$ and $t \in \RR_{>0}$.

\item
$\avol(-)^{1/d}$ is concave, that is, for arithmetic $\RR$-divisors $\overline{D}_1, \overline{D}_2$ of $C^0$-type, 
if $\overline{D}_1$ and $\overline{D}_2$ are pseudo-effective 
(for the definition of pseudo-effectivity, see SubSection~\ref{subsec:def:positivity:arith:div}),%
\query{$\ah(X, n_1\overline{D}_1) \not= 0$ and $\ah(X, n_2\overline{D}_2) \not= 0$ for some $n_1, n_2 \in \ZZ_{>0}$\quad
$\Longrightarrow$\quad
$\overline{D}_1$ and $\overline{D}_2$ are pseudo-effective (for the definition of pseudo-effectivity, see SubSection~\ref{subsec:def:positivity:arith:div})\\
(28/February/2010)}
then
\[
\avol(\overline{D}_1 + \overline{D}_2)^{1/d} \geq \avol(\overline{D}_1)^{1/d} + \avol(\overline{D}_2)^{1/d}.
\]

\item
\rom{(}Fujita's approximation theorem for $\RR$-divisors\rom{)}
If $\overline{D}$ is an arithmetic $\RR$-divisor of $C^0$-type and
$\avol(\overline{D}) > 0$, then, for any positive number $\epsilon$, there are a birational morphism $\mu : Y \to X$ of
generically smooth and normal projective arithmetic varieties and an ample arithmetic $\QQ$-divisor $\overline{A}$ of $C^{\infty}$-type on $Y$ \rom{(}cf. Section~\rom{\ref{sec:positivity:arith:div}}\rom{)}
such that $\overline{A} \leq \mu^*(\overline{D})$ and $\avol(\overline{A}) \geq \avol(\overline{D}) - \epsilon$.

\item \rom{(}The generalized Hodge index theorem for $\RR$-divisors\rom{)} If $\overline{D}$ is an arithmetic $\RR$-divisor
of $(C^0 \cap \Tpsh)$-type and $D$ is nef on every fiber of $X \to \Spec(\ZZ)$,
\query{$(\Tpsh \cap C^0)$-type\quad$\Longrightarrow$\quad
$(C^0 \cap \Tpsh)$-type,\quad
the same rules are applied for
$(\Tpsh \cap C^{\infty})$-type as below\\
(28/February/2010)}
then 
$\avol(\overline{D}) \geq \adeg(\overline{D}^d)$  \rom{(}see
descriptions in ``Positivity of arithmetic $\RR$-divisors'' below or Proposition~\rom{\ref{prop:intersection:Div:nef:C:0}}
for the definition of $\adeg(\overline{D}^d)$\rom{)}.
\end{enumerate}

\subsubsection*{$\star$ Intersection number of an arithmetic $\RR$-divisor with a $1$-dimensional subscheme}
Let $\TT$ be a real valued type for Green functions such that $C^0 \subseteq \TT$ and $-u \in \TT$ whenever $u \in \TT$.
Let $\overline{D} =(D, g)$ be an arithmetic $\RR$-divisor of $\TT$-type. 
Let $C$ be a $1$-dimensional closed  integral subscheme of $X$.
Let $D = a_1 D_1 + \cdots + a_l D_l$ be a decomposition
such that $a_1, \ldots, a_l \in \RR$ and $D_i$'s are Cartier divisors.
For simplicity, we assume that $D_i$'s are effective,
$C \not\subseteq \Supp(D_i)$ for all $i$ and that $C$ is flat over $\ZZ$.
In this case, $\adeg(\rest{\overline{D}}{C})$ is defined by
\[
\adeg(\rest{\overline{D}}{C}) := \sum_{i=1}^l a_i \log \#(\OO_C(D_i)/\OO_C) + \frac{1}{2} \sum_{x \in C(\CC)} g_{\rm can}(x).
\]
In general, see Section~\ref{subsec:intersection:arithmetic:R:divisor:curve}. 
Let $Z$ be a $1$-cycle on $X$ with coefficients in $\RR$, that is,
there are $a_1, \ldots, a_l \in \RR$ and $1$-dimensional closed integral subschemes $C_1, \ldots, C_l$ on $X$ such that
$Z = a_1 C_1 + \cdots + a_l C_l$. Then $\adeg\left(\overline{D} \mid Z\right)$ is defined by
\[
\adeg\left(\overline{D} \mid Z\right) := \sum_{i=1}^l a_i \adeg \left( \rest{\overline{D}}{C_i} \right).
\]

\subsubsection*{$\star$ Positivity of arithmetic $\RR$-divisors}
An arithmetic $\RR$-divisor $\overline{D}$ is called {\em nef}  
if $\overline{D}$ is of $\Tpsh_{\RR}$-type and
$\adeg(\rest{D}{C}) \geq 0$ for all  $1$-dimensional closed integral subschemes $C$ of $X$.
The cone of all nef arithmetic $\RR$-divisors on $X$ is denoted by $\aNef(X)_{\RR}$.
Moreover, the cone of all nef arithmetic $\RR$-divisors of $C^{\infty}$-type (resp. $C^0$-type) on $X$ is denoted by $\aNef_{C^{\infty}}(X)_{\RR}$
(resp. $\aNef_{C^{0}}(X)_{\RR}$).
Further, we say $\overline{D}$ is {\em big}
\query{big \quad$\Longrightarrow$\quad{\em big}\\ (28/February/2010)}
if $\avol(\overline{D}) > 0$.

Let $\aDiv_{C^0}^{\Nef}(X)_{\RR}$ be the vector subspace of $\aDiv_{C^0}(X)_{\RR}$ generated by
$\aNef_{C^0}(X)_{\RR}$. Then, by Proposition~\ref{prop:intersection:Div:nef:C:0},  
\[
\aDiv_{C^{\infty}}(X)_{\RR} + \aDiv_{C^0 \cap \Tpsh}(X)_{\RR}\subseteq \aDiv_{C^0}^{\Nef}(X)_{\RR}
\]
and
the symmetric multi-linear map 
\[
\aDiv_{C^{\infty}}(X)_{\RR} \times \cdots \times  \aDiv_{C^{\infty}}(X)_{\RR} \to \RR
\]
given by
$(\overline{D}_1, \ldots, \overline{D}_d) \mapsto \adeg(\overline{D}_1 \cdots \overline{D}_d)$ (cf. Proposition-Definition~\ref{propdef:intersection:Div:C:infty})
extends to a unique symmetric multi-linear map 
\[
\aDiv_{C^0}^{\Nef}(X)_{\RR} \times \cdots \times  \aDiv_{C^0}^{\Nef}(X)_{\RR} \to \RR
\]
such that
$(\overline{D}, \ldots, \overline{D}) \mapsto \avol(\overline{D})$ for $\overline{D} \in \aNef_{C^0}(X)_{\RR}$.

\subsection*{$\bullet$ Zariski decompositions on arithmetic surfaces}
Let $X$ be a regular projective arithmetic surface.
The main theorem of this paper is the following:

\begin{Theorem}[cf. Theorem~\ref{thm:Zariski:decomp:arith:surface} and Theorem~\ref{thm:Zariski:decomp:big}]
Let $\overline{D}$ be an arithmetic $\RR$-divisor of $C^{0}$-type on $X$ such that the set
\[
(-\infty, \overline{D}] \cap \aNef(X)_{\RR} = \{ \overline{M} \mid \text{$\overline{M}$ is a nef arithmetic $\RR$-divisor on $X$ and $\overline{M} \leq \overline{D}$} \}
\]
is not empty.
Then there is a nef arithmetic $\RR$-divisor $\overline{P}$ of $C^0$-type such that 
$\overline{P}$ gives the greatest element of
$(-\infty, \overline{D}] \cap \aNef(X)_{\RR}$,
that is,  $\overline{P} \in (-\infty, \overline{D}] \cap \aNef(X)_{\RR}$ and $\overline{M} \leq \overline{P}$ for all 
$\overline{M} \in (-\infty, \overline{D}] \cap \aNef(X)_{\RR}$.
Moreover, if we set $\overline{N} = \overline{D} - \overline{P}$, then the following properties hold:
\begin{enumerate}
\renewcommand{\labelenumi}{(\arabic{enumi})}
\item $\avol(\overline{D}) = \avol(\overline{P}) = \adeg(\overline{P}^2)$.

\item
$\adeg (\rest{\overline{P}}{C}) = 0$ for all $1$-dimensional closed integral  subschemes $C$ with $C \subseteq \Supp(N)$.

\item
If $\overline{L}$ is an arithmetic $\RR$-divisor of $\Tpsh_{\RR}$-type on $X$ such that
$0 \leq \overline{L} \leq \overline{N}$ and $\deg(\rest{\overline{L}}{C}) \geq 0$ 
for all $1$-dimensional  closed integral subschemes $C$ with $C \subseteq \Supp(N)$,
then $\overline{L} = 0$.
\end{enumerate}
\end{Theorem}

The above decomposition $\overline{D} = \overline{P} + \overline{N}$
is called the {\em Zariski decomposition of $\overline{D}$} and we say $\overline{P}$ (resp. $\overline{N}$) is the {\em positive part} (resp.
the {\em negative part}) of the decomposition.
For example,
let $\PP^1_{\ZZ} = \Proj(\ZZ[x,y])$, $C_0 = \{ x = 0 \}$, $z = x/y$ and
$\alpha, \beta \in \RR_{>0}$ with $\alpha > 1$ and $\beta < 1$.
Then the positive part of an arithmetic divisor
\[
(C_0, -\log \vert z \vert^2 + \log \max \{ \alpha^2 \vert z \vert^2, \beta^2\})
\]
of $(C^0 \cap \Tpsh)$-type on $\PP^1_{\ZZ}$ is
\[
(\theta C_0, -\theta \log \vert z \vert^2 + \log \max \{ \alpha^2 \vert z\vert^{2\theta}, 1 \}),
\]
where $\theta = \log\alpha/(\log\alpha - \log \beta)$ (cf. Subsection~\ref{subsec:Zariski:decomp:P:1}).
This example shows that an $\RR$-divisor is necessary for the arithmetic Zariski decomposition.
In addition, an example in Remark~\ref{rem:Zariski:Faltings} shows that 
the Arakelov Chow group consisting of admissible metrics due to Faltings is insufficient to get
the Zariski decomposition.

We assume that $N \not= 0$. Let $N = c_1 C_1 + \cdots + c_l C_l$ be the decomposition of $N$ such that
$c_1, \ldots, c_l \in \RR_{>0}$ and $C_i$'s are $1$-dimensional closed integral subschemes on $X$.
Let $(C_1, g_1), \ldots, (C_l, g_l)$ be effective arithmetic divisors of $\Tpsh_{\RR}$-type
such that 
\[
c_1 (C_1, g_1) + \cdots + c_l (C_l, g_l) \leq \overline{N},
\]
which is possible by Proposition~\ref{prop:green:irreducible:decomp} and Lemma~\ref{lem:good:psh:arith:surface}.
Then, by using Lemma~\ref{lem:negative:definite}, the above (3) yields an inequality
\[
(-1)^l \det \left( \adeg\left( \rest{(C_i, g_i)}{C_j} \right)\right) > 0.
\]
This is a counterpart of the property (4.2) of the Zariski decomposition on an algebraic surface.
On the other hand, our Zariski decomposition
is a refinement of Fujita's approximation theorem due to Chen \cite{HChenFujita} and Yuan \cite{YuanVol}
on an arithmetic surface.
Actually Fujita's approximation theorem on an arithmetic surface is a consequence of the above theorem
(cf. Proposition~\ref{prop:Fujita:app}).

Let $\overline{D}$ be an effective arithmetic $\RR$-divisor of $C^0$-type.
For each $n \geq 1$, we set $F_n(\overline{D})$ and $M_n(\overline{D})$ as follows:
\[
\begin{cases}
{\displaystyle F_n(\overline{D}) = \frac{1}{n} \sum_{C} \min \left\{ \mult_C((\phi) + nD) \mid \phi \in \aH(X, n\overline{D}) \setminus \{ 0 \} \right\} C}, \\
M_n(\overline{D}) = D - F_n(\overline{D}).
\end{cases}
\]
Let $V(n\overline{D})$ be a complex vector space generated by $\aH(X, n\overline{D})$.
It is easy to see that 
\[
g_{M_n(\overline{D})} := g + \frac{1}{n} \log \dist(V(n\overline{D}); ng)
\]
is an $M_n(\overline{D})$-Green function
of $C^{\infty}$-type (for the definition of distorsion functions, see Subsection~\ref{subsec:distorsion:R:div}).
Then we have the following:

\begin{Theorem}[Asymptotic orthogonality]
If $\overline{D}$ is big, then
\[
\lim_{n\to\infty} \adeg \left.\left( \left(M_n(\overline{D}), g_{M_n(\overline{D})}\right) \ \right| \ F_n(\overline{D}) \right) = 0.
\]
\end{Theorem}

\subsection*{$\bullet$ Technical results for the proof of the arithmetic Zariski decomposition}
In order to get the greatest element of $(-\infty, \overline{D}] \cap \aNef(X)_{\RR}$,
we need to consider  the nefness of the limit of a convergent sequence of nef arithmetic $\RR$-divisors.
The following theorem is our solution for this problem:

\begin{Theorem}[cf. Theorem~\ref{thm:limit:nef:div:surface}]
Let $X$ be a regular projective arithmetic surface.
Let $\{ \overline{M}_n = (M_n, h_n) \}_{n=0}^{\infty}$ be a sequence of
nef arithmetic $\RR$-divisors on $X$ with the following properties:
\begin{enumerate}
\renewcommand{\labelenumi}{(\alph{enumi})}
\item
There is an arithmetic divisor $\overline{D} = (D, g)$ of $C^0$-type such that
$\overline{M}_n \leq \overline{D}$ for all $n \geq 1$.

\item
There is a proper closed subset $E$ of $X$ such that $\Supp(D) \subseteq E$ and
$\Supp(M_n) \subseteq E$ for all $n \geq 1$.

\item
$\lim_{n\to\infty} \mult_{C}(M_n)$ exists for all $1$-dimensional closed integral  subschemes $C$ on $X$.

\item
$\limsup_{n\to\infty} (h_n)_{\rm can}(x)$ exists in $\RR$ for all $x \in X(\CC) \setminus E(\CC)$.
\end{enumerate}
Then there is a nef arithmetic $\RR$-divisor $\overline{M} = (M, h)$ on $X$ such that $\overline{M} \leq \overline{D}$,
\[
M = \sum_{C}\left( \lim_{n\to\infty} \mult_{C}(M_n)\right) C
\]
and that $\rest{h_{\rm can}}{X(\CC) \setminus E(\CC)}$ is the upper semicontinuous regularization of the function
given by 
$x \mapsto  \limsup_{n\to\infty}(h_n)_{\rm can}(x)$
over $X(\CC) \setminus E(\CC)$.
\end{Theorem}

Moreover, for the first property $\avol(\overline{P}) = \avol(\overline{D})$ of the arithmetic Zariski decomposition,
it is necessary to observe the following behavior of distorsion functions, 
which is a consequence of Gromov's inequality for an $\RR$-divisor (cf. Proposition~\ref{prop:Gromov:ineq}).

\begin{Theorem}[cf. Theorem~\ref{thm:dist:dist:dist:ineq}]
Let $V$ be an equidimensional smooth projective variety over $\CC$ and
let $D$ be an $\RR$-divisor on $V$.
Let $R = \bigoplus_{n\geq 0} R_n$ be a graded subring of $\bigoplus_{n\geq 0} H^0(V, nD)$.
If $g$ is a $D$-Green function of $C^{\infty}$-type, then
there is a positive constant $C$ with the following properties:
\begin{enumerate}
\renewcommand{\labelenumi}{(\arabic{enumi})}
\item $\dist(R_n;ng) \leq C(n+1)^{3\dim V}$ for all $n \geq 0$.

\item
${\displaystyle \frac{\dist(R_n;ng)}{C(n+1)^{3\dim V}} \cdot \frac{\dist(R_m;mg)}{C(m+1)^{3\dim V}} \leq \frac{\dist(R_{n+m};(n+m)g)}{C(n+m+1)^{3\dim V}}}$
for all $n, m \geq 0$.
\end{enumerate}
\end{Theorem}

The most difficult point for the proof of the arithmetic Zariski decomposition is to check the continuous property of the positive part.
For this purpose, the following theorem is needed:

\begin{Theorem}[cf. Theorem~\ref{thm:cont:upper:envelope}]
Let $V$ be an equidimensional  smooth projective variety over $\CC$.
Let $A$  and $B$ be $\RR$-divisors on $V$ with $A \leq B$.
If there is an $A$-Green function $h$ of $C^{\infty}$-type
such that $dd^c([h]) + \delta_A$ is represented by either a positive $C^{\infty}$-form or the zero form,
then, for a  $B$-Green function $g_B$ of $C^{0}$-type,
there is an $A$-Green function $g$ of $(C^0 \cap \Tpsh)$-type such that
$g$ is the greatest element of the set 
\[
G_{\Tpsh}(V;A)_{\leq g_B} := \{ u \in G_{\Tpsh}(V;A) \mid u \leq g_B\ \aew \}
\]
modulo null functions, that is, $g \in G_{\Tpsh}(V;A)_{\leq g_B}$ and
$u \leq g\ \aew$  for all $u \in G_{\Tpsh}(V;A)_{\leq g_B}$.
\end{Theorem}
\query{$\Tpsh_{\RR}$ $\Longrightarrow$ $\Tpsh$
(21/September/2010)}

For the proof, we actually use a recent regularity result due to Berman-Demailly \cite{BD}.
Even starting from an arithmetic divisor $\overline{D}$ of $C^{\infty}$-type,
it is not expected that the positive part  $\overline{P}$ is of $C^{\infty}$-type again (cf \cite{MoBig}).
\query{it seems to be not expected that the positive part  $\overline{P}$ is of $C^{\infty}$-type again
\quad$\Longrightarrow$\quad
it is not expected that the positive part  $\overline{P}$ is of $C^{\infty}$-type again  (cf \cite{MoBig})\\
(28/February/2010)}
It could be that $\overline{P}$ is of $C^{1,1}$-type.

\subsection*{Acknowledgement}
I would like to thank Prof. Bauer, Prof. Caib\u{a}r and Prof. Kennedy for sending me their wonderful paper concerning
Zariski decompositions in vector spaces, which were done independently.
I also express my hearty thanks to Prof. Yuan for his questions and comments.

\renewcommand{\theTheorem}{\arabic{section}.\arabic{subsection}.\arabic{Theorem}}
\renewcommand{\theClaim}{\arabic{section}.\arabic{subsection}.\arabic{Theorem}.\arabic{Claim}}
\renewcommand{\theequation}{\arabic{section}.\arabic{subsection}.\arabic{Theorem}.\arabic{Claim}}

\section{Zariski decompositions in vector spaces}
\label{sec:Zariski:decomp:vec:space}
Logically the contexts of this section are not necessary  except Lemma~~\ref{lem:negative:definite}.
They however give an
elementary case for our considerations and provide a good overview of our paper without any materials.

\subsection{}
\label{subsec:appendix:intro}
\setcounter{Theorem}{0}
In the paper \cite{Bauer},
Bauer presents
a simple proof of the existence of Zariski decompositions on an algebraic surface.
Unfortunately, he uses liner series on the algebraic surface to show
the negative definiteness of the negative part of the Zariski decomposition.
In this section, we would like to give a linear algebraic proof without
using any materials of algebraic geometry.
The technical main result for our purpose is Lemma~\ref{lem:negative:definite}. 
After writing the first draft of this paper,
Bauer, Caib\u{a}r and Kennedy kindly informed me that, in the paper \cite{BCK},
they had independently obtained several results similar to the contexts of this section.
Their paper is written for a general reader.

Let $V$ be a vector space over $\RR$. Let
$\pmb{e} = \{ e_{\lambda} \}_{\lambda \in \Lambda}$ be a basis of $V$ and let
$\pmb{\phi} = \{ \phi_{\lambda} \}_{\lambda \in \Lambda}$ be a family of elements of
$\Hom_{\RR}(V, \RR)$ such that $\phi_{\lambda}(e_{\mu}) \geq 0$ for $\lambda \not= \mu$.
This pair $(\pmb{e}, \pmb{\phi})$ of $\pmb{e}$ and $\pmb{\phi}$ is called a {\em system of Zariski decompositions} in $V$.

Let us fix several notations which work only in this section.
For $\lambda \in \Lambda$, the coefficient of $x$ at $e_\lambda$ in
the linear combination of $x$ with respect to the basis $\pmb{e}$ is
denoted by $x(\lambda; \pmb{e})$, that is,
$x = \sum_{\lambda} x(\lambda; \pmb{e}) e_{\lambda}$.
Let $\leq_{\pmb{e}}$ be an order relation of $V$ given by
\[
x \leq_{\pmb{e}} y \quad \overset{\text{def}}{\Longleftrightarrow} \quad \text{$x(\lambda; \pmb{e}) \leq  y(\lambda; \pmb{e})$ for all $\lambda \in \Lambda$}.
\]
We often use $y \geq_{\pmb{e}} x$ instead of $x \leq_{\pmb{e}} y$.
$\Supp(x; \pmb{e})$,
$[x, y]_{\pmb{e}}$,
$(-\infty, x]_{\pmb{e}}$,
$[x, \infty)_{\pmb{e}}$, $\Nef(\pmb{\phi})$ and
$\Num(\pmb{\phi})$ are defined as follows:
\[
\begin{cases}
\Supp(x; \pmb{e}) := \{ \lambda \in \Lambda \mid x(\lambda; \pmb{e}) \not= 0 \}, \\
[x, y]_{\pmb{e}} := \{ v \in V \mid x \leq_{\pmb{e}} v \leq_{\pmb{e}} y \}, \\
(-\infty, x]_{\pmb{e}} := \{ v \in V \mid v \leq_{\pmb{e}} x \}, \\
[x, \infty)_{\pmb{e}} := \{ v \in V \mid v \geq_{\pmb{e}} x \}, \\
\Nef(\pmb{\phi}) := \left\{ v \in V  \mid \text{$\phi_\lambda(v) \geq 0$ for all $\lambda \in \Lambda$} \right\}, \\
\Num(\pmb{\phi}) := \{ v \in V \mid \text{$\phi_{\lambda}(v) = 0$ for all $\lambda \in \Lambda$} \}.
\end{cases}
\]

For an element $x$ of $V$, a decomposition $x = y + z$
is called a {\em Zariski decomposition} of $x$ with respect to $(\pmb{e}, \pmb{\phi})$ if
the following conditions are satisfied:
\begin{enumerate}
\renewcommand{\labelenumi}{(\arabic{enumi})}
\item
$y \in \Nef(\pmb{\phi})$ and $z \geq_{\pmb{e}} 0$.

\item
$\phi_{\lambda}(y) = 0$ for all $\lambda \in \Supp(z; \pmb{e})$.

\item
$\left\{ x \in \sum_{\lambda \in \Supp(z; \pmb{e})} \RR_{\geq 0} e_{\lambda} \mid \text{$\phi_{\lambda}(x) \geq 0$ for all $\lambda \in \Supp(z; \pmb{e})$} \right\} = \{ 0 \}$.
\end{enumerate}
We call $y$ (resp. $z$) the {\em positive part} of $x$ (resp. {\em negative part} of $x$).

The purpose of this section is to give the proof of the following proposition.

\begin{Proposition}
\label{prop:zariski:decomp:in:vector:space}
For an element $x$ of $V$, we have the following:
\begin{enumerate}
\renewcommand{\labelenumi}{(\arabic{enumi})}
\item The following are equivalent:
\begin{enumerate}
\renewcommand{\labelenumii}{(\arabic{enumi}.\arabic{enumii})}
\item
A Zariski decomposition of $x$ with respect to $(\pmb{e}, \pmb{\phi})$ exists.

\item
$(-\infty, x]_{\pmb{e}} \cap \Nef(\pmb{\phi}) \not= \emptyset$.
\end{enumerate}

\item
If a Zariski decomposition exists,
then it is uniquely determined.

\item
If  a  Zariski decomposition of $x$ with respect to $(\pmb{e}, \pmb{\phi})$ exists and the negative part $z$ of $x$ is non-zero,
then $z$ has the following properties:
\begin{enumerate}
\renewcommand{\labelenumii}{(\arabic{enumi}.\arabic{enumii})}
\item
Let $Q$ be the matrix given by $(\phi_{\lambda}(e_{\mu}))_{\lambda, \mu \in \Supp(z; \pmb{e})}$.
Then 
\[
(-1)^{\#(\Supp(z;\pmb{e}))} \det Q > 0.
\]
Moreover, if $Q$ is symmetric, then $Q$ is negative definite.

\item
$\{ e_{\lambda} \}_{\lambda \in \Supp(z; \pmb{e})}$ is linearly independent on $V/\Num(\pmb{\phi})$.
\end{enumerate}
\end{enumerate}
\end{Proposition}

\subsection{Proofs}
\setcounter{Theorem}{0}
Here let us give the proof of Proposition~\ref{prop:zariski:decomp:in:vector:space}.

For $x_1, \ldots, x_r \in V$, $\max_{\pmb{e}} \{ x_1, \ldots, x_r \} \in V$ is given by
\[
\max\nolimits_{\pmb{e}} \{ x_1, \ldots, x_r \} := \sum_{\lambda \in \Lambda} \max \{ x_1(\lambda;\pmb{e}), \ldots, x_r(\lambda;\pmb{e}) \} e_{\lambda}.
\]
Let us begin with the following lemma.

\begin{Lemma}
\label{lem:nef:max}
If $x_1, \ldots, x_r \in \Nef(\pmb{\phi})$, then
$\max \{ x_1, \ldots, x_r \} \in \Nef(\pmb{\phi})$.
\end{Lemma}

\begin{proof}
It is sufficient to see that if
$\phi_\lambda(x_i) \geq 0$ for all $i$,
then $\phi_\lambda(\max_{\pmb{e}} \{ x_1, \ldots, x_r  \}) \geq 0$. We set $z = \max_{\pmb{e}} \{ x_1, \ldots, x_r \}$.
Note that $\Supp(z - x_1; \pmb{e}) \cap \cdots \cap \Supp(z - x_r; \pmb{e}) = \emptyset$.
Thus there is $i$ with $\lambda \not\in \Supp(z - x_i; \pmb{e})$.
Then $\phi_{\lambda}(z - x_i) \geq 0$, and hence 
\[
\phi_{\lambda}(z) = \phi_{\lambda}(z - x_i) + \phi_{\lambda}(x_i) \geq 0.
\]
\end{proof}

\begin{Lemma}
\label{lem:greatest:element}
Let $x$ be an element of $V$ such that 
$(-\infty, x]_{\pmb{e}} \cap \Nef(\pmb{\phi}) \not= \emptyset$.
Then there is the greatest element $y$ in $(-\infty, x]_{\pmb{e}} \cap \Nef(\pmb{\phi})$, that is,
$y \in \Nef(\pmb{\phi}) \cap (-\infty,x]_{\pmb{e}}$ and
$y \geq_{\pmb{e}} v$ for all $v \in \Nef(\pmb{\phi}) \cap (-\infty,x]_{\pmb{e}}$.
This greatest element $y$ is denoted by 
\[
\max (\Nef(\pmb{\phi}) \cap (-\infty,x]_{\pmb{e}}).
\]
Further, $y$ and $z := x - y$ satisfy the following properties:
\begin{enumerate}
\renewcommand{\labelenumi}{(\alph{enumi})}
\item $y \in \Nef(\pmb{\phi})$, $z \geq_{\pmb{e}} 0$ and $x = y + z$.

\item $\phi_\lambda(y) = 0$ for all $\lambda \in \Supp(z;\pmb{e})$.

\item
$\left\{ v \in \sum\nolimits_{ \lambda \in \Supp(z; \pmb{e})} \RR_{\geq 0} e_\lambda \mid \text{$\phi_\lambda(v) \geq 0$ for all $\lambda \in \Supp(z; \pmb{e})$} \right\} = \{ 0 \}$.
\end{enumerate}
\end{Lemma}

\begin{proof}
We choose $x' \in (-\infty, x]_{\pmb{e}} \cap \Nef(\pmb{\phi})$.
Let us see the following claim.

\begin{Claim}
There is the greatest element $y$ of $\Nef(\pmb{\phi}) \cap [x',x]_{\pmb{e}}$.
\end{Claim}

\begin{proof}
Note that $[x', x]_{\pmb{e}} = x' + [0, x-x']_{\pmb{e}}$.
Moreover, 
it is easy to see that
\begin{multline*}
\Nef(\pmb{\phi}) \cap [x',x]_{\pmb{e}} \\
= x'
+ \left\{ v \in [0, x-x']_{\pmb{e}} \mid
\text{$\phi_\lambda(v) \geq -\phi_\lambda(x')$ for all $\lambda \in \Supp(x-x'; \pmb{e})$} \right\}.
\end{multline*}
Therefore, $\Nef(\pmb{\phi}) \cap [x',x]_{\pmb{e}}$ is a translation of a bounded convex polyhedral set in
a finite dimensional vector space
$\bigoplus\nolimits_{\lambda \in \Supp(x-x'; \pmb{e})} \RR e_\lambda$.
Hence $\Nef(\pmb{\phi}) \cap [x',x]_{\pmb{e}}$ is a convex polytope, that is,
there are $\gamma_1, \ldots, \gamma_l \in \Nef(\pmb{\phi}) \cap [x',x]_{\pmb{e}}$ such that
$\Nef(\pmb{\phi})\cap [x',x]_{\pmb{e}} = \operatorname{Conv}\{ \gamma_1, \ldots, \gamma_l \}$ 
(cf. \cite[Theorem~3.2.5 or Finite basis theorem]{Web}).
If we set $y = \max\{ \gamma_1, \ldots, \gamma_l \}$, then, by Lemma~\ref{lem:nef:max},
$y \in \Nef(\pmb{\phi})\cap [x',x]_{\pmb{e}}$. Moreover, 
for $v = a_1 \gamma_1 + \cdots + a_l \gamma_l \in \Nef(\pmb{\phi})\cap [x',x]_{\pmb{e}}$ ($a_1, \ldots, a_l \in \RR_{\geq 0}$ and
$a_1 + \cdots  + a_l = 1$), we have 
\[
y = a_1 y + \cdots + a_l y \geq_{\pmb{e}} a_1 \gamma_1 + \cdots + a_l \gamma_l = v.
\]
\end{proof}
This $y$ is actually the greatest element of $(-\infty, x]_{\pmb{e}} \cap \Nef(\pmb{\phi})$.
Indeed, if $v \in (-\infty, x]_{\pmb{e}} \cap \Nef(\pmb{\phi})$, then
$\max \{ v, y \} \in  [x', x]_{\pmb{e}} \cap \Nef(\pmb{\phi})$ by lemma~\ref{lem:nef:max}, and hence 
\[
v \leq \max \{ v, y \} \leq y.
\]

\medskip
Let us check the properties (a), (b) and (c).
First of all, (a) is obvious. In order to see (b) and (c),
we may assume that $z \not= 0$.

(b)  We assume that $\phi_\lambda(y) > 0$ for $\lambda \in \Supp(z; \pmb{e})$.
Let $\epsilon$ be a sufficiently small positive number. Then $y + \epsilon e_\lambda \leq_{\pmb{e}} x$ and
\[
\phi_\mu(y + \epsilon e_\lambda) = \phi_\mu(y) + \epsilon \phi_\mu(e_\lambda) \geq 0
\]
for all $\mu \in \Lambda$ because $0 < \epsilon \ll 1$. Thus $y + \epsilon e_\lambda \in \Nef(\pmb{\phi})$, which contradicts to the maximality of $y$.
Therefore, $\phi_\lambda(y) = 0$ for $\lambda \in \Supp(z; \pmb{e})$.

(c)  Next we assume that there is $v \in \left( \sum\nolimits_{ \lambda \in \Supp(z; \pmb{e})} \RR_{\geq 0} e_\lambda \right) \setminus \{ 0 \}$ such that
$\phi_\lambda(v) \geq 0$ for all $\lambda \in \Supp(z; \pmb{e})$. Then there is a sufficiently small positive number $\epsilon'$ such that
$y + \epsilon' v \leq_{\pmb{e}} x$. Note that $\phi_\mu(y + \epsilon' v) \geq 0$ for all $\mu$, which yields a contradiction, as before.
\end{proof}

\begin{Lemma}
\label{lem:negative:definite}
Let $W$ be a vector space over $\RR$.
Let $e_1, \ldots, e_n \in W$ and
$\phi_1, \ldots, \phi_n \in \Hom_{\RR}(W, \RR)$ with the following properties:
\begin{enumerate}
\renewcommand{\labelenumi}{(\alph{enumi})}
\item
$\{ (a_1,\ldots,a_n) \in \RR_{\geq 0}^n \mid a_1 e_1 + \cdots + a_n e_n = 0 \} = \{ (0,\ldots , 0) \}$.

\item $\phi_i(e_j) \geq 0$ for all $i \not=j $.

\item
$\{ x \in \RR_{\geq 0} e_1 + \cdots + \RR_{\geq 0} e_n \mid \text{$\phi_i(x) \geq 0$ for all $i$} \} = \{ 0 \}$.
\end{enumerate}
Then we have the following:
\begin{enumerate}
\renewcommand{\labelenumi}{(\arabic{enumi})}
\item
Let $Q$ be the $(n \times n)$-matrix given by $(\phi_i(e_j))$.
Then there are $(n \times n)$-matrices $A$ and $B$ with the following properties:
\begin{enumerate}
\renewcommand{\labelenumii}{(\arabic{enumi}.\arabic{enumii})}
\item
$A$ \rom{(}resp. $B$\rom{)} is a lower \rom{(}resp. upper\rom{)} triangle matrix consisting of non-negative numbers.

\item
$\det A > 0$, $\det B > 0$ and
\[
A Q B = \begin{pmatrix}-1 & \cdots & 0 \\ \vdots &  \ddots & \vdots \\ 0 & \cdots & -1 \end{pmatrix}.
\]

\item If $Q$ is symmetric, then $B = {}^t A$.
\end{enumerate}

\item
The vectors
$e_1, \ldots, e_n$ are linearly independent in 
\[
W/\{ x \in W \mid \phi_1(x) = \cdots = \phi_n(x) = 0 \}.
\]
\end{enumerate}
\end{Lemma}

\begin{proof}
(1) Let us begin with the following claim.

\begin{Claim}
\label{claim:prop:negative:definite:0}
$\phi_i(e_i) < 0$ for all $i $.
\end{Claim}

\begin{proof}
If $\phi_i(e_i) \geq 0$, then $e_i \in \{ x \in \RR_{\geq 0} e_1 + \cdots + \RR_{\geq 0} e_n \mid \text{$\phi_j(x) \geq 0$ for all $j$} \}$.
This is a contradiction because $e_i \not= 0$.
\end{proof}

The above claim proves (1) in the case where $n=1$. Here
we set 
\[
\phi'_i = -\phi_1(e_1) \phi_i + \phi_i(e_1) \phi_1\ (i \geq 2),
\quad
e'_j = -\phi_1(e_1) e_j + \phi_1(e_j) e_1\ (j \geq 2).
\]
We claim the following:

\begin{Claim}
\label{claim:prop:negative:definite:2}
\begin{enumerate}
\renewcommand{\labelenumi}{(\roman{enumi})}
\item
$\phi_i(e_1) =  0$ and $\phi_1(e'_j) = 0$ for all $i \geq 2$ and $j \geq 2$.

\item
$e'_2,\ldots, e'_n$ and $\phi'_2, \ldots, \phi'_n$ satisfy 
all assumptions \rom{(a)} $\sim$ \rom{(c)} of the lemma.

\item 
Let $Q'$ be the matrix given by $(\phi'_i(e'_j))_{2 \leq i, j \leq n}$. Then
\[
 A_1 Q B_1 = \begin{pmatrix} \phi_1(e_1) & 0 \\ 0 & Q' \end{pmatrix},
\]
where 
$A_1$ and $B_1$ are matrices given by
{\tiny
\[
\hspace{5em}
\begin{pmatrix}
1 & 0 & 0 & \cdots & 0 \\
\phi_2(e_1) & -\phi_1(e_1) & 0 & \cdots & 0 \\
\phi_3(e_1) & 0 & -\phi_1(e_1) & \cdots & 0 \\
\vdots & \vdots & \vdots & \ddots & \vdots \\
\phi_n(e_1) & 0 & 0 & \cdots & -\phi_1(e_1)
\end{pmatrix},
\begin{pmatrix}
1 & \phi_1(e_2) & \phi_1(e_3) & \cdots & \phi_1(e_n) \\
0 & -\phi_1(e_1) & 0 & \cdots & 0 \\
0 & 0 & -\phi_1(e_1) & \cdots & 0 \\
\vdots & \vdots & \vdots & \ddots & \vdots \\
0 & 0 & 0 & \cdots & -\phi_1(e_1)
\end{pmatrix}
\]}
respectively. Note that if $Q$ is symmetric,
then $B_1 = {}^t A_1$ and $Q'$ is also symmetric.
\end{enumerate}
\end{Claim}

\begin{proof}
(i) is obvious.

(ii) It is easy to see (a) for $e'_2, \ldots, e'_n$ by using Claim~\ref{claim:prop:negative:definite:0}.
For $i, j \geq 2$ with $i \not= j$, by Claim~\ref{claim:prop:negative:definite:0},
\[
\phi'_i(e'_j) = \phi_1(e_1)^2\phi_i(e_j) + (-\phi_1(e_1))\phi_i(e_1)\phi_1(e_j) \geq 0.
\]
Finally let $x \in \sum_{j\geq 2} \RR_{\geq 0} e'_j$
such that $\phi'_i(x) \geq 0$ for all $i \geq 2$.
Note that $\phi'_i(x) = (-\phi_1(e_1))\phi_i(x)$ for $i \geq 2$.
Therefore, $\phi_i(x) \geq 0$ for all $i \geq 1$, and hence 
$x = 0$ because $\sum_{j\geq 2} \RR_{\geq 0} e'_j  \subseteq \sum_{j \geq 1} \RR_{\geq 0} e_j$.

(iii) is a straightforward calculation.
\end{proof}

We prove (1) by induction on $n$. 
By hypothesis of induction, there are matrices $A'$ and $B'$ satisfying (1.1), (1.2) and (1.3)
for $Q'$, that is,
\[
A' Q' B' = \begin{pmatrix}-1 & \cdots & 0 \\ \vdots &  \ddots & \vdots \\ 0 & \cdots & -1 \end{pmatrix}.
\]
Therefore,
\[
\begin{pmatrix} \frac{1}{\sqrt{-\phi_1(e_1)}} & 0 \\ 0 & A' \end{pmatrix}
A_1 Q B_1 \begin{pmatrix} \frac{1}{\sqrt{-\phi_1(e_1)}} & 0 \\ 0 & B' \end{pmatrix} =
\begin{pmatrix}-1 & \cdots & 0 \\ \vdots &  \ddots & \vdots \\ 0 & \cdots & -1 \end{pmatrix}.
\]
Thus (1) follows.

\bigskip
(2) Let $a_1 e_1 + \cdots + a_n e_n = 0$ be a linear relation on 
\[
W/\{ x \in W \mid \phi_1(x) = \cdots = \phi_n(x) = 0 \}.
\]
Then there is $x \in W$ such that
$x = a_1 e_1 + \cdots + a_n e_n$ and $\phi_1(x) = \cdots = \phi_n(x) = 0$.
Thus $0 = \phi_i(x) = \sum \phi_i(e_j) a_j$. Hence (1) yields (2).
\end{proof}

\begin{proof}[Proof of Proposition~\ref{prop:zariski:decomp:in:vector:space}]
(1) Clearly (1.1) implies (1.2). If we assume (1.2), then
(1.1) follows from Lemma~\ref{lem:greatest:element}.

(2) Let $x = y+ z$ be a Zariski decomposition of $x$ with respect to $(\pmb{e}, \pmb{\phi})$ and
$y' = \max (\Nef_{\pmb{\phi}} \cap (-\infty, x]_{\pmb{e}})$.
Then $y \leq_{\pmb{e}} y'$.
As $\phi_{\lambda}(y) = 0$ for all $\lambda \in \Supp(z; \pmb{e})$,
\[
y' - y \in \left\{ x \in \sum\nolimits_{\lambda \in \Supp(z; \pmb{e})} \RR_{\geq 0} e_{\lambda} \mid \text{$\phi_{\lambda}(x) \geq 0$ for all $\lambda \in \Supp(z; \pmb{e})$} \right\},
\]
and
hence $y' = y$.

(3) follows from Lemma~\ref{lem:negative:definite}.
\end{proof}

\begin{Remark}
We assume that $\phi_{\lambda}(e_{\mu}) \in \QQ$ for all $\lambda, \mu \in \Lambda$.
Let $x \in \bigoplus\nolimits_{\lambda} \QQ e_{\lambda}$ such that  $(-\infty, x]_{\pmb{e}} \cap \Nef_{\pmb{\phi}} \not= \emptyset$.
Let $x = y+z$ be the Zariski decomposition of $x$ with respect to $(\pmb{e}, \pmb{\phi})$.
Then $y, z \in \bigoplus\nolimits_{\lambda} \QQ e_{\lambda}$.
Indeed, if we set
$\Supp(z; \pmb{e}) = \{ \lambda_1, \ldots, \lambda_n \}$ and
$z = \sum a_i e_{\lambda_i}$, then
\[
\sum \phi_{\lambda_i}(e_{\lambda_j}) a_j = \phi_{\lambda_i}(x) \in \QQ.
\]
On the other hand, by our assumption and (3.1) in Proposition~\ref{prop:zariski:decomp:in:vector:space},
$(\phi_{\lambda_i}(e_{\lambda_j}))_{1 \leq i, j \leq n} \in \operatorname{GL}_n(\QQ)$.
Thus $(a_1, \ldots, a_n) \in \QQ^n$.
\end{Remark}

\section{Green functions for $\RR$-divisors}
\label{sec:Green:function:R:divisor}

\subsection{Plurisubharmonic functions}
\label{subsec:pluri:subharmonic}
Here we recall plurisubharmonic functions and the upper semicontinuous regularization of a function
locally bounded above.

Let $T$ be a metric space with a metric $d$.
A function $f : T \to \{ -\infty\} \cup \RR$ is said to be {\em upper semicontinuous} if
$\{ x \in T \mid f(x) < c \}$ is open for any $c \in \RR$. In other words,
\[
f(a) = \limsup_{x\to a} f(x) \left(:= \inf_{\epsilon > 0} (\sup \{ f(y) \mid d(a, y) \leq \epsilon \} )\right)
\]
for all $a \in T$.
Let $u : T \to \{ -\infty\} \cup \RR$ be a function such that
$u$ is locally bounded above.
The {\em upper semicontinuous regularization} $u^*$ of $u$ is given by
\[
u^*(x) = \limsup_{y \to x} u(y). 
\]
Note that $u^*$ is upper semicontinuous and $u \leq u^*$.

Let $D$ be an open set in $\CC$.
A function $u : D \to \{ -\infty \} \cup \RR$ is said to be {\em subharmonic} if
$u$ is upper semicontinuous and
\[
u(a) \leq \frac{1}{2\pi} \int_{0}^{2\pi} u\left(a + r e^{\sqrt{-1}\theta}\right) d\theta
\]
holds for any $a \in D$ and $r \in \RR_{>0}$ with
$\{ z \in \CC \mid \vert z - a \vert \leq r \} \subseteq D$.

Let $X$ be a $d$-equidimensional complex manifold.
A function $u : X \to \{ - \infty\} \cup \RR$ is said to be {\em plurisubharmonic} if
$u$ is upper semicontinuous and $u \circ \phi$ is subharmonic for
any analytic map $\phi : \{ z \in \CC \mid \vert z \vert < 1 \} \to X$.
We say $u$ is a  {\em real valued plurisubharmonic function} if
$u(x) \not= -\infty$ for all $x \in X$.
If $X$ is an open set of $\CC^d$, then an upper semicontinuous function $u : X \to \RR \cup \{ -\infty \}$
is plurisubharmonic if and only if
\[
u(a) \leq \frac{1}{2\pi} \int_{0}^{2\pi} u(a + \xi \exp(\sqrt{-1} \theta)) d\theta
\]
holds
for any $a \in X$ and $\xi \in \CC^d$ with
$\{ a + \xi \exp(\sqrt{-1}\theta) \mid 0 \leq \theta \leq 2\pi \} \subseteq X$.
As an example of plurisubharmonic functions, we have the following:
if $f_1, \ldots, f_r$ are holomorphic functions on $X$,
then 
\[
\log (\vert f_1 \vert^2 + \cdots + \vert f_r \vert^2)
\]
is a plurisubharmonic function on $X$.
In particular, if 
\[
x \not\in \{ z \in X \mid f_1(z) = \cdots = f_r(z) = 0 \},
\]
then $dd^c(\log (\vert f_1 \vert^2 + \cdots + \vert f_r \vert^2))$ is semipositive around $x$.

Let $\{ u_{\lambda} \}_{\lambda \in \Lambda}$ be a family of plurisubharmonic functions on $X$ such that
$\{ u_{\lambda} \}_{\lambda \in \Lambda}$ is locally uniformly bounded above.
If we set $u(x) := \sup_{\lambda \in \Lambda} \{ u_{\lambda}(x) \}$ for $x \in X$, then
the upper semicontinuous regularization $u^*$ of $u$ is plurisubharmonic and
$u = u^*\ \aew$ (cf. \cite[Theorem~2.9.14 and Proposition~2.6.2]{MK}).
Moreover, let $\{ v_{n} \}_{n=1}^{\infty}$ be a sequence of plurisubharmonic functions on $X$ such that
$\{ v_{n} \}_{n=1}^{\infty}$ is  locally uniformly bounded above.
If we set $v(x) := \limsup_{n\to\infty} v_n(x)$ for $x \in X$, then
the upper semicontinuous regularization $v^*$ of $v$ is plurisubharmonic and
$v = v^*\ \aew$ (cf. \cite[Proposition~2.9.17 and Theorem~2.6.3]{MK}).

\subsection{$\RR$-divisors}
\setcounter{Theorem}{0}
\label{subsec:R:divisor}
Let $X$ be either a $d$-equidimensional smooth algebraic scheme over $\CC$, or a $d$-equi\-dimensional complex manifold.
Let $\Div(X)$ be the group of divisors on $X$.
An element $D$ of $\Div(X)_{\RR} :=  \Div(X) \otimes_{\ZZ} \RR$ is called an {\em $\RR$-divisor} on $X$.
Let $D = \sum_{i=1}^n a_i D_i$ be the irreducible decomposition of $D$, that is,
$a_1, \ldots, a_n \in \RR$ and $D_i$'s are reduced and irreducible divisors on $X$.
For a prime divisor  $\Gamma$ on $X$ (i.e., a reduced and irreducible divisor on $X$),
the coefficient of $D$ at $\Gamma$ in the above irreducible decomposition is denoted by $\mult_{\Gamma}(D)$, that is,
\[
\mult_{\Gamma}(D) = \begin{cases}
a_i & \text{if $\Gamma = D_i$ for some $i$}, \\
0 & \text{if $\Gamma \not= D_i$ for all $i$},
\end{cases}
\]
and $D = \sum_{\Gamma}\mult_{\Gamma}(D) \Gamma$.
The support $\Supp(D)$ of $D$ is defined by $\bigcup_{\mult_{\Gamma}(D) \not= 0} \Gamma$.
If $a_i \geq 0$ for all $i$, then $D$ is said to be {\em effective} and it is denoted by $D\geq 0$.
More generally, for $D_1, D_2 \in \Div(X)_{\RR}$,
\[
\text{$D_1 \leq D_2$ (or $D_2 \geq D_2$)}\quad\overset{\rm def}{\Longleftrightarrow}\quad
D_2 - D_1 \geq 0.
\]
The {\em round-up} $\lceil D \rceil$ of $D$ and the {\em round-down} $\lfloor D \rfloor$ of $D$
are defined by
\[
\lceil D \rceil = \sum_{i=1}^n \lceil a_i \rceil D_i\quad\text{and}\quad
\lfloor D \rfloor = \sum_{i=1}^n \lfloor a_i \rfloor D_i,
\]
where $\lceil x \rceil = \min \{ a \in \ZZ \mid x \leq a \}$ and $\lfloor x \rfloor = \max \{ a \in \ZZ \mid a \leq x \}$ 
for $x \in \RR$.

We assume that $X$ is algebraic.
Let $\Rat(X)$ be the ring of rational functions on $X$.
Note that $X$ is not necessarily connected, so that
$\Rat(X)$ is not necessarily a field.
In the case where $X$ is connected,
$H^0(X, D)$ is defined to be
\[
H^0(X, D) := \{ \phi \in \Rat(X)^{\times} \mid (\phi) + D \geq 0 \} \cup \{ 0 \}.
\]
In general, let $X = \coprod_{\alpha} X_{\alpha}$ be the decomposition into
connected components, and let $D_{\alpha} = \rest{D}{X_{\alpha}}$.
Then $H^0(X, D)$ is defined to be
\[
H^0(X, D) := \bigoplus_{\alpha} H^0(X_{\alpha}, D_{\alpha}).
\]
Note that if $D$ is effective, then
$H^0(X, D)$ is generated by
\[
\{ \phi \in \Rat(X)^{\times} \mid (\phi) + D \geq 0 \}.
\]
Indeed, for $\phi_{\alpha} \in H^0(X_{\alpha}, D_{\alpha})$,
if we choose $c \in \CC^{\times}$ with $\phi_{\alpha} + c \not= 0$, then
\[
(0,\ldots, 0, \phi_{\alpha},0, \ldots,0) \\
= (1,\ldots, 1, \phi_{\alpha}+c,1,\ldots,1) - (1,\ldots,1, c,1,\ldots,1),
\]
which shows the assertion.
Since
\[
(\phi_{\alpha}) + D_{\alpha} \geq 0 \quad\Longleftrightarrow\quad (\phi_{\alpha}) + \lfloor D_{\alpha} \rfloor \geq 0,
\]
we have
$H^0(X, D) = H^0(X, \lfloor D\rfloor)$.

In the case where $X$ is not necessarily algebraic,
the ring of meromorphic functions on $X$ is denoted by $\mathcal{M}(X)$.
By using $\mathcal{M}(X)$ instead of $\Rat(X)$, we can define $H^0_{\mathcal{M}}(X, D)$ in the same way as above,
that is, if $X$ is connected, then
\[
H_{\mathcal{M}}^0(X, D) := \{ \phi \in \mathcal{M}(X)^{\times} \mid (\phi) + D \geq 0 \} \cup \{ 0 \}.
\]
If $X$ is a proper smooth algebraic scheme over $\CC$,
then $\Rat(X) = \mathcal{M}(X)$ by GAGA, and hence $H^0(X, D) = H_{\mathcal{M}}^0(X, D)$.

\subsection{Definition of Green functions for $\RR$-divisors}
\setcounter{Theorem}{0}
\label{subsec:Green:function:P}
Let $X$ be a $d$-equidimensional complex manifold.
Let $\mathcal{L}^1_{\rm loc}$ 
be the sheaf consisting of locally integrable functions,
that is,
\[
\mathcal{L}^1_{\rm loc}(U) := \{ g : U \to \RR  \cup \{ \pm\infty \} \mid \text{$g$ is locally integrable} \} \\
\]
for an open set $U$ of $X$.
Let $\TT$ be a subsheaf  of $\mathcal{L}^1_{\rm loc}$ and
let $S$ be a subset of $\RR \cup \{ \pm \infty \}$. 
Then $\TT_S$, $\TT^b$ and $-\TT$ are defined as follows: 
\begin{align*}
\TT_S(U) &:= \{ g \in \TT(U) \mid \text{$g(x) \in S$ for all $x \in U$}\}, \\
\TT^b(U) &:= \{ g \in \TT(U) \mid \text{$g$ is locally bounded on $U$}\}, \\
-\TT(U) &:=  \{ -g  \in \mathcal{L}^1_{\rm loc}(U) \mid g \in \TT(U)\}.
\end{align*}
Let $\TT'$ be another subsheaf of $\mathcal{L}^1_{\rm loc}$.
We assume that $u + u'$ is well-defined as functions for any open set $U$,
$u \in \TT(U)$ and $u' \in \TT'(U)$.
Then $\TT + \TT'$ is  defined to be
\[
(\TT + \TT')(U) := \left\{
g \in \mathcal{L}^1_{\rm loc}(U)  \left|   \hspace{-0.2em}
\begin{array}{l} 
\text{For any $x \in U$, we can find an open } \\
\text{neighborhood $V_x$, $u \in \TT(V_x)$ and} \\
\text{$u' \in \TT'(V_x)$ such that $\rest{g}{V_x} = u + u'$.}
\end{array}\hspace{-0.2em} \right\}\right..
\]
Similarly, if $u - u'$ is well-defined as functions for any open set $U$,
$u \in \TT(U)$ and $u' \in \TT'(U)$, then $\TT - \TT'$ is  defined to be
\[
(\TT - \TT')(U)  := \left\{
g \in \mathcal{L}^1_{\rm loc}(U)  \left|   \hspace{-0.2em}
\begin{array}{l} 
\text{For any $x \in U$, we can find an open } \\
\text{neighborhood $V_x$, $u \in \TT(V_x)$ and} \\
\text{$u' \in \TT'(V_x)$ such that $\rest{g}{V_x} = u - u'$.}
\end{array}\hspace{-0.2em} \right\}\right..
\]
Note that $\TT - \TT' = \TT + (-\TT')$.
A subsheaf $\TT$ of $\mathcal{L}^1_{\rm loc}$
is called a {\em type for Green functions} on $X$ if the following conditions are satisfied
(in the following (1), (2) and (3), $U$ is an arbitrary open set of $X$):
\begin{enumerate}
\renewcommand{\labelenumi}{(\arabic{enumi})}
\item
If $u, v \in \TT(U)$ and $a \in \RR_{\geq 0}$,
then $u + v \in \TT(U)$ and $a u \in \TT(U)$.

\item
If $u, v \in \TT(U)$ and $u \leq v\ \aew$, then $u \leq v$.

\item
If $\phi \in \OO_X^{\times}(U)$
(i.e., $\phi$ is a  nowhere vanishing holomorphic function on $U$),
then 
$\log \vert \phi  \vert^2 \in \TT(U)$.
\end{enumerate}
Note that, for $u, v \in \TT(U)$,  $u = v$ if $u = v\ \aew$.
If $\TT = \TT_{\RR}$, that is,
$u(x) \in \RR$ for any open set $U$, $u \in \TT(U)$ and $x \in U$,
then $\TT$ is called a {\em real valued type}.
As examples of  types for Green functions on $X$,
we have the following 
$C^0$, $C^{\infty}$ and $\Tpsh$:  
\begin{list}{}{\leftmargin=5em\labelwidth=3.7em}
\renewcommand{\makelabel}{}
\item[$C^{0}$ :]   the sheaf consisting of continuous functions. 

\item[$C^{\infty}$ :]  the sheaf consisting of $C^{\infty}$-functions. 

\item[$\Tpsh$ :]   the sheaf consisting of  plurisubharmonic functions. 
\end{list}
Note that
\[
\Tpsh_{\RR}(U) = \{ g \in \Tpsh(U) \mid \text{$g(x) \not= -\infty$ for all $x \in U$} \}
\]
for an open set $U$ of $X$.
Let $\TT$ and $\TT'$ be types for Green functions on $X$.
We say $\TT'$ is a {\em subjacent type} of $\TT$ if the following property holds for
any open set $U$ of $X$: 
\[
\text{$u' \leq u\ \aew$ on $U$ for $u' \in \TT'(U)$ and $u \in \TT(U)$ $\quad\Longrightarrow\quad$
$u' \leq u$ on $U$.}
\]

\begin{Lemma}
\label{lem:fqpssh:ineq:ae}
Let $\TT$ be either $C^0 + \Tpsh$ or $C^0 + \Tpsh_{\RR} -\Tpsh_{\RR}$.
Then $\TT$ is a type for Green functions on $X$.
Moreover, $\Tpsh$ is a subjacent type of $\TT$.
\end{Lemma}

\begin{proof}
The conditions (1) and (3) are obvious for $\TT$.
Let us see (2).
For $z = (z_1, \ldots, z_d) \in \CC^d$, we set $\Vert z \Vert = \sqrt{\vert z_1 \vert^2 + \cdots + \vert z_d \vert^2}$.
Moreover, for $a \in \CC^d$ and $r > 0$,  
\[
\{ z \in \CC^d \mid \Vert z - a \Vert < r \}
\]
is denoted by $B^d(a; r)$.

The assertion of (2) is local, so that we may assume that $X = B^d((0,\ldots,0); 1)$.
It is sufficient to see that, for $u_1, u_2 \in \TT(X)$,
if $u_1 \leq u_2\ \aew$, then $u_1 \leq u_2$.
Let us fix $a \in B^d((0,\ldots,0); 1)$. There are a sufficiently small $r > 0$ and 
$v_{ij} \in \mathcal{L}^1_{\rm loc}(B^d(a;r))$ ($i=1,2$ and $j=1,2,3$) with the following properties:
\begin{list}{}{\leftmargin=4em\labelwidth=1.5em}
\renewcommand{\makelabel}{}
\item[(a)]
$u_1 = v_{11} + v_{12} - v_{13}$ and $u_2 = v_{21} + v_{22} - v_{23}$.

\item[(b)]
$v_{11}, v_{21}   \in C^0(B^d(a;r))$.

\item[(c)]
$v_{12}, v_{22} \in \Tpsh(B^d(a;r))$ in the case $\TT = C^0 + \Tpsh$.

\item[(c)']
$v_{12}, v_{22} \in \Tpsh_{\RR}(B^d(a;r))$ in the  case $\TT = C^0 + \Tpsh_{\RR} -\Tpsh_{\RR}$.

\item[(d)]
$v_{13} = v_{23} = 0$ in the case $\TT = C^0+ \Tpsh$.

\item[(d)']
$v_{13}, v_{23} \in \Tpsh_{\RR}(B^d(a;r))$ in the  case $\TT = C^0 + \Tpsh_{\RR} -\Tpsh_{\RR}$.
\end{list}
Let $\chi_{\epsilon}$ ($\epsilon > 0$) be the standard smoothing kernels on $\CC^d$
(cf. \cite[Section~2.5]{MK}).
It is well known that $v_{ij}(a) = \lim_{\epsilon \to 0} (v_{ij} \ast \chi_{\epsilon})(a)$ for $i=1,2$ and $j=1,2,3$
(cf. \cite[Proposition~2.5.2 and Theorem~2.9.2]{MK}).
In the case $\TT = C^0 + \Tpsh$,
since $v_{11}(a), v_{21}(a) \in \RR$, $v_{12}(a), v_{22}(a) \in \RR \cup \{-\infty\}$ and $v_{13} = v_{23} = 0$,
\begin{align*}
\lim_{\epsilon \to 0} (u_i * \chi_{\epsilon})(a) & = \lim_{\epsilon\to 0} \left(  (v_{i1} * \chi_{\epsilon})(a) + (v_{i2} * \chi_{\epsilon})(a) - (v_{i3} * \chi_{\epsilon})(a)\right) \\
& = \lim_{\epsilon\to 0} (v_{i1} * \chi_{\epsilon})(a) + \lim_{\epsilon\to 0} (v_{i2} * \chi_{\epsilon})(a) - \lim_{\epsilon\to 0} (v_{i3} * \chi_{\epsilon})(a) \\
& = v_{i1}(a) + v_{i2}(a) - v_{i3}(a) = u_i(a).
\end{align*}
If $\TT = C^0 + \Tpsh_{\RR} -\Tpsh_{\RR}$, then, in the same way as above, we can also see
$u_i(a) = \lim_{\epsilon \to 0} (u_i * \chi_{\epsilon})(a)$ for $i=1,2$
because $v_{ij}(a) \in \RR$ for $i=1,2$ and $j=1,2,3$.
Therefore, (2) follows from inequalities $(u_1 \ast \chi_{\epsilon})(a) \leq (u_2 \ast \chi_{\epsilon})(a)$ ($\forall \epsilon > 0$).
The last assertion can be checked similarly.
\end{proof}

Let $\TT$ be a type for Green functions on $X$.
Let $g$ be a locally integrable function on $X$ and let $D = \sum_{i=1}^l a_i D_i$ be an $\RR$-divisor on $X$,
where $D_i$'s are reduced and irreducible divisors on $X$.
We say $g$ is a {\em $D$-Green function of $\TT$-type} (or a {\em Green function of $\TT$-type for $D$}) if,
for each point $x \in X$,
$g$ has a local expression 
\[
g = u + \sum_{i=1}^l (-a_i) \log \vert f_i \vert^2\quad\aew
\]
over an open neighborhood $U_x$ of $x$
such that $u \in \TT(U_x)$, 
where $f_1, \ldots, f_l$ are local equations of $D_1, \ldots, D_l$ on $U_x$ respectively.
Note that this definition does not depend on the choice of local equations
$f_1, \ldots, f_l$ on $U_x$  by the properties (1) and (3) of $\TT$.
The set of all $D$-Green functions of $\TT$-type is denoted by $G_{\TT}(X;D)$.

Let $g$ be a $D$-Green function of $\TT$-type.
We say $g$ is {\em of upper bounded type} (resp. {\em of lower bounded type})
if, in the above local expression $g = u + \sum_{i=1}^l (-a_i) \log \vert f_i \vert^2\ \aew$ around each point of $X$, 
$u$ is locally bounded above (resp. locally bounded below).
If $g$ is  of upper and lower bounded type,
then $g$ is said to be {\em of bounded type}.
These definitions also do not depend on the choice of local equations.
Note that the set of all $D$-Green functions of $\TT$-type and of bounded type is nothing more than $G_{\TT^b}(X; D)$.

We assume $x \not\in \Supp(D)$. Let $g$ be a $D$-Green function of $\TT$-type.
Let $f_1, \ldots, f_l$ and $f'_1, \ldots, f'_l$ be two sets of local equations of $D_1, \ldots, D_l$ on an open neighborhood $U_x$ of $x$.
Let
\[
g = u + \sum_{i=1}^l (-a_i) \log \vert f_i \vert^2\ \aew\quad\text{and}\quad
g = u' + \sum_{i=1}^l (-a_i) \log \vert f'_i \vert^2\ \aew
\]
be two local expressions of $g$ over $U_x$, where $u, u' \in \TT(U_x)$.
Since $x \not\in \Supp(D)$, there is an open neighborhood $V_x$ of $x$ such that
$V_x \subseteq U_x$ and $f_1, \ldots, f_l, f'_1, \ldots, f'_l \in \OO^{\times}_X(V_x)$.
Thus, by the properties (1) and (3) of $\TT$,
\[
u + \sum_{i=1}^l (-a_i) \log \vert f_i \vert^2, \quad u' + \sum_{i=1}^l (-a_i) \log \vert f'_i \vert^2 \in \TT(V_x),
\]
and hence
\[
u + \sum_{i=1}^l (-a_i) \log \vert f_i \vert^2 = u' + \sum_{i=1}^l (-a_i) \log \vert f'_i \vert^2 \in \TT(V_x)
\]
over $V_x$ by the second property of $\TT$.
This observation shows that
\[
u(x) + \sum_{i=1}^l (-a_i) \log \vert f_i(x) \vert^2
\]
does not depend on the choice of the local expression of $g$.
In this sense, the value 
\[
u(x) + \sum_{i=1}^l (-a_i) \log \vert f_i(x) \vert^2
\]
is called the {\em canonical value of $g$ at $x$} and
it is denoted by $g_{\rm can}(x)$. Note that $g_{\rm can} \in \TT(X \setminus \Supp(D))$ and
$g = g_{\rm can}\ \aew$ on $X \setminus \Supp(D)$.
Moreover, if $\TT$ is real valued, then $g_{\rm can}(x) \in \RR$.
It is easy to see the following propositions.

\begin{Proposition}
\label{prop:lelong:formula}
Let $g$ be a $D$-Green function of $C^{\infty}$-type.
Then the current $dd^c([g]) + \delta_D$ is represented by a unique $C^{\infty}$-form $\alpha$,
that is,  $dd^c([g]) + \delta_D = [\alpha]$.
We often identifies the current $dd^c([g]) + \delta_D$ with $\alpha$, and denote it by $c_1(D,g)$.
\query{$\delta_A$ $\Longrightarrow$ $\delta_D$. add ``, and denote it by $c_1(D,g)$''
(24/Sep/2010)}
\end{Proposition}

\begin{Proposition}
\label{prop:intersec:two:type}
Let $\TT'$ and $\TT''$ be two types for Green functions on $X$ such that
$\TT', \TT'' \subseteq \TT$.
Then $G_{\TT' \cap \TT''}(X;D) = G_{\TT'}(X;D) \cap G_{\TT''}(X;D)$.
\end{Proposition}

\begin{Proposition}
\label{prop:scalar:sum:Green:function}
\begin{enumerate}
\renewcommand{\labelenumi}{(\arabic{enumi})}
\item
If $g$ is a $D$-Green function of $\TT$-type and $a \in \RR_{\geq 0}$,
then $a g$ is an $(aD)$-Green function of $\TT$-type.
Moreover, if $x \not\in \Supp(D)$, then $(ag)_{\rm can}(x) = ag_{\rm can}(x)$.

\item
If $g_1$ \rom{(}resp. $g_2$\rom{)} is a $D_1$-Green function of $\TT$-type \rom{(}resp. $D_2$-Green function of $\TT$-type\rom{)},
then $g_1 + g_2$ is a $(D_1 + D_2)$-Green function of $\TT$-type.
Moreover, if $x \not\in \Supp(D_1) \cup \Supp(D_2)$,
then $(g_1 + g_2)_{\rm can}(x) = (g_1)_{\rm can}(x) + (g_2)_{\rm can}(x)$.

\item
We assume that $-\TT \subseteq \TT$.
If $g$ is a $D$-Green function of $\TT$-type,
then $-g$ is a $(-D)$-Green function of $\TT$-type.
Moreover, if $x \not\in \Supp(D)$, then $(-g)_{\rm can}(x) = -g_{\rm can}(x)$.

\item Let $g$ be a $D$-Green function of $\TT$-type.
If  $g \geq 0\ \aew$ and $x \not\in \Supp(D)$,
then $g_{\rm can}(x) \geq 0$.
\end{enumerate}
\end{Proposition}

Finally let us consider the following three propositions.
\query{two\quad
$\Longrightarrow$\quad
three
(12/July/2010)}

\begin{Proposition}
Let $D = b_1 E_1 + \cdots + b_r E_r$  be an $\RR$-divisor on $X$ such that
$b_1, \ldots, b_r \in \RR$ and $E_i$'s are Cartier divisors on $X$.
Let $g$ be a $D$-Green function of $\TT$-type on $X$.
Let $U$ be an open set of $X$ and let
$\phi_1, \ldots, \phi_r$ be local equations of $E_1, \ldots, E_r$ over $U$ respectively.
Then there is a unique expression
\[
g = u + \sum_{i=1}^r (-b_i) \log \vert \phi_i \vert^2\quad\aew\qquad(u \in \TT(U))
\]
on $U$ modulo null functions. 
This expression is called {\em the local expression of $g$ over $U$ with respect to $\phi_1, \ldots, \phi_r$}.
\end{Proposition}

\begin{proof}
Let us choose reduced and irreducible divisors $D_1, \ldots, D_l$ and $\alpha_{ij} \in \ZZ$ such that
$E_i = \sum_{j=1}^l \alpha_{ij} D_j$ for each $i$.
If we set $a_j = \sum_{i=1}^r b_i \alpha_{ij}$, then $D = \sum_{j=1}^l a_j D_j$.
For each point $x \in U$, there are an open neighborhood $U_x$ of $x$,
local equations $f_{1,x}, \ldots, f_{l,x}$ of $D_1, \ldots, D_l$ on $U_x$ and
$u_x \in \TT(U_x)$ such that
$U_x \subseteq U$ and
\[
g = u_x + \sum_{j=1}^l (-a_j) \log \vert f_{j,x} \vert^2\quad \aew
\]
on $U_x$. 
Note that 
\[
g = u_x + \sum_{i=1}^r (-b_i) \log \left| \prod_{j=1}^l f_{j,x}^{\alpha_{ij} }\right|^2\quad \aew
\]
and $\prod_{j=1}^l f_{j,x}^{\alpha_{ij}}$ is a local equation of $E_i$ over $U_x$, so that
we can find nowhere vanishing holomorphic functions
$e_{1,x}, \ldots, e_{r,x}$ on $U_x$ such that
$\prod_{j=1}^l f_{j,x}^{\alpha_{ij}} = e_{i,x} \phi_i$ on $U_x$ for all $i =1, \ldots, r$.
Then
\[
g = u_x + \sum_{i=1}^r (-b_i) \log \vert e_{i,x} \vert^2 + \sum_{i=1}^r (-b_i) \log \vert \phi_i \vert^2 \quad\aew
\]
on $U_x$. Thus, for $x, x' \in U$,
\[
u_x + \sum_{i=1}^r (-b_i) \log \vert e_{i,x} \vert^2 = u_{x'} + \sum_{i=1}^r (-b_i) \log \vert e_{i,x'} \vert^2\quad\aew
\]
on $U_x \cap U_{x'}$, and hence
\[
u_x + \sum_{i=1}^r (-b_i) \log \vert e_{i,x} \vert^2 = u_{x'} + \sum_{i=1}^r(-b_i) \log \vert e_{i,x'} \vert^2
\]
on $U_x \cap U_{x'}$. This means that there is $u \in \TT(U)$ such that
$u$ is locally given by $u_x + \sum_{i=1}^r (-b_i) \log \vert e_{i,x} \vert^2$.
Therefore, $g = u + \sum_{i=1}^r (-b_i) \log \vert \phi_i \vert^2\ \aew$ on $U$.
The uniqueness of the expression modulo null functions is obvious by the second property of $\TT$.
\end{proof}

\begin{Proposition}
\label{prop:upper:bounded:Green:C:infty}
Let $g$ be a $D$-Green function of $\TT$-type.
Then we have the following:
\begin{enumerate}
\renewcommand{\labelenumi}{(\arabic{enumi})}
\item
If $g$ is of lower bounded type,
then locally $\vert \phi \vert \exp(-g/2)$ is  essentially bounded above for $\phi \in H_{\mathcal{M}}^0(X, D)$.

\item
If $g$ is of upper bounded type,
then there is a $D$-Green function $g'$ of $C^{\infty}$-type
such that $g \leq g'\ \aew$.
\end{enumerate}
\end{Proposition}

\begin{proof}
We set $D = \sum_{i=1}^l a_i D_i$ such that
$a_1, \ldots, a_l \in \RR$ and $D_i$'s are reduced and irreducible divisors on $X$.

(1)
Clearly we may assume that $X$ is connected.
For $x \in X$, let 
\[
g = u + \sum_{i=1}^l (-a_i) \log \vert f_i \vert^2\quad\aew
\]
be 
a local expression of $g$ around $x$, where $f_1, \ldots, f_l$ are local equations of $D_1, \ldots, D_l$.
For $\phi \in H^0_{\mathcal{M}}(X, D)$, we set $\phi = f_1^{b_1} \cdots f_l^{b_l} \cdot v$ around $x$ such that
$v$ has no factors of $f_1, \ldots, f_l$.
Then, as $(\phi) + D \geq 0$, we can see that $a_i + b_i \geq 0$ for all $i$, and that
$v$ is a holomorphic function around $x$. On the other hand,
\[
\exp(-g/2) \vert \phi \vert = \exp(-u/2) \vert f_1 \vert^{a_1 + b_1} \cdots \vert f_n \vert^{a_n + b_n} \vert v \vert\quad\aew,
\]
as required.

\medskip
(2) By our assumption, there is a locally finite open covering $\{ U_{\lambda} \}_{\lambda \in \Lambda}$ with the following properties:
\begin{enumerate}
\renewcommand{\labelenumi}{(\alph{enumi})}
\item
There are local equations $f_{\lambda, 1}, \ldots, f_{\lambda, n}$ of
$D_1, \ldots, D_n$ on $U_{\lambda}$.

\item
There is a constant $C_{\lambda}$ such that
$g \leq C_{\lambda} - \sum a_i \log \vert f_{\lambda, i} \vert^2\ \aew$ on $U_{\lambda}$.
\end{enumerate}
Let $\{ \rho_{\lambda} \}_{\lambda \in \Lambda}$ be a partition of unity subordinate to the covering $\{ U_{\lambda} \}_{\lambda \in \Lambda}$.
We set 
\[
g' = \sum_{\lambda \in \lambda} \rho_{\lambda}
\left(C_{\lambda} - \sum a_i \log \vert f_{\lambda, i} \vert^2\right).
\]
Clearly $g \leq g'\ \aew$.
Moreover, by Lemma~\ref{lem:green:partition:unity},
$g'$ is a $D$-Green function of $C^{\infty}$-type.
\end{proof}

\query{Add Proposition~\ref{prop:psh:plus:positive}
(12/July/2010)}
\begin{Proposition}
\label{prop:psh:plus:positive}
Let $g$ be a $D$-Green function of $(\Tpsh + C^{\infty})$-type. 
Let $A$ be an $\RR$-divisor on $X$, and let $h$ be an $A$-Green function of $C^{\infty}$-type.
Let $\alpha = c_1(A,h)$, that is,
$\alpha$ is a $C^{\infty}$ $(1,1)$-form on $X$ such that $dd^c([h]) + \delta_A = [\alpha]$ \rom{(}cf. Proposition~\rom{\ref{prop:lelong:formula}}\rom{)}.
\query{add ``$\alpha = c_1(A,h)$'' (24/Sep/2010)}
If $X$ is compact and $\alpha$ is positive, then there is a positive number $t_0$ such that
$g + th$ is a $(D + tA)$-Green function of $\Tpsh$-type  
for all $t \in \RR_{\geq t_0}$.
\end{Proposition}

\begin{proof}
For each $x \in X$, let
\[
g = u_x + \sum_{i} (-a_i) \log \vert f_i \vert^2\quad\aew,\qquad
h = v_x +  \sum_{i} (-b_i) \log \vert f_i \vert^2\quad\aew
\]
be local expressions of $g$ and $h$ respectively over an open neighborhood $U_x$ of $x$.
By our assumption, shrinking $U_x$ if necessarily, there are a plurisubharmonic function $p_x$ and
a $C^{\infty}$-function $q_x$ such that $u_x = p_x + q_x$.
Moreover, since $\alpha$ is positive, shrinking $U_x$ if necessarily, we can find a positive number $t_x$ such that
$dd^c(q_x) + t \alpha$ is positive for all $t \geq t_x$.
Because of the compactness of $X$,
we can choose finitely many $x_1, \ldots, x_r \in X$ such that
$X = U_{x_1} \cup \cdots \cup U_{x_r}$. If we set $t_0 = \max \{ t_{x_1}, \ldots, t_{x_r} \}$, then, for $t \geq t_0$,
\[
g + th 
= p_{x_j} + (q_{x_j} + t v_{x_{j}}) + \sum_i -(a_i + t b_i) \log \vert f_i \vert^2\quad\aew
\]
over $U_{x_j}$. Note that $dd^c(q_{x_j} + t v_{x_{j}}) = dd^c(q_{x_j}) + t \alpha$ is positive, which means that
$q_{x_j} + t v_{x_{j}}$ is a $C^{\infty}$-plurisubharmonic function. Thus
$g + th$ is of $\Tpsh$-type.
\end{proof}

\subsection{Partitions of Green functions}
\setcounter{Theorem}{0}
\label{subsec:Green:function:decomp}
Let $X$ be a $d$-equidimensional complex manifold.
Let $\TT$ be a type for Green functions. Besides the properties (1), (2) and (3) as in Subsection~\ref{subsec:Green:function:P},
we assume the following additional property (4):
\begin{enumerate}
\renewcommand{\labelenumi}{(\arabic{enumi})}
\setcounter{enumi}{3}
\item
For an open set $U$,
if $u \in \TT(U)$ and $v \in C^{\infty}(U)$, then
$vu \in \TT(U)$.
\end{enumerate}
As examples, $C^0$ and $C^{\infty}$ satisfy the property (4).

\begin{Lemma}
\label{lem:green:partition:unity}
Let $D$ be an $\RR$-divisor on $X$.
Let $\{ U_{\lambda} \}$ be a locally finite covering of $X$ and let
$\{ \rho_{\lambda} \}_{\lambda \in \Lambda}$ be a partition of unity subordinate to the covering $\{ U_{\lambda} \}_{\lambda \in \Lambda}$.
Let $g_{\lambda}$ be a $(\rest{D}{U_{\lambda}})$-Green function of $\TT$-type
on $U_{\lambda}$ for each $\lambda$.
Then $g := \sum_{\lambda} \rho_{\lambda} g_{\lambda}$ is a $D$-Green function of $\TT$-type
on $X$.
\end{Lemma}

\begin{proof}
We set $D = a_1 D_1 + \cdots + a_r D_r$.
Let $f_{i,x}$ be a local equation of $D_i$ on an open neighborhood $U_x$ of $x$.
As $g_{\lambda}$ is a $(\rest{D}{U_{\lambda}})$-Green function  of $\TT$-type
on $U_\lambda$,
for $\lambda$ with $x \in U_{\lambda}$, 
\[
g_{\lambda} = v_{\lambda,x} - \sum a_i \log \vert f_{i,x} \vert^2\quad\aew
\]
around $x$,
where $v_{\lambda,x} \in \TT(U_{\lambda} \cap U_x)$.
Thus
\begin{align*}
g & = \sum_{\lambda} \rho_{\lambda} ( v_{\lambda,x} - \sum a_i \log \vert f_{i,x} \vert^2) \quad \aew\\
& = \left( \sum_{\lambda} \rho_{\lambda} v_{\lambda,x} \right) -  \sum a_i \log \vert f_{i,x} \vert^2
\end{align*}
around $x$, as required.
\end{proof}

The main result of this subsection is the following proposition.

\begin{Proposition}
\label{prop:green:irreducible:decomp}
Let $g$ be a $D$-Green function of $\TT$-type
on $X$ and let
\[
D = b_1 E_1 + \cdots + b_r E_r
\]
be a decomposition such that
$E_1, \ldots, E_r \in \Div(X)$ and $b_1, \ldots, b_r \in \RR$.
Note that $E_i$ is not necessarily a prime divisor. Then we have the following:
\begin{enumerate}
\renewcommand{\labelenumi}{(\arabic{enumi})}
\item
There are locally integrable functions $g_1, \ldots, g_r$ 
such that $g_i$ is an $E_i$-Green function of $\TT$-type for each $i$ and
$g = b_1 g_1 + \cdots + b_r g_r\ \aew$.

\item
If $E_1, \ldots, E_r$ are effective, $b_1, \ldots, b_r \in \RR_{\geq 0}$, $g \geq 0\ \aew$ and $g$ is of lower bounded type, 
then
there are locally integrable functions $g_1, \ldots, g_r$
such that $g_i$ is a non-negative  $E_i$-Green function of $\TT$-type for each $i$ and
$g = b_1 g_1 + \cdots + b_r g_r\ \aew$.
\end{enumerate}
\end{Proposition}

\begin{proof}
(1)
Clearly we may assume that $b_i \not= 0$ for all $i$.
Let $g'_i$ be an $E_i$-Green function of $C^{\infty}$-type.
Then there is $f \in \TT(X)$ such that 
$f = g - (b_1 g'_1 + \cdots + b_r g'_r)\ \aew$.
Thus
\[
g = b_1(g'_1 + f/b_1) + b_2 g'_2 + \cdots + b_r g'_r \quad\aew.
\]

\medskip
(2) Clearly we may assume that $b_i > 0$ for all $i$.
First let us see the following claim:

\begin{Claim}
For each $x \in X$, there are locally integrable functions $g_{1,x}, \ldots, g_{r,x}$ and an open neighborhood $U_x$ of $x$ 
such that $g_{i,x}$ is a non-negative $E_i$-Green function  of $\TT$-type
on $U_x$ for every $i$, and that
$g = b_1 g_{1,x} + \cdots + b_r g_{r,x}\ \aew$ on $U_x$.
\end{Claim}

\begin{proof}
Let $U_x$ be a sufficiently small open neighborhood of $x$ and let $f_{i,x}$ be a local equation of $E_i$ on $U_x$ for every $i$.
Let  $g = v_x + \sum_{i=1}^r (-b_i) \log \vert f_{i,x} \vert^2\ \aew$ be the local expression of $g$ on $U_x$ with respect to $f_{1,x},
\ldots, f_{r,x}$.
We set $I = \{ i \mid f_{i,x}(x) = 0\}$ and $J = \{ i \mid f_{i,x}(x) \not= 0\}$.

First we assume $I = \emptyset$. Then, shrinking $U_x$ if necessarily,  we may assume that
\[
v_x + \sum_{i=1}^r (-b_i) \log \vert f_{i,x} \vert^2 \in \TT(U_x)
\]
and $E_i = 0$ on $U_x$ for all $i$. Thus if  we set 
\[
g_{i,x} = (1/rb_i) \left( v_x + \sum_{i=1}^r (-b_i) \log \vert f_{i,x} \vert^2\right)
\]
for each $i$, then we have our assertion.

Next we consider the case where $I \not= \emptyset$.
We put $f = v_x + \sum_{j \in J} (-b_j) \log \vert f_{j,x} \vert^2$. Then, shrinking $U_x$ if necessarily,  we may assume that
$f \in \TT(U_x)$ and is bounded below.
We set
\[
g_{i,x} = \begin{cases}
 f/(b_i \#(I)) - \log \vert f_{i,x} \vert^2 & \text{if $i \in I$}, \\
 0 & \text{if $i \in J$}.
\end{cases}
\]
Note that
$g = \sum_{i = 1}^r b_i g_{i,x}\ \aew$
and that $g_{i,x} \geq 0$ around $x$ for $i \in I$.
Thus, shrinking $U_x$ if necessarily, we have our assertion.
\end{proof}

By using the above claim, we can construct an open covering $\{ U_{\lambda} \}_{\lambda \in \Lambda}$ and
locally integrable functions $g_{1,\lambda}, \ldots, g_{r,\lambda}$ on $U_{\lambda}$ with the following properties:
\begin{enumerate}
\renewcommand{\labelenumi}{(\roman{enumi})}
\item
$\{ U_{\lambda} \}_{\lambda \in \Lambda}$ is locally finite and the closure of $U_{\lambda}$ is compact for every $\lambda$.

\item
$g_{i,\lambda}$ is a non-negative $E_i$-Green function  of $\TT$-type
on $U_\lambda$ for every $i$.

\item
$g = b_1 g_{1,\lambda} + \cdots + b_r g_{r,\lambda}\ \aew$ on $U_\lambda$.
\end{enumerate}
Let $\{ \rho_{\lambda} \}_{\lambda \in \Lambda}$ be a partition of unity subordinate to the covering $\{ U_{\lambda} \}_{\lambda \in \Lambda}$.
We set $g_i = \sum_{\lambda} \rho_{\lambda} g_{i,\lambda}$.
Clearly $g_i \geq 0$ and
\[
g =  \sum_{\lambda} \rho_{\lambda} \rest{g}{U_{\lambda}} \overset{\aew}{=} \sum_{\lambda} \rho_{\lambda} \sum_{i=1}^r b_i g_{i,\lambda} = \sum_{i=1}^r b_i g_i.
\]
Moreover, 
by Lemma~\ref{lem:green:partition:unity},
$g_i$ is an $E_i$-Green function of $\TT$-type.
\end{proof}

\subsection{Norms arising from Green functions}
\setcounter{Theorem}{0}
\label{subsec:Green:function:norm}
Let $X$ be a $d$-equidimensional complex manifold.
Let $g$ be a locally integral function on $X$.
For $\phi \in \mathcal{M}(X)$,
we define $\vert \phi \vert_{g}$ to be
\[
\vert \phi \vert_{g} := \exp(-g/2) \vert \phi \vert.
\]
Moreover,
the essential supremum of $\vert \phi \vert_{g}$ is denoted by $\Vert \phi \Vert_{g}$, that is,
\query{${\operatorname{ess\ sup}}$ $\Longrightarrow$ $\esssup$ (2010/12/06)}
\[
\Vert \phi \Vert_g :=\esssup \left\{ \vert \phi \vert_g(x) \mid x \in X \right\}.
\]

\begin{Lemma}
\label{lem:norm:locally:integrable}
\begin{enumerate}
\renewcommand{\labelenumi}{(\arabic{enumi})}
\item $\Vert \cdot \Vert_g$ satisfies the following properties:
\begin{enumerate}
\renewcommand{\labelenumii}{(\arabic{enumi}.\arabic{enumii})}
\item
$\Vert \lambda \phi \Vert_g  = \vert \lambda \vert \Vert \phi \Vert_g$ for all $\lambda \in \CC$ and $\phi \in \mathcal{M}(X)$.

\item
$\Vert \phi + \psi \Vert_g \leq \Vert \phi \Vert_g + \Vert \psi \Vert_g$ for all $\phi,\psi \in \mathcal{M}(X)$.

\item
For $\phi \in \mathcal{M}(X)$, $\Vert \phi \Vert_g = 0$ if and only if $\phi = 0$.
\end{enumerate}
\item Let $V$ be a  vector subspace of $\mathcal{M}(X)$ over $\CC$.
If $\Vert \phi \Vert_g < \infty$ for all $\phi \in V$, then
$\Vert\cdot\Vert_g$ yields a norm on $V$.
In particular, if $D$ is an $\RR$-divisor, $g$ is a $D$-Green function of $\TT$-type and
$g$ is of lower bounded type, then
$\Vert\cdot\Vert_g$ is a norm of $H^0_{\mathcal{M}}(X, D)$ \rom{(}cf. Proposition~\rom{\ref{prop:upper:bounded:Green:C:infty}}\rom{)},
where $\TT$ is a type for Green functions.
\end{enumerate}
\end{Lemma}

\begin{proof}
(1) (1.1) and (1.2) are obvious.
If $\Vert \phi \Vert_g = 0$, then $\vert \phi \vert_g = 0\ \aew$.
Moreover, as $g$ is integrable, the measure of $\{ x \in X \mid g(x) = \infty \}$ is zero.
Thus $\vert \phi \vert = 0\ \aew$, 
and hence $\phi = 0$.

(2) follows from (1).
\end{proof}

Let $\Phi$ be a continuous volume form on $X$.
For $\phi, \psi \in \mathcal{M}(X)$,
if $\phi \bar{\psi} \exp(-g)$ is integrable, then we denote its integral
\[
\int_{X} \phi \bar{\psi} \exp(-g) \Phi
\]
by $\langle \phi, \psi \rangle_{g}$.

We assume that $g$ is a $D$-Green function of $C^0$-type.
We set 
\[
D = a_1 D_1 + \cdots + a_l D_l,
\]
where
$D_i$'s are reduced and irreducible divisors on $X$ and $a_1, \ldots, a_r \in \RR$.
Let us fix $x \in X$. Let $f_1, \ldots, f_l$ be local equations of $D_1, \ldots, D_l$ around $x$, and let
\[
g = u + \sum_{i=1}^l (-a_i) \log \vert f_i \vert^2\quad\aew
\]
be the local expression of $g$ around $x$ with respect to $f_1, \ldots, f_l$.
For $\phi \in H^0_{\mathcal{M}}(X, D)$, we set $\phi = f_1^{b_1} \cdots f_l^{b_l} v$ around $x$,
where $v$ has no factors of $f_1, \ldots, f_l$.
Note that $b_1, \ldots, b_l$ do not depend on the choice of $f_1, \ldots, f_l$.
Since $(\phi) + D \geq 0$, we have $a_i + b_i \geq 0$ for all $i$ and
$v$ is holomorphic around $x$. Then
\[
\vert \phi \vert_g = \vert f_1 \vert^{a_1 + b_1} \cdots \vert f_l \vert^{a_l + b_l} \vert v \vert \exp(-u/2)\quad\aew.
\]
Let us choose another local equations $f'_1, \ldots, f'_l$ of $D_1, \ldots, D_l$ around $x$, and let
\[
g = u' + \sum_{i=1}^l (-a_i) \log \vert f'_i \vert^2\quad\aew
\]
be the local expression of $g$ around $x$ with respect to $f'_1, \ldots, f'_l$.
Moreover, we set $\phi = {f'_1}^{b_1} \cdots {f'_l}^{b_l} v'$ around $x$ as before.
Then
\[
\vert \phi \vert_g = \vert f'_1 \vert^{a_1 + b_1} \cdots \vert f'_l \vert^{a_l + b_l} \vert v' \vert \exp(-u'/2)\quad\aew.
\]
Note that 
\[
\vert f_1 \vert^{a_1 + b_1} \cdots \vert f_l \vert^{a_l + b_l} \vert v \vert \exp(-u/2)\quad\text{and}\quad
\vert f'_1 \vert^{a_1 + b_1} \cdots \vert f'_l \vert^{a_l + b_l} \vert v' \vert \exp(-u'/2)
\]
are continuous, so that
\[
\vert f_1 \vert^{a_1 + b_1} \cdots \vert f_l \vert^{a_l + b_l} \vert v \vert \exp(-u/2) =
\vert f'_1 \vert^{a_1 + b_1} \cdots \vert f'_l \vert^{a_l + b_l} \vert v' \vert \exp(-u'/2)
\]
around $x$. This observation shows that there is a unique continuous function $h$ on $X$ such that
$\vert \phi \vert_g = h\ \aew$. In this sense, in the case where $g$ is of $C^0$-type,
we always assume that
$\vert \phi \vert_g$ means the above continuous function $h$.
Then we have the following proposition.

\begin{Proposition}
Let $g$ be a $D$-Green function of $C^0$-type.
\begin{enumerate}
\renewcommand{\labelenumi}{(\arabic{enumi})}
\item
For $\phi \in H^0_{\mathcal{M}}(X, D)$, $\vert \phi \vert_g$ is locally bounded above.

\item
If $X$ is compact,
then $\langle \phi, \psi \rangle_{g}$ exists for $\phi, \psi \in H^0_{\mathcal{M}}(X, D)$.
Moreover, $\langle\ , \ \rangle_{g}$ yields a hermitian inner product on $H^0_{\mathcal{M}}(X, D)$.
\end{enumerate}
\end{Proposition}

\section{Gromov's inequality and distorsion functions for $\RR$-divisors}
Let $X$ be a $d$-equidimensional compact complex manifold.
Let $D$ be an $\RR$-divisor on $X$ and let $g$ be a $D$-Green function of $C^0$-type.
Let us fix a continuous volume form $\Phi$ on $X$.
Recall that $\vert \phi \vert_g$, $\Vert \phi \Vert_g$ and $\langle \phi, \psi \rangle_g$ for
$\phi, \psi \in H^0_{\mathcal{M}}(X, D)$ are given by
\query{${\operatorname{ess\ sup}}$ $\Longrightarrow$ $\esssup$ (2010/12/06)}
\[
\begin{cases}
\vert \phi \vert_g := \vert \phi \vert \exp(-g/2), \\
\Vert \phi \Vert_g := \esssup \{ \vert \phi \vert_g(x) \mid x \in X \}, \\
{\displaystyle \langle \phi, \psi \rangle_g = \int_X \phi \bar{\psi} \exp(-g) \Phi}.
\end{cases}
\]
As described in Subsection~\ref{subsec:Green:function:norm},
we can view $\vert \phi \vert_g$ as a continuous function, so that
$\vert \phi \vert_g$ is always assumed to be continuous.

In this section, let us consider Gromov's inequality and distorsion functions for $\RR$-divisors.

\subsection{Gromov's inequality for $\RR$-divisors}
\setcounter{Theorem}{0}

Here we observe Gromov's inequality for $\RR$-divisors.

\begin{Proposition}[Gromov's inequality for an $\RR$-divisor]
\label{prop:Gromov:ineq}
Let $D_1, \ldots, D_l$ be $\RR$-divisors on $X$ and let $g_1, \ldots, g_l$ be locally integrable functions on $X$ such that
$g_i$ is a $D_i$-Green function of $C^{\infty}$-type for each $i$.
Then there is a positive constant $C$ such that
\[
\Vert \phi \Vert^2_{a_1 g_1 + \cdots + a_l g_l } \leq C (1 + \vert a_1 \vert + \cdots + \vert a_l \vert)^{2d}  \langle \phi, \phi  \rangle_{a_1 g_1 + \cdots + a_l g_l }
\]
holds for all $\phi \in H^0_{\mathcal{M}}(X, a_1D_1 + \cdots + a_l D_l)$ and $a_1, \ldots, a_l \in \RR$.
\end{Proposition}

\begin{proof}
We can find distinct prime divisors $\Gamma_{1}, \ldots, \Gamma_r$ on $X$, 
locally integrable functions $\gamma_{1}, \ldots, \gamma_{r}$ on $X$,
$C^{\infty}$-functions $f_1, \ldots, f_l$ and 
real numbers $\alpha_{ij}$ such that
$\gamma_{j}$ is a $\Gamma_{j}$-Green functions of $C^{\infty}$-type for each $j=1, \ldots, r$,
\[
D_i = \sum_{j=1}^r \alpha_{ij} \Gamma_{j}\quad\text{and}\quad
g_i = f_i + \sum_{j=1}^r \alpha_{ij} \gamma_{j}\ \aew.
\]
Then
\begin{align*}
a_1 D_1 + \cdots + a_l D_l & = \sum_{j=1}^r \left( \sum_{i=1}^l a_i \alpha_{ij}\right) \Gamma_{j} + \sum_{i=1}^l a_i (\text{the zero divisor}), \\
a_1 g_1 + \cdots + a_l g_l & =  \sum_{j=1}^r \left(\sum_{i=1}^l a_i \alpha_{ij}\right) \gamma_{j} + \sum_{i=1}^l a_i f_i\quad\aew.
\end{align*}
Moreover, if we set $A = \max \{ \vert \alpha_{ij} \vert \}$, then
\[
1 + \sum_{i=1}^l \vert a_i \vert + \sum_{j=1}^r \left\vert \sum_{i=1}^l a_i \alpha_{ij} \right\vert \leq 1 + (A r+1) \sum_{i=1}^l \vert a_i \vert \leq
(Ar+1) \left( 1 + \sum_{i=1}^l \vert a_i \vert\right).
\]
Thus we may assume that $D_1, \ldots D_r$ are distinct prime divisors and 
\[
D_{r+1} = \cdots = D_l = 0.
\]

Let $U$ be an open set of $X$ over which there are local equations $f_1, \ldots, f_r$ of $D_1, \ldots, D_r$ respectively.

\begin{Claim}
For all $\phi \in H^0_{\mathcal{M}}(X, a_1D_1 + \cdots + a_l D_l)$ and $a_1, \ldots, a_l \in \RR$,
\[
\phi f_1^{\lfloor a_1 \rfloor} \cdots f_r^{\lfloor a_r \rfloor}
\]
is holomorphic over $U$, that is,
there are $b_1, \ldots, b_r \in \ZZ$ and a holomorphic function $f$ on $U$ such that
$\phi = f_1^{b_1} \cdots f_r^{b_r} f$ and $b_1 + a_1 \geq 0, \ldots, b_r + a_r \geq 0$.
\end{Claim}

\begin{proof}
Fix $x \in U$. Let $f_i = e_{i} f_{i1} \cdots f_{ic_i}$ be the prime decomposition of $f_i$ in $\OO_{X,x}$,
where $e_i \in \OO_{X,x}^{\times}$ and $f_{ij}$'s are distinct prime elements of $\OO_{X,x}$.
Let $D_{ij}$ be the prime divisor given by $f_{ij}$ around $x$.
Since $\phi \in H^0_{\mathcal{M}}(X, a_1D_1 + \cdots + a_l D_l)$, we have 
\ifpxfont
\[
(\phi) + a_1D_1 + \cdots + a_l D_l
= (\phi) + a_1 D_{11} + \cdots + a_1 D_{1c_1} + \cdots
+ a_r D_{r1} + \cdots + a_r D_{rc_r} \geq 0
\]
\else
\begin{multline*}
(\phi) + a_1D_1 + \cdots + a_l D_l \\
= (\phi) + a_1 D_{11} + \cdots + a_1 D_{1c_1} + \cdots
+ a_r D_{r1} + \cdots + a_r D_{rc_r} \geq 0
\end{multline*}
\fi
around $x$. 
Note that $ D_{11}, \ldots,  D_{1c_1}, \ldots, D_{r1}, \ldots,  D_{rc_r}$ are distinct prime divisors around $x$.
Thus $\phi  f_{11}^{\lfloor a_1 \rfloor} \cdots  f_{1c_1}^{\lfloor a_1 \rfloor} \cdots f_{r1}^{\lfloor a_r \rfloor} \cdots f_{rc_r}^{\lfloor a_r \rfloor}$
is holomorphic around $x$. Therefore, as
\[
f_1^{\lfloor a_1 \rfloor} \cdots f_r^{\lfloor a_r \rfloor} = e_1^{\lfloor a_1 \rfloor} \cdots e_r^{\lfloor a_r \rfloor} 
f_{11}^{\lfloor a_1 \rfloor} \cdots  f_{1c_1}^{\lfloor a_1 \rfloor} \cdots f_{r1}^{\lfloor a_r \rfloor} \cdots f_{rc_r}^{\lfloor a_r \rfloor},
\]
$\phi f_1^{\lfloor a_1 \rfloor} \cdots f_r^{\lfloor a_r \rfloor}$ is holomorphic around $x$.
\end{proof}

By the above observation, the assertion of the proposition follows from the following local version.
\end{proof}

\begin{Lemma}
\label{lem:Gromov:ineq:R:div:local}
Let $a, b, c$ be real numbers with $a > b > c > 0$.
We set 
\[
U = \{ z \in \CC^d \mid \vert z \vert < a \},\ 
V = \{ z \in \CC^d \mid \vert z \vert < b \}\ \text{and}\ 
W = \{ z \in \CC^d \mid \vert z \vert < c \}.
\]
Let $\Phi$ be a continuous volume form on $U$,
$f_1, \ldots, f_l \in \OO_U(U)$,
$v_1, \ldots, v_l \in C^{\infty}(U)$ and
\[
g_i = v_i - \log \vert f_i \vert^2
\]
for $i=1, \ldots, l$.
For $a_1, \ldots, a_l \in \RR$, we set
\[
V(a_1, \ldots, a_l) = \left\{ f_1^{b_1} \cdots f_l^{b_l} f \left|
\begin{array}{ll}
\text{$f \in \OO_U(U)$ and
$b_1, \ldots, b_l \in \ZZ$ with} \\
\text{$b_1 + a_1 \geq 0, \ldots,  b_l + a_l \geq 0$}
\end{array}
\right\}\right..
\]
\rom{(}Note that $V(a_1, \ldots, a_l)$ is a complex vector space.\rom{)}
Then there is a positive constant $C$ such that
\begin{multline*}
\max_{ z \in \overline{W}} \{ \vert \phi \vert^2 \exp(-a_1 g_1 - \cdots - a_l g_l)(z) \} \\
\leq C (\vert a_1 \vert + \cdots + \vert a_l \vert + 1)^{2d} \int_V \vert \phi \vert^2 \exp(-a_1 g_1 - \cdots - a_l g_l)\Phi
\end{multline*}
holds for all $\phi \in V(a_1, \ldots, a_l)$ and all $a_1, \ldots, a_l \in \RR$.
\end{Lemma}

\begin{proof}
We set
\[
u_1 = \exp(-v_1), \ldots, u_l = \exp(-v_l), u_{l+1} = \exp(v_1), \ldots, u_{2l} = \exp(v_l).
\]
Then in the same way as \cite[Lemma~1.1.1]{MoCont},
we can find a positive constant $D$ with the following properties:
\begin{enumerate}
\renewcommand{\labelenumi}{(\alph{enumi})}
\item
For $x_0, x \in \overline{V}$,
$u_i(x) \geq u_i(x_0)(1 - D \vert x - x_0 \vert')$ for all $i=1, \ldots, 2l$,
where $\vert z \vert' = \vert z_1 \vert + \cdots + \vert z_d \vert$ for $z = (z_1, \ldots, z_d) \in \CC^d$.

\item
If $x_0 \in \overline{W}$, then $B(x_0, 1/D) \subseteq \overline{V}$, where
\[
B(x_0, 1/D) = \{ x \in \CC^d \mid \vert x - x_0 \vert' \leq 1/D\}.
\]

\end{enumerate}
We set
\[
\Phi_{can} = \left(\frac{\sqrt{-1}}{2}\right)^d dz_1 \wedge d\bar{z}_1 \wedge \cdots \wedge dz_d \wedge d\bar{z}_d.
\]
\query{$\left(\frac{\sqrt{-1}}{2}\right)$ $\Longrightarrow$ $\left(\frac{\sqrt{-1}}{2}\right)^d$
(10/September/2010)}
Then we can choose a positive constant $e$ with $\Phi \geq e \Phi_{can}$.
For 
\[
\phi = f_1^{b_1} \cdots f_l^{b_l} f \in V(a_1, \ldots, a_l),
\]
we assume that the continuous function
\ifpxfont
\[
\vert \phi \vert^2 \exp(-a_1 g_1 - \cdots - a_l g_l)
= \vert f_1 \vert^{2(b_1 + a_1)} \cdots \vert f_l \vert^{2(b_l + a_l)} \vert f \vert^2  \exp(-a_1 v_1 - \cdots - a_l v_l)
\]
\else
\begin{multline*}
\vert \phi \vert^2 \exp(-a_1 g_1 - \cdots - a_l g_l) \\
= \vert f_1 \vert^{2(b_1 + a_1)} \cdots \vert f_l \vert^{2(b_l + a_l)} \vert f \vert^2  \exp(-a_1 v_1 - \cdots - a_l v_l)
\end{multline*}
\fi
on $\overline{W}$ takes the maximal value at $x_0 \in \overline{W}$.
Let us choose $\epsilon_i \in \{\pm 1\}$ such that $a_i = \epsilon_i \vert a_i \vert$.
Note that
\begin{multline*}
\exp(-a_1 v_1(x) - \cdots - a_l v_l(x)) = \prod_{i=1}^l \exp(-\epsilon_i v_i(x))^{\vert a_i \vert} \\
\geq \left( \prod_{i=1}^l \exp(-\epsilon_i v_i(x_0))^{\vert a_i \vert}\right) (1 -D\vert x - x_0\vert')^{\vert a_1 \vert + \cdots
+ \vert a_l \vert} \\
= \exp(-a_1 v_1(x_0) - \cdots - a_l v_l(x_0)) (1 -D\vert x - x_0\vert')^{\vert a_1 \vert + \cdots
+ \vert a_l \vert}
\end{multline*}
on $B(x_0,1/D)$.
Therefore,
\begin{multline*}
\int_V \vert \phi \vert^2 \exp(-a_1 g_1 - \cdots - a_l g_l) \Phi 
\geq
e \exp(-a_1 v_1(x_0) - \cdots - a_l v_l(x_0))  \times \\
\int_{B(x_0, 1/D)} \vert f_1 \vert^{2(b_1 + a_1)} \cdots \vert f_l \vert^{2(b_l + a_l)} \vert f \vert^2 (1 -D\vert x - x_0\vert')^{\vert a_1 \vert + \cdots
+ \vert a_l \vert} \Phi_{can}.
\end{multline*}
If we set
$x - x_0 = (r_1 \exp(\sqrt{-1}\theta_1), \ldots, r_d \exp(\sqrt{-1}\theta_d))$,
then, by using \cite[Theorem~4.1.3]{Hor} and the pluriharmonicity of $\vert f_1 \vert^{2(b_1 + a_1)} \cdots \vert f_l \vert^{2(b_l + a_l)} \vert f \vert^2$,
\begin{multline*}
 \int_{B(x_0, 1/D)}
 \vert f_1 \vert^{2(b_1 + a_1)} \cdots \vert f_l \vert^{2(b_l + a_l)} \vert f \vert^2 (1 -D\vert x - x_0\vert')^{\vert a_1 \vert + \cdots
+ \vert a_l \vert} \Phi_{can} \\
=
 \int_{\substack{r_1 + \cdots + r_d \leq 1/D \\ r_1 \geq 0, \ldots, r_d \geq 0}} 
\left( \int_0^{2\pi} \cdots \int_0^{2\pi}  \vert f_1 \vert^{2(b_1 + a_1)} \cdots \vert f_l \vert^{2(b_l + a_l)} \vert f \vert^2 d \theta_1 \cdots d\theta_d \right) \\
\times
r_1 \cdots r_d  (1 - D(r_1 + \cdots + r_d) )^{\vert a_1 \vert + \cdots
+ \vert a_l \vert} dr_1 \cdots dr_d \\
\hspace{-7em}\geq (2\pi)^d
\vert f_1(x_0) \vert^{2(b_1 + a_1)} \cdots \vert f_l(x_0)\vert^{2(b_l + a_l)} \vert f(x_0) \vert^2 \\
\hspace{5em}\times  \int_{[0, 1/(dD)]^d} 
r_1 \cdots r_d  (1 - D(r_1 + \cdots + r_d) )^{\vert a_1 \vert + \cdots
+ \vert a_l \vert} dr_1 \cdots dr_d.
\end{multline*}
Therefore, we have
\begin{multline*}
\int_V \vert \phi \vert^2 \exp(-a_1 g_1 - \cdots - a_l g_l) \Phi  \\
\hspace{-7em}\geq
\frac{e (2\pi)^d}{(dD)^{2d}} \max_{z \in \overline{W}} \{ \vert \phi \vert^2 \exp(-a_1 g_1 - \cdots - a_l g_l)(z) \} \\
\times
 \int_{\substack{[0,1]^d}} 
t_1 \cdots t_d  (1 - (1/d)(t_1 + \cdots + t_d) )^{\vert a_1 \vert + \cdots
+ \vert a_l \vert} dt_1 \cdots dt_d.
\end{multline*}
Hence our assertion follows from \cite[Claim~1.1.1.1 in Lemma~1.1.1]{MoCont}.
\end{proof}

\subsection{Distorsion functions for $\RR$-divisors}
\label{subsec:distorsion:R:div}
\setcounter{Theorem}{0}

Let 
$D$ be an $\RR$-divisor on $X$ and let $g$ be a $D$-Green function of $C^0$-type.
Let $V$ be a complex vector subspace of $H^0_{\mathcal{M}}(X, D)$.
Let $\phi_1, \ldots, \phi_{l}$ be an orthonormal basis of $V$ with respect to $\langle\ ,\rangle_{g}$.
It is easy to see that
\[
\vert \phi_1 \vert_{g}^2 + \cdots + \vert \phi_{l} \vert_{g}^2
\]
does not depend on the choice of the orthonormal basis $\phi_1, \ldots, \phi_{l}$ of $V$, so that
it is denoted by $\dist(V;g)$ and it is called the {\em distorsion function} of $V$ with respect to $g$.

\begin{Proposition}
\label{prop:decomposition:hermitian:metric}
Let $V$ be a complex vector subspace of $H^0(X, D)$. Then an inequality
\[
\vert s \vert^2_g(x) \leq  \langle s, s \rangle_g  \dist(V; g)(x)\quad(\forall x \in X)
\]
holds for all $s \in V$.
In particular,
\[
\vert s \vert_g(x) \leq  \left(\int_X \Phi\right)^{1/2}\Vert s \Vert_{g} \sqrt{\dist(V;g)(x) }\quad(\forall x \in X).
\]
\end{Proposition}

\begin{proof}
Let $e_1, \ldots, e_{N}$ be an orthonormal basis of $V$ with respect to $\langle\ , \ \rangle_g$.
If we set $s = a_1 e_1 + \cdots + a_{N} e_{N}$ for $s \in V$,
then 
\[
\langle s, s \rangle_g  = \vert a_1 \vert^2 + \cdots + \vert a_{N} \vert^2.
\]
Therefore, by the Cauchy-Schwarz inequality,
\begin{multline*}
\vert s \vert_g(x) \leq \vert a_1 \vert \vert e_1 \vert_g (x) + \cdots +
 \vert a_{N} \vert \vert e_{N} \vert_g (x) \\
 \leq
 \sqrt{\vert a_1 \vert^2  + \cdots +
 \vert a_{N} \vert^2 } \sqrt{\vert e_1 \vert_g^2(x) + \cdots +
 \vert e_{N} \vert_g^2 (x) } \\
= \sqrt{\langle s, s \rangle_g} \sqrt{\dist(V; \overline{L})(x) }.
 \end{multline*}
\end{proof}

\begin{Lemma}
\label{lem:dist:comp}
Let $g'$ be another $D$-Green function of $C^0$-type such that $g \leq g'\ \aew$.
Let $V$ be a complex vector subspace of $H^0(X, D)$. Then 
$\dist(V; g) \leq \exp(g'-g)\dist(V; g')$.
\end{Lemma}

\begin{proof}
We can find a continuous function $u$ on $X$
such that $u \geq 0$ on $X$ and
$g' = g + u \ \aew$.
Let $\phi_1, \ldots, \phi_{l}$ be an orthonormal basis of $V$ with respect to $\langle\ ,\ \rangle_{g'}$ such that
$\phi_1, \ldots, \phi_{l}$ are orthogonal with respect to $\langle\ ,\ \rangle_{g}$.
This is possible because any hermitian matrix can be diagonalizable by a unitary matrix.
Then
\[
\frac{\phi_1}{\sqrt{\langle\phi_1,\phi_1 \rangle_{g}}}, \ldots, \frac{\phi_{l}}{\sqrt{\langle\phi_{l},\phi_{l} \rangle_{g}}}
\]
form an orthonormal basis of $V$ with respect to $\langle\ ,\ \rangle_{g}$.
Thus
\[
\dist(V;  g) = \frac{\vert \phi_1\vert^2_{g}}{\langle\phi_1,\phi_1 \rangle_{g}} + \cdots + \frac{\vert \phi_{l}\vert^2_{g} }{\langle\phi_{l},\phi_{l} \rangle_{g}}. 
\]
On the other hand, as  $\vert \phi_i \vert^2_{g'}\exp(u) = \vert \phi_i \vert^2_g$,
\[
\langle\phi_i,\phi_i \rangle_{g} = \int_{X} \vert \phi_i \vert^2_{g'} \exp(u)\Phi \geq  \int_{X} \vert \phi_i \vert^2_{g'} \Phi = 1
\]
Therefore the lemma follows.
\end{proof}

Let us consider the following fundamental estimate.

\begin{Theorem}
\label{thm:dist:dist:dist:ineq}
Let $R = \bigoplus_{n\geq 0} R_n$ be a graded subring of $\bigoplus_{n\geq 0} H^0_{\mathcal{M}}(X, nD)$.
If $g$ is a $D$-Green function of $C^{\infty}$-type, then
there is a positive constant $C$ with the following properties:
\begin{enumerate}
\renewcommand{\labelenumi}{(\arabic{enumi})}
\item $\dist(R_n;ng) \leq C(n+1)^{3d}$ for all $n \geq 0$.

\item
${\displaystyle \frac{\dist(R_n;ng)}{C(n+1)^{3d}} \cdot \frac{\dist(R_m;mg)}{C(m+1)^{3d}} \leq \frac{\dist(R_{n+m};(n+m)g)}{C(n+m+1)^{3d}}}$
for all $n, m \geq 0$.
\end{enumerate}
\end{Theorem}

\begin{proof}
Let us begin with the following claim:

\begin{Claim}
There is a positive constant $C_1$ such that
$\dist(R_n;ng) \leq C_1 (n+1)^{3d}$ for all $n \geq 0$
\end{Claim}

\begin{proof}
First of all, by Gromov's inequality for an $\RR$-divisor (cf. Proposition~\ref{prop:Gromov:ineq}), there is a positive constant $C'$ such that
\[
\Vert \phi \Vert_{ng}^2 \leq C' (n+1)^{2d} \langle \phi, \phi \rangle_{ng}
\]
for all $\phi \in H^0_{\mathcal{M}}(X, nD)$ and
$n \geq 0$.
Let $\phi_1, \ldots, \phi_{l_n}$ be an orthonormal basis of $R_n$. Then
\begin{multline*}
\dist(R_n;ng) \leq \Vert \phi_1 \Vert_{ng}^2 + \cdots + \Vert \phi_{l_n} \Vert_{ng}^2 \\
\leq C' (n+1)^{2d} ( \langle \phi_1, \phi_1 \rangle_{ng} + \cdots +  \langle \phi_{l_g}, \phi_{l_g} \rangle_{ng})
\leq C' (n+1)^{2d} \dim R_n,
\end{multline*}
as required.
\end{proof}

\begin{Claim}
There is a positive constant $C_2$ such that
\[
\dist(R_n;ng) \cdot \dist(R_m;mg) \leq C_2(m+1)^{3d} \dist(R_{n+m};(n+m)g)
\]
for $n \geq m \geq 0$.
\end{Claim}

\begin{proof}
Let $t_1, \ldots, t_l$ be an orthonormal basis of $R_m$.
For each $j= 1, \ldots, l$,
we choose an orthonormal basis $s_1, \ldots, s_r$ of $R_n$ such that
$s_1 t_j, \ldots, s_r t_j$ are orthogonal in $R_{n+m}$.
We set $I = \{ 1 \leq i \leq r \mid s_i t_j \not= 0 \}$.
As 
\[
\left\{ \frac{s_i t_j}{\sqrt{\langle s_i t_j, s_i t_j\rangle_{(n+m)g}}} \right\}_{i \in I}
\]
can be extended to an orthonormal basis of $R_{n+m}$, we have
\[
\sum_{i \in I} \frac{\vert s_i t_j\vert_{(n+m)g}^2}{\langle s_i t_j, s_i t_j\rangle_{(n+m)g}} 
\leq \dist(R_{n+m};(n+m)g).
\]
By using Gromov's inequality as in the previous claim,
\[
 \langle s_i t_j, s_i t_j \rangle_{(n+m)g} \leq \langle s_i, s_i \rangle_{ng} \Vert t_j \Vert^2_{mg} \leq C' (m+1)^{2d}  \langle t_j, t_j  \rangle_{mg} = C' (m+1)^{2d}.
\]
Hence
\begin{multline*}
\dist(R_n;ng) \vert t_j \vert_{mg}^2 = \sum_{i=1}^r \vert s_i t_j \vert_{(n+m)g} = \sum_{i \in I} \vert s_i t_j \vert_{(n+m)g} \\ 
\leq \sum_{i \in I}
\frac{C' (m+1)^{2d}}{\langle s_i t_j, s_i t_j\rangle_{(n+m)g}} \vert s_i t_j\vert_{(n+m)g}^2 \\
\leq C' (m+1)^{2d} \dist(R_{n+m};(n+m)g),
\end{multline*}
which implies 
\[
\dist(R_n;ng) \cdot \dist(R_m;mg) \leq \dim (R_m) C'  (m+1)^{2d} \dist(R_{n+m};(n+m)g),
\]
as required.
\end{proof}

We set $C = \max\{ C_1, 8^d C_2\}$. Then, for $n \geq m \geq 0$,
\begin{align*}
\frac{C(n+1)^{3d} C(m+1)^{3d}}{C(n+m+1)^{3d}} & \geq C_2(m+1)^{3d} 8^d \left( \frac{n+1}{n+m+1}\right)^{3d} \\
& \geq  C_2(m+1)^{3d} 8^d \left( \frac{n+1}{2n+1}\right)^{3d} \\
& >  C_2(m+1)^{3d} 8^d \left( \frac{1}{2}\right)^{3d} = C_2(m+1)^{3d}.
\end{align*}
Thus the proposition follows from the above claims.
\end{proof}

\renewcommand{\theTheorem}{\arabic{section}.\arabic{Theorem}}
\renewcommand{\theClaim}{\arabic{section}.\arabic{Theorem}.\arabic{Claim}}
\renewcommand{\theequation}{\arabic{section}.\arabic{Theorem}.\arabic{Claim}}

\section{Plurisubharmonic upper envelopes}
The main result of this section is the continuity of the upper envelope
of a family of Green functions of $\Tpsh_{\RR}$-type bounded above by
a Green function of $C^0$-type.
This will give the continuity of the positive part of the Zariski decomposition.

Throughout this section, let $X$ be a $d$-equidimensional complex manifold.
Let us begin with the following fundamental estimate.

\begin{Lemma}
\label{lem:subharmonic:local:bound}
Let 
$f_1, \ldots, f_r$ be holomorphic functions on $X$ such that $f_1, \ldots, f_r$ are not zero on each connected component of $X$. Let $a_1, \ldots, a_r \in \RR_{\geq 0}$ and $M \in \RR$.
We denote by $\Tpsh(X; f_1, \ldots, f_r, a_1, \ldots, a_r, M)$ the set of
all
plurisubharmonic functions $u$  on $X$ such that
\[
u \leq M - \sum_{i=1}^r a_i \log \vert f_i \vert^2 \quad\aew
\]
holds over $X$. Then,
for each point $x \in X$, there are an open neighborhood $U_x$ of $x$ and a constant $M'_x$ depending only on $f_1, \ldots, f_r$ and $x$ such that
\[
u \leq M + M'_x(a_1 + \cdots + a_r)
\]
on $U_x$ for any $u \in \Tpsh(X; f_1, \ldots, f_r, a_1, \ldots, a_r, M)$.
\end{Lemma}

\begin{proof}
Let us begin with the following claim:

\begin{Claim}
For any $u \in \Tpsh(X; f_1, \ldots, f_r, a_1, \ldots, a_r, M)$,
\[
u \leq M - \sum_{i=1}^r a_i \log \vert f_i \vert^2
\]
holds over $X$ .
\end{Claim}

\begin{proof}
Clearly we may assume that $a_i > 0$ for all $i$.
Let us fix $x \in X$. If $f_i(x) = 0$ for some $i$, then
the right hand side is $\infty$, so that the assertion is obvious.
We assume that $f_i(x) \not= 0$ for all $i$. Then the right hand side is continuous around $x$.
Thus it follows from Lemma~\ref{lem:fqpssh:ineq:ae}.
\end{proof}

\begin{Claim}
Let $\epsilon \in \RR_{>0}$, $a_1, \ldots, a_d \in \RR_{\geq 0}$, and
$M \in \RR$. Then
\[
u \leq M - 2 \log(\epsilon/4) (a_1 + \cdots + a_d)
\]
holds on $\Delta_{\epsilon/4}^d$ for any $u \in \Tpsh(\Delta_{\epsilon}^d; z_1, \ldots, z_d, a_1, \ldots, a_d, M)$,
where $(z_1, \ldots, z_d)$ is the coordinate of $\CC^d$ and 
\[
\Delta_{t}^d = \{ (z_1,\ldots,z_d) \in \CC^d \mid \vert z_1 \vert < t, \ldots,  \vert z_d \vert < t \}.
\]
for $t \in \RR_{>0}$.
\end{Claim}

\begin{proof}
Note that if $(z_1, \ldots, z_d) \in \Delta_{\epsilon/4}^d$, then
\[
\{ (z_1 + (\epsilon/2)e^{2\pi i \theta_1}, \ldots, z_d + (\epsilon/2)e^{2\pi i \theta_d}) \mid \theta_1, \ldots, \theta_d \in [0, 1] \} \subseteq \Delta_{\epsilon}^d.
\]
Moreover, as
\begin{multline*}
\epsilon/2 = \vert (\epsilon/2)e^{2\pi i \theta_j} \vert = \vert z_j + (\epsilon/2)e^{2\pi i \theta_j} - z_j \vert \\
\leq \vert z_j + (\epsilon/2)e^{2\pi i \theta_j} \vert + \vert z_j \vert <
 \vert z_j + (\epsilon/2)e^{2\pi i \theta_j} \vert + \epsilon/4,
 \end{multline*}
we have $\vert z_j + (\epsilon/2)e^{2\pi i \theta_j} \vert >  \epsilon/4$ for $j=1,\ldots, d$.
Thus, by \cite[Theorem~4.1.3]{Hor},
\begin{align*}
u(z_1, \ldots, z_d) & \leq \int_{0}^1 \cdots \int_{0}^1u(z_1 + (\epsilon/2)e^{2\pi i \theta_1}, \ldots, z_d + (\epsilon/2)e^{2\pi i \theta_d}) d\theta_1 \cdots d\theta_d  \\
& \leq
\int_{0}^1 \cdots \int_{0}^1 \left(  M - \sum_{j=1}^n a_j \log \vert  z_j + (\epsilon/2)e^{2\pi i \theta_j} \vert^2 \right)d\theta_1 \cdots d\theta_d \\
& = M - \sum_{j=1}^d a_j \int_{0}^1\log \vert  z_j + (\epsilon/2)e^{2\pi i \theta_j} \vert^2 d\theta_j \\
& \leq M - \sum_{j=1}^d a_j  \int_{0}^1\log (\epsilon/4)^2 d\theta_j = M - 2\log (\epsilon/4) \sum_{j=1}^d a_j.
\end{align*}
\end{proof}

Next we observe the following claim:

\begin{Claim}
If $\Supp \{ x \in X \mid f_1(x) \cdots f_r(x) = 0 \}$ is a normal crossing divisor on $X$, then
the lemma holds.
\end{Claim}

\begin{proof}
We choose an open neighborhood $V_x$ such that $V_x = \Delta_1^d$ and
\[
\Supp \{ x \in X \mid f_1(x) \cdots f_r(x) = 0 \}
\]
is given by $\{ z_1 \cdots z_l = 0 \}$.
Then there are $b_{ij} \in \ZZ_{\geq 0}$ and nowhere vanishing holomorphic functions
$v_1, \ldots, v_r$ on $\Delta_1^d$ such that
\[
f_1 = z_1^{b_{11}} \cdots z_l^{b_{1l}} v_1, \ldots, f_r = z_1^{b_{r1}} \cdots z_l ^{b_{rl}} v_r.
\]
Thus
\[
M - \sum_{i=1}^r a_i \log \vert f_i \vert^2 = M - \sum_{i=1}^r a_i \log \vert v_i \vert^2 - \sum_{j=1}^l \left( \sum_{i=1}^r a_i b_{ij} \right) \log \vert z_j \vert^2.
\]
We choose $M_1, M_2\in \RR$ such that $M_1 = \max \{ b_{ij} \mid i=1, \ldots, r, j = 1, \ldots, l \}$ and
$M_2 \geq \max_{z \in \Delta_{1/2}^d} \{ - \log \vert v_i(z) \vert^2 \}$ for all $i$.
Then
\[
M - \sum_{i=1}^r a_i \log \vert f_i \vert^2 \leq M + M_2(a_1 + \cdots + a_r) - \sum_{j=1}^l M_1(a_1 + \cdots + a_r) \log \vert z_j \vert^2
\]
on $\Delta_{1/2}^d$. Thus, by the previous claim, for any
$u \in \Tpsh(X; f_1, \ldots, f_r, a_1, \ldots, a_r, M)$,
\[
u \leq M + (M_2 - 2 \log(1/8) lM_1)(a_1 + \cdots + a_r)
\]
on $\Delta_{1/8}^d$.
\end{proof}

Let us start a general case.
Let $\pi : X' \to X$ be a proper bimeromorphic map such that
$\Supp ( \{ \pi^*(f_1) \cdots \pi^*(f_r) = 0 \})$ is a normal crossing divisor on $X'$.
Note that if $u$ is a plurisubharmonic function on $X$, then
$\pi^*(u)$ is also plurisubharmonic on $X'$ (cf. \cite[Corollary~2.9.5]{MK}).
By the above claim, for each point $y \in \pi^{-1}(x)$, there is an open neighborhood $U_y$ of $y$ and
a constant $M'_y$ depending only on $f_1, \ldots, f_r$ and $y$ such that, for any
$u \in \Tpsh(X; f_1, \ldots, f_r, a_1, \ldots, a_r, M)$,
\[
f^*(u) \leq M + M'_y (a_1 + \cdots + a_r)
\]
on $U_y$.
As $\pi^{-1}(x) \subseteq \bigcup_{y \in \pi^{-1}(x)} U_y$ and $\pi^{-1}(x)$ is compact,
there are $y_1, \ldots, y_s$ such that $\pi^{-1}(x) \subseteq  U_{y_1} \cup \cdots \cup U_{y_s}$.
We can choose an open neighborhood $U_x$ of $x$ such that $\pi^{-1}(U_x) \subseteq U_{y_1} \cup \cdots \cup U_{y_s}$.
Thus, if we set $M'_x = \max \{ M'_{y_1}, \ldots, M'_{y_s} \}$, then
\[
f^*(u) \leq M + M'_x (a_1 + \cdots + a_r)
\]
on $\pi^{-1}(U_x)$, and hence the lemma follows.
\end{proof}

Let $\alpha$ be a continuous $(1,1)$-form on $X$. We set
\[
\Tpsh(X; \alpha): =
\left\{ \phi \ \left|\ \begin{array}{cl} \text{(i)} & \text{$\phi : X \to \{-\infty\} \cup \RR$.} \\
\text{(ii)} & \text{$\phi \in (C^{\infty} + \Tpsh)(X)$.} \\ 
\text{(iii)} & \text{$[\alpha] + dd^c([\phi]) \geq 0$.}
\end{array}
\right\}\right..
\]
First we observe the following lemma.

\begin{Lemma}
\label{lem:approximation:plurisub:lower}
We assume that $X$ is compact and that $\alpha + dd^c(\psi_0)$ is either positive or zero for some
$C^{\infty}$-function $\psi_0$ on $X$.
If $\phi \in \Tpsh(X; \alpha) \cap C^0(X)$, then
there are  sequences $\{ \phi_n \}_{n=1}^{\infty}$ and $\{ \varphi_n \}_{n=1}^{\infty}$ in 
\[
\Tpsh(X; \alpha) \cap C^{\infty}(X)
\]
such that $\phi_n \leq \phi \leq \varphi_n$ on $X$ for all $n \geq 1$ and that
\[
\lim_{n\to\infty} \Vert \phi - \phi_n \Vert_{\sup} = \lim_{n\to\infty} \Vert \varphi_n - \phi \Vert_{\sup}= 0.
\]
\end{Lemma}

\begin{proof}
First we assume that $\alpha = dd^c(-\psi_0)$ for some $C^{\infty}$-function $\psi_0$ on $X$.
Then 
\[
\Tpsh(X; \alpha) = \{ \psi_0 + c \mid c \in \RR \cup \{ -\infty \}\}
\]
because $X$ is compact.
Thus the assertion of the lemma is obvious.

Next we assume that $\alpha$ is positive.
By \cite[Theorem~1]{BK},
there is a sequence of $\{ \varphi_n \}_{n=1}^{\infty}$ in $\Tpsh(X; \alpha) \cap C^{\infty}(X)$ such that
\[
\varphi_1(x) \geq \varphi_2(x) \geq \cdots \geq \varphi_n(x) \geq \varphi_{n+1}(x) \geq \cdots \geq \phi(x)
\]
and 
$\phi(x) = \lim_{n\to\infty} \varphi_n(x)$ for all $x \in X$.
Since $X$ is compact and $\phi$ is continuous, it is easy to see that $\lim_{n\to\infty} \Vert \varphi_n - \phi \Vert_{\sup} = 0$.
We set $\phi_n = \varphi_n - \Vert \varphi_n - \phi \Vert_{\sup}$ for all $n \geq 1$.
Then $\phi_n \in \Tpsh(X; \alpha) \cap C^{\infty}(X)$ and $\phi_n \leq \phi$.
Note that $\Vert \phi - \phi_n \Vert_{\sup} \leq 2 \Vert \varphi_n - \phi \Vert_{\sup}$.
Thus $\lim_{n\to\infty} \Vert \phi - \phi_n \Vert_{\sup} = 0$.

Finally we assume that $\alpha' = \alpha + dd^c(\psi_0)$ is positive for some $C^{\infty}$-function $\psi_0$ on $X$.
Then 
\[
\phi' := \phi - \psi_0 \in \Tpsh(X; \alpha') \cap C^0(X).
\]
Thus, by the previous observation,
there are sequences $\{ \phi'_n \}_{n=1}^{\infty}$ and $\{ \varphi'_n \}_{n=1}^{\infty}$in 
\[
\Tpsh(X; \alpha') \cap C^{\infty}(X)
\]
such that $\phi'_n \leq \phi' \leq \varphi'_n$ on $X$ for all $n \geq 1$ and that
\[
\lim_{n\to\infty} \Vert \phi' - \phi'_n \Vert_{\sup} = \lim_{n\to\infty} \Vert \varphi'_n - \phi' \Vert_{\sup} = 0.
\]
We set $\phi_n := \phi'_n + \psi_0$ and $\varphi_n := \varphi'_n + \psi_0$ for every $n \geq 1$.
Then 
\[
\phi_n, \varphi_n  \in \Tpsh(X; \alpha) \cap C^{\infty}(X)\quad\text{and}\quad 
\phi_n \leq \phi \leq \varphi_n
\]
for all $n \geq 1$.
Moreover, $\lim_{n\to\infty}\Vert \phi_n - \phi \Vert_{\sup} = \lim_{n\to\infty}\Vert \varphi_n - \phi \Vert_{\sup} = 0$.
\end{proof}

Let $A$ be an $\RR$-divisor and let $g_A$ be an $A$-Green function of $C^{\infty}$-type on $X$.
Let $\alpha = c_1(A, g_A)$, that is, $\alpha$ is a $C^{\infty}$-form such that
\[
[\alpha] = dd^c([g_A]) + \delta_A
\]
\rom{(}cf. Proposition~\rom{\ref{prop:lelong:formula}}\rom{)}.
\query{add ``$\alpha = c_1(A, g_A)$'' (24/Sep/2010)}
Here let us consider the natural correspondence between $G_{\Tpsh}(X;A)$ and
$\Tpsh(X; \alpha)$ in terms of $g_A$.

\begin{Proposition}
\label{prop:psh:Green:isom}
If $\phi \in \Tpsh(X; \alpha)$, then $\phi + g_A \in G_{\Tpsh}(X;A)$.
Moreover, we have the following:
\begin{enumerate}
\renewcommand{\labelenumi}{(\arabic{enumi})}
\item
For $u \in G_{\Tpsh}(X;A)$, there is $\phi \in \Tpsh(X; \alpha)$ such that $\phi+g_A = u\ \aew$.

\item
For $\phi_1, \phi_2 \in  \Tpsh(X; \alpha)$,
\[
\phi_1 \leq \phi_2\quad\Longleftrightarrow\quad\phi_1+g_A \leq \phi_2 + g_A\ \aew.
\]

\item
For $\phi \in \Tpsh(X; \alpha)$, 
\[
\phi(x) \not= -\infty\ (\forall x \in X)
\quad\Longleftrightarrow\quad
\phi+g_A \in G_{\Tpsh_{\RR}}(X;A).
\]

\item
For $\phi \in \Tpsh(X; \alpha)$,   
\[
\phi \in C^{\infty}(X)\quad\Longleftrightarrow\quad
\phi+g_A \in G_{C^{\infty}}(X;A).
\]

\item
For $\phi \in \Tpsh(X; \alpha)$, 
\[
\phi \in C^{0}(X)\quad\Longleftrightarrow\quad
\phi+g_A \in G_{C^{0}}(X;A).
\]
\end{enumerate}
\end{Proposition}

\begin{proof}
We set $A = a_1 D_1 + \cdots + a_l D_l$, where
$D_i$'s are reduced and irreducible divisors on $X$ and $a_1, \ldots, a_l \in \RR$.
Let $U$ be an open set of $X$ and let $f_1, \ldots, f_l$ be local equations of $D_1, \ldots, D_l$ on $U$
respectively. Let
\[
g_A = h - \sum_{i=1}^l a_i \log \vert f_i \vert^2 \quad\aew
\]
be the local expression of $g_A$ with respect to $f_1, \ldots, f_l$,
where $h \in C^{\infty}(U)$.
Then
\[
g_A + \phi = (h+\phi) - \sum_{i=1}^l a_i \log \vert f_i \vert^2 \quad\aew.
\]
Since $\alpha = dd^c(h)$ on $U$, we have
\[
dd^c([h + \phi]) = [\alpha] + dd^c([\phi]) \geq 0.
\]
Thus $g_A + \phi \in G_{\Tpsh}(X;A)$ and
\[
g_A + \phi = (h+\phi) - \sum_{i=1}^l a_i \log \vert f_i \vert^2 \quad\aew.
\]
is the local expression of $g_A + \phi$ with respect to $f_1, \ldots, f_l$.

(1) For $u \in G_{\Tpsh}(X;A)$, let 
\[
u = p - \sum_{i=1}^l a_i \log \vert f_i \vert^2\ \aew
\]
be the local expression of $u$ with respect to $f_1, \ldots, f_l$, where $p$ is plurisubharmonic.
It is easy to see that $p - h$ does not depend on the choice of the local equations $f_1, \ldots, f_l$.
Thus there is a function $\phi : X \to \{-\infty\} \cup \RR$ such that
$\phi$ is locally given by $p - h$. Moreover
\[
dd^c([p-h]) + [\alpha] = dd^c([p]) \geq 0.
\]
Hence $\phi \in  \Tpsh(X; \alpha)$ and $\phi + g_A = u\ \aew$.

(2) Clearly
\[
\phi_1 \leq \phi_2\ \aew \quad\Longleftrightarrow\quad \phi_1 + g_A \leq \phi_2 + g_A\ \aew.
\]
On the other hand, by Lemma~\ref{lem:fqpssh:ineq:ae},
\[
\phi_1 \leq \phi_2 \quad\Longleftrightarrow\quad\phi_1 \leq \phi_2\ \aew.
\]

(3), (4) and (5) are obvious because
\[
\phi+g_A  = (h + \phi) - \sum_{i=1}^{l} a_i \log \vert f_i \vert^2\ \aew
\]
 is a local expression of $\phi+g_A$ and $h$ is $C^{\infty}$.
\end{proof}

Let $\TT$ be a type for Green functions on $X$ such that 
$\Tpsh$ is a subjacent type of $\TT$, that is,
the following property holds for an arbitrary open set $U$ of $X$:
if $u \leq v\ \aew$ on $U$ for $u \in \Tpsh(U)$ and $v \in \TT(U)$, then $u \leq v$ on $U$.

\begin{Proposition}
\label{prop:limit:Green:function:PSH}
Let $A$ and $B$ be $\RR$-divisors on $X$ with $A \leq B$.
Let $h$ be a $B$-Green function of $\TT$-type on $X$ such that $h$ is of upper bounded type.
Let $\{ g_{\lambda} \}_{\lambda \in \Lambda}$ be a family of $A$-Green functions of $\Tpsh$-type on $X$.
We assume that 
$g_{\lambda} \leq h\ \aew$ for all $\lambda \in \Lambda$.
Then there is an $A$-Green function $g$ of $\Tpsh$-type on $X$ with the following properties:
\begin{enumerate}
\renewcommand{\labelenumi}{(\alph{enumi})}
\item
Let us fix an $A$-Green function $g_A$ of $C^{\infty}$-type. Let $\alpha$ be a unique $C^{\infty}$-form with
$[\alpha] = dd^c([g_A]) + \delta_A$. If we choose $\phi \in \Tpsh(X;\alpha)$ and $\phi_{\lambda} \in \Tpsh(X;\alpha)$ for each $\lambda \in \Lambda$
such that $g = g_A + \phi \ \aew$ and $g_{\lambda} = g_A + \phi_{\lambda} \ \aew$ \rom{(}cf. Proposition~\rom{\ref{prop:psh:Green:isom}}\rom{)},
then $\phi$ is the upper semicontinuous regularization of
the function given by
\[
x \mapsto \sup_{\lambda \in \Lambda} \{ \phi_{\lambda}(x) \}.
\]
In particular, $g_{\rm can}$ is  the upper semicontinuous regularization of
the function given by
\[
x \mapsto \sup_{\lambda \in \Lambda} \{ (g_{\lambda})_{\rm can}(x) \}
\]
over $X \setminus \Supp(A)$.

\item $g \leq h\ \aew$.

\item
If there is $g_{\lambda}$
such that $g_{\lambda} \in G_{\Tpsh_{\RR}}(X;A)$, then $g \in G_{\Tpsh_{\RR}}(X;A)$.
\end{enumerate}
\end{Proposition}

\begin{proof}
Let $A = a_1 D_1 + \cdots + a_l D_l$ and $B = b_1 D_1 + \cdots + b_l D_l$ be the decompositions of $A$ and $B$ such that
$D_i$'s are reduced and irreducible divisors, $a_1, \ldots, a_l, b_1, \ldots, b_l \in \RR$ and
$D_1 \cup \cdots \cup D_l = \Supp(A) \cup \Supp(B)$.
Let $U$ be an open set of $X$ and let $f_1, \ldots, f_l$ be local equations of $D_1, \ldots, D_l$ over $U$ respectively.
Let 
\[
h = v + \sum_{i=1}^l (-b_i) \log \vert f_i \vert^2\quad\aew
\]
be the local expression of $h$ with respect to $f_1, \ldots, f_l$.
Moreover, let 
\[
g_{\lambda} = u_{\lambda} + \sum_{i=1}^l (-a_i) \log \vert f_i \vert^2\quad\aew
\]
be the local expression of $g_{\lambda}$ with respect to $f_1, \ldots, f_l$.
Then 
\[
u_{\lambda} \leq v - \sum_{i=1}^l (b_i - a_i) \log \vert f_i \vert^2\quad\aew
\]
holds for every $\lambda \in \Lambda$. Note that $v$ is locally bounded above.
Thus $\{ u_{\lambda} \}_{\lambda \in \Lambda}$ is uniformly locally bounded above by Lemma~\ref{lem:subharmonic:local:bound}.
Let $u$ be the function on $U$ given by 
\[
u(x) = \sup \{ u_{\lambda}(x) \mid \lambda \in \Lambda  \}.
\]
Let $\tilde{u}$ be the upper semicontinuous regularization of $u$.
Then $\tilde{u}$ is plurisubharmonic on $U$ (cf. Subsection~\ref{subsec:pluri:subharmonic}). 
Let $f'_1, \ldots, f'_l$ be another local equations of $D_1, \ldots, D_l$.
Then there are $e_1, \ldots, e_n \in \OO_{U}^{\times}(U)$
such that $f'_i = e_i f_i$ for all $i$, so that
\[
g_{\lambda} = \left(u_{\lambda} + \sum_{i=1} a_i \log \vert e_i \vert^2\right) + \sum_{i=1}^l (-a_i) \log \vert f'_i \vert^2\quad\aew
\]
is the local expression of $g_{\lambda}$ with respect to $f'_1, \ldots, f'_l$.
Thus, if we denote the plurisubharmonic function arising from $f'_1, \ldots, f'_l$ by $\tilde{u}'$,
then, by Lemma~\ref{lem:fqpssh:ineq:ae},
\[
\tilde{u}' = \tilde{u} + \sum_{i=1}^l a_i \log \vert e_i \vert^2.
\]
This means that
\[
\tilde{u} + \sum_{i=1}^l (-a_i) \log \vert f_i \vert^2
\]
does not depend on the choice of $f_1, \ldots, f_l$ over $U \setminus \Supp(A)$.
Thus there is $g \in G_{\Tpsh}(X;A)$ such that
\[
\rest{g}{U} = \tilde{u} + \sum_{i=1}^l (-a_i) \log \vert f_i \vert^2\quad\aew.
\]

Let
$g_{A} = u_{A} + \sum_{i=1}^l (-a_i) \log \vert f_i \vert^2\ \aew$
be the local expression of $g_{A}$ with respect to $f_1, \ldots, f_l$.
Then $\phi_{\lambda} = u_{\lambda} - u_A$ and $\phi = \tilde{u} - u_A$.
Thus (a) follows.

By (a), $g_{\rm can}$ is the upper semicontinuous regularization of
the function $g'$ given by $g'(x) = \sup_{\lambda \in \Lambda} \{ (g_{\lambda})_{\rm can}(x) \}$ over $X \setminus \Supp(A)$.
As $\Tpsh$ is a subjacent type of $\TT$, we have 
$(g_{\lambda})_{\rm can} \leq h_{\rm can}$ on $X \setminus (\Supp(A) \cup \Supp(B))$ for all $\lambda \in \Lambda$.
Note that $g = g' \ \aew$ (cf. Subsection~\ref{subsec:pluri:subharmonic}).  
Thus we have $g \leq h\ \aew$.

Finally we assume that $g_{\lambda} \in G_{\Tpsh_{\RR}}(X;A)$ for some $\lambda$.
Then $u_{\lambda} \leq \tilde{u}\ \aew$, so that $u_{\lambda} \leq \tilde{u}$ by Lemma~\ref{lem:fqpssh:ineq:ae}.
Thus $\tilde{u}(x) \not= -\infty$.
Therefore, $g \in G_{\Tpsh_{\RR}}(X;A)$.
\end{proof}

Let $A$ be an $\RR$-divisor on $X$ and let $g$ be
a locally integrable function on $X$. We set
\[
G_{\TT}(X;A)_{\leq g} :=\{ u \in G_{\TT}(X;A) \mid u \leq g\ \aew \},
\]
where $G_{\TT}(X;A)$ is the set of all $A$-Green functions of $\TT$-type on $X$.

\begin{Lemma}
\label{lem:Green:A:B:}
Let $A$ and $B$ be
$\RR$-divisors on $X$ with
$A \leq B$.
Let $g_B$ be a $B$-Green function of $C^{\infty}$-type \rom{(}resp. $C^{0}$-type\rom{)}.
There is an $A$-Green function $g_A$ of $C^{\infty}$-type \rom{(}resp. $C^{0}$-type\rom{)} such that
\[
g_A \leq g_B\ \aew\quad\text{and}\quad
G_{\Tpsh}(X;A)_{\leq g_A} = G_{\Tpsh}(X;A)_{\leq g_B}.
\]
\end{Lemma}
\query{
$\Tpsh_{\RR}$ $\Longrightarrow$ $\Tpsh$ in Lemma~\ref{lem:Green:A:B:}
(21/September/2010)
}
\begin{proof}
We set $A = a_1 D_1 + \cdots + a_n D_n$ and $B = b_1 D_1 + \cdots + b_n D_n$,
where
$D_i$'s are reduced and irreducible divisors on $X$ and  
$a_1, \ldots, a_n, b_1, \ldots, b_n \in \RR$.
For $x \in X$,
let $U_x$ be a small open neighborhood of $x$ and let
$f_1, \ldots, f_n$ be local equations of $D_1, \ldots, D_n$ on $U_x$ respectively.
Note that if $x \not\in D_i$, then we take $f_i$ as the constant function $1$.
Let $g_B = h_x - \sum_i b_i \log \vert f_i \vert^2\ \aew$ be the local expression of $g_B$ on $U_x$ with respect to $f_1, \ldots, f_n$.
Shrinking $U_x$ if necessarily, we may assume that there is a constant $M_x$ such that
$\vert h_x \vert \leq M_x$ on $U_x$.

\begin{Claim}
\label{claim:lem:Green:A:B:1}
There are an open neighborhood $V_x$ of $x$ and a positive constant $C_x$ such that $V_x \subseteq U_x$,
\[
h_x + C_x - \sum_i a_i \log \vert f_i \vert^2 \leq g_B\quad\aew
\]
on $V_x$ and that
\[
u \leq h_x + C_x - \sum_i a_i \log \vert f_i \vert^2\quad\aew
\]
on $V_x$ for all $u \in G_{\Tpsh}(X;A)_{\leq g_B}$.
\end{Claim}

\begin{proof}
For $u \in  G_{\Tpsh}(X;A)_{\leq g_B}$, let $u = p_x(u) - \sum_i a_i \log \vert f_i \vert^2\ \aew$ be the local expression of $u$ on $U_x$ with respect to $f_1, \ldots, f_n$.
Then $u \leq g_B\ \aew$ is nothing more than 
\[
p_x(u) \leq h_x - \sum_i (b_i - a_i) \log \vert f_i \vert^2\quad\aew.
\]
If either $a_i = b_i$ or $x \not\in D_i$ for all $i$,
then $\sum_i (b_i - a_i) \log \vert f_i \vert^2 = 0$ on $U_x$.
Thus our assertion is obvious by taking $C_x = 0$, so that we may assume that 
$a_i < b_i$  and $x \in D_i$ for some $i$.
By Lemma~\ref{lem:subharmonic:local:bound},
there are an open neighborhood $U'_x$ of $x$ and
a positive constant $M'_x$ such that $U'_x \subseteq U_x$ and
$p_x(u) \leq M'_x$ on $U'_x$ for all $u \in G_{\Tpsh}(X;A)_{\leq g_B}$.
Note that 
\[
M_x' = -M_x + (M'_x + M_x) \leq h_x + (M'_x + M_x)
\]
on $U_x$.
Thus if we set $C_x = M'_x + M_x$, then
$p_x(u) \leq h_x + C_x$ on $U'_x$ for all $u \in G_{\Tpsh}(X;A)_{\leq g_B}$.
As $\lim_{y \to x} \sum_{i} (b_i - a_i) \log \vert f_i \vert^2(y) = -\infty$, 
we can find an open neighborhood $V_x$ of $x$ such that
$V_x \subseteq U'_x$ and $C_x \leq -\sum_i (b_i - a_i) \log \vert f_i \vert^2$ on $V_x$.
Therefore,
\[
p_x(u) \leq h_x + C_x \leq h_x - \sum_i (b_i - a_i) \log \vert f_i \vert^2
\]
on $V_x$ for all $u \in G_{\Tpsh}(X;A)_{\leq g_B}$, as required. 
\end{proof}

By using Claim~\ref{claim:lem:Green:A:B:1}, we can find an open covering $\{ V_{\lambda} \}_{\lambda \in \Lambda}$ of $X$ and
a family of constants $\{ C_{\lambda} \}_{\lambda \in \Lambda}$ with the following properties:
\begin{enumerate}
\renewcommand{\labelenumi}{(\arabic{enumi})}
\item
$\{ V_{\lambda} \}_{\lambda \in \Lambda}$ is a locally finite covering.

\item
There are local equations $f_{\lambda,1}, \ldots, f_{\lambda,n}$ of
$D_1, \ldots, D_n$ on $V_{\lambda}$ respectively.

\item
Let $g_B = h_{\lambda} - \sum_i b_{i} \log \vert f_{\lambda,i} \vert^2\ \aew$ be the local expression of $g_B$ on $V_{\lambda}$
with respect to$f_{\lambda,1}, \ldots, f_{\lambda,n}$.
Then 
\[
h_{\lambda} + C_{\lambda} - \sum a_{i} \log \vert f_{\lambda,i} \vert^2 \leq g_B\quad\aew
\]
on $V_{\lambda}$ and that
\[
u \leq h_{\lambda} + C_{\lambda} - \sum_i a_{i} \log \vert f_{\lambda,i} \vert^2\quad\aew
\]
on $V_{\lambda}$ for all $u \in G_{\Tpsh}(X;A)_{\leq  g_B}$.
\end{enumerate}
Let $\{ \rho_{\lambda} \}_{\lambda \in \Lambda}$ be a partition of unity subordinate to the covering $\{ V_{\lambda} \}_{\lambda \in \Lambda}$.
We set 
\[
g_A = \sum_{\lambda} \rho_{\lambda} \left( h_{\lambda} + C_{\lambda} - \sum_i a_{i} \log \vert f_{\lambda,i} \vert^2\right).
\]
By Lemma~\ref{lem:green:partition:unity}, $g_A$ is an $A$-Green function of $C^{\infty}$-type (resp. $C^{0}$-type).
Moreover, 
$g_A \leq g_B\ \aew$ and $u \leq g_A\ \aew$ for all $u \in G_{\Tpsh}(X;A)_{\leq g_B}$.
Therefore the lemma follows.
\end{proof}

The following theorem is the main result of this section.

\begin{Theorem}
\label{thm:cont:upper:envelope}
Let $A$  be an $\RR$-divisor on $X$.
If $X$ is projective and there is an $A$-Green function $h$ of $C^{\infty}$-type
such that $dd^c([h]) + \delta_A$ is represented by either a positive $C^{\infty}$-form or the zero form,
then we have the following:
\begin{enumerate}
\renewcommand{\labelenumi}{(\arabic{enumi})}
\item
Let $B$ be an
$\RR$-divisor on $X$ with
$A \leq B$.
Let $g_B$ be a $B$-Green function of $C^{0}$-type.
Then there is $g \in G_{C^0 \cap \Tpsh}(X;A)$ such that
$g \leq g_B\ \aew$ and 
\[
u \leq g\ \aew \quad (\forall u \in G_{\Tpsh}(X;A)_{\leq g_B}).
\]

\item
If $u \in G_{C^0 \cap \Tpsh}(X;A)$,
then there are sequences $\{ u_n \}_{n=1}^{\infty}$ and $\{ v_n \}_{n=1}^{\infty}$ of continuous functions on $X$ with the following properties:
\begin{enumerate}
\renewcommand{\labelenumii}{(\arabic{enumi}.\arabic{enumii})}
\item $u_n \geq 0$ and $v_n \geq 0$ for all $n \geq 1$.

\item $\lim_{n\to\infty} \Vert u_n \Vert_{\sup} = \lim_{n\to\infty} \Vert v_n \Vert_{\sup} = 0$.

\item
$u - u_n, u + v_n  \in G_{C^{\infty} \cap \Tpsh}(X;A)$ all $n \geq 1$.
\end{enumerate}
\end{enumerate}
\end{Theorem}
\query{
$\Tpsh_{\RR}$ $\Longrightarrow$ $\Tpsh$ in Theorem~\ref{thm:cont:upper:envelope} except
$\Tpsh_{\RR}$ in Claim~\ref{claim:thm:cont:upper:envelope:1}
(21/September/2010)
}

\begin{proof}
(1) Let us begin with the following claim:

\begin{Claim}
\label{claim:thm:cont:upper:envelope:1}
There is $g \in G_{\Tpsh_{\RR}}(X;A)$ such that
$g \leq g_B\ \aew$ and 
\[
u \leq g\quad \aew \quad (\forall u \in G_{\Tpsh}(X;A)_{\leq g_B}).
\]
We say $g$ is the greatest element of $G_{\Tpsh}(X;A)_{\leq g_B}$ modulo null functions.
\end{Claim}

\begin{proof}
Note that $\Tpsh$ is a subjacent type of $C^0$ by Lemma~\ref{lem:fqpssh:ineq:ae}, and
that $h - c \in G_{\Tpsh_{\RR}}(X;A)_{\leq g_B}$ for some constant $c$. Thus the assertion follows from
Proposition~\ref{prop:limit:Green:function:PSH}.
\end{proof}

\begin{Claim}
If $g_B$ is of $C^{\infty}$-type, then the assertion of \rom{(1)} holds.
\end{Claim}

\begin{proof}
By Lemma~\ref{lem:Green:A:B:}, 
we may assume that $A = B$.
Let $\alpha = c_1(A,g_A)$, that is, $\alpha$ is a $C^{\infty}$-form such that
$[\alpha] = dd^c([g_A]) + \delta_A$.
\query{add ``$\alpha = c_1(A,g_A)$'' (24/Sep/2010)}
We set
\[
\Tpsh(X; \alpha)_{\leq 0} = \{ \psi \in \Tpsh(X; \alpha) \mid \psi \leq 0 \}.
\]
By our assumption, we can find a $C^{\infty}$-function $\psi_0$ such that $g_A + \psi_0 = h\ \aew$.
Note that $[\alpha + dd^c(\psi_0)] = dd^c([h]) + \delta_A$. Thus
$\alpha + dd^c(\psi_0)$ is either positive or zero.

\smallskip
First we assume that $\alpha + dd^c(\psi_0)$ is positive.
Let $g$ be the greatest element of 
\[
G_{\Tpsh}(X;A)_{\leq g_A}
\]
modulo null functions (cf. Claim~\ref{claim:thm:cont:upper:envelope:1}).
We choose $\phi \in \Tpsh(X;\alpha)$ and $\psi_u \in \Tpsh(X;\alpha)$ for each $u \in G_{\Tpsh}(X;A)_{\leq g_A}$ such that
$g = g_A + \phi \ \aew$ and $u = g_A + \psi_u \ \aew$ (cf. Proposition~\ref{prop:psh:Green:isom}).
Then 
\[
\{ \psi_u \mid u \in G_{\Tpsh}(X;A)_{\leq g_A} \} = \Tpsh(X; \alpha)_{\leq 0}.
\]
Moreover, by our construction of $g$ (cf. Proposition~\ref{prop:limit:Green:function:PSH} and Claim~\ref{claim:thm:cont:upper:envelope:1}),
$\phi$ is the upper semicontinuous regularization of the function $\phi'$ given by 
\[
\phi'(x) = \sup \{ \psi_u(x) \mid u \in G_{\Tpsh}(X;A)_{\leq g_A}  \} (= \sup \{ \psi(x) \mid \psi \in  \Tpsh(X; \alpha)_{\leq 0} \})
\]
for $x \in X$. On the other hand, by \cite[Theorem~1.4]{BD},
$\phi'$ is continuous. Thus $\phi = \phi'$ and $\phi$ is continuous.
Therefore the claim follows by Proposition~\ref{prop:psh:Green:isom}.

\smallskip
Next we assume that $\alpha + dd^c(\psi_0) = 0$,
that is,
$\alpha = dd^c(-\psi_0)$. 
Then
\[
\Tpsh(X;\alpha) = \left\{ \psi_0 + c \mid c \in \RR \cup \{-\infty\} \right\}.
\]
Let $g$ be the greatest element of $G_{\Tpsh}(X;A)_{\leq g_A}$ modulo null functions.
Then, by Proposition~\ref{prop:psh:Green:isom}, there is $c \in \RR$ such that $g = g_A + (\psi_0 + c) \ \aew$.
Thus the claim follows in this case.
\end{proof}

Finally, let us consider a general case.
First of all, we may assume $A = B$ as before.
We can take a continuous function $f$ on $X$ such that $g_A = h + f\ \aew$.
By using the Stone-Weierstrass theorem,
there is a sequence $\{ u_n \}_{n=1}^{\infty}$ of continuous functions on $X$ such that $\lim_{n\to\infty} \Vert u_n \Vert_{\sup} = 0$ and
$f + u_n$ is $C^{\infty}$ for every $n$. Then, as 
$g_A + u_n = h + (f + u_n)\ \aew$, 
$g_A + u_n$ is of $C^{\infty}$-type for all all $n$.
Let $g$ (resp. $g_n$) be the greatest element of $G_{\Tpsh}(X;A)_{\leq g_A}$ (resp. $G_{\Tpsh}(X;A)_{\leq g_A+u_n}$)
modulo null functions.
Note that the greatest element of $G_{\Tpsh}(X;A)_{\leq g_A\pm\Vert u_n \Vert_{\sup}}$ modulo null functions
is given by
$g \pm \Vert u_n \Vert_{\sup}$.
By the previous claim, $g_n \in G_{C^0 \cap \Tpsh}(X;A)$.
Moreover, since 
\[
g_A - \Vert u_n \Vert_{\sup} \leq g_A + u_n \leq g_A + \Vert u_n \Vert_{\sup}\quad \aew,
\]
we have 
\[
g -  \Vert u_n \Vert_{\sup} \leq g_n \leq g + \Vert u_n \Vert_{\sup}\ \aew
\]
for all $n$.
Let $g = v + \sum_{i=1}^l (-a_i) \log \vert f_i \vert^2\ \aew$ and $g_n = v_n + \sum_{i=1}^l (-a_i) \vert f_i \vert^2\ \aew$ be
local expression of $g$ and $g_n$. Note that $v_n$  is continuous for every $n$.
By Lemma~\ref{lem:fqpssh:ineq:ae},
$v -  \Vert u_n \Vert_{\sup} \leq v_n \leq v + \Vert u_n \Vert_{\sup}$ holds for all $n$.
Thus $v_n$ converges to $v$ uniformly, which implies that $v$ is continuous.

\medskip
(2) 
Let $\alpha'$ be a $C^{\infty}$-form such that $[\alpha'] = dd^c([h]) + \delta_{A}$.
By our assumption, $\alpha'$ is either positive or zero.
By Proposition~\ref{prop:psh:Green:isom}, there is $\psi \in \Tpsh(X; \alpha')$ such that
$\psi$ is continuous and
$\psi + h = u\ \aew$.
Thus, by Lemma~\ref{lem:approximation:plurisub:lower},
there are sequences $\{ u_n \}_{n=1}^{\infty}$ and $\{ v_n \}_{n=1}^{\infty}$ of continuous functions on $X$ with the following properties:
\begin{enumerate}
\renewcommand{\labelenumi}{(\alph{enumi})}
\item $u_n \geq 0$ and $v_n \geq 0$ for all $n \geq 1$.

\item $\lim_{n\to\infty} \Vert u_n \Vert_{\sup} = \lim_{n\to\infty} \Vert v_n \Vert_{\sup} = 0$.

\item
$\psi - u_n, \psi + v_n \in \Tpsh(X; \alpha') \cap C^{\infty}(X)$ for every $n \geq 1$.
\end{enumerate}
Note that $u - u_n = (\psi - u_n) + h\ \aew$ and $u + v_n = (\psi + v_n) + h\ \aew$. Therefore, by Proposition~\ref{prop:psh:Green:isom},
$u-u_n, u + v_n \in G_{C^{\infty} \cap \Tpsh}(X;A)$.
\end{proof}

\renewcommand{\theTheorem}{\arabic{section}.\arabic{subsection}.\arabic{Theorem}}
\renewcommand{\theClaim}{\arabic{section}.\arabic{subsection}.\arabic{Theorem}.\arabic{Claim}}
\renewcommand{\theequation}{\arabic{section}.\arabic{subsection}.\arabic{Theorem}.\arabic{Claim}}

\section{Arithmetic $\RR$-divisors}
\label{sec:arithmetic:R:divisor}
Throughout this section, let $X$ be a $d$-dimensional generically smooth and normal arithmetic variety, that is,
$X$ is a flat and quasi-projective integral scheme over $\ZZ$ such that
$X$ is normal, $X$ is smooth over $\QQ$ and the Krull dimension of $X$ is $d$.

\subsection{Definition of arithmetic $\RR$-divisor}
\label{subsec:def:arithmetic:R:divisor}
Let $\Div(X)$ 
be the group of Cartier divisors  
on $X$.
An element of 
\[
\Div(X)_{\RR} := \Div(X) \otimes_{\ZZ} \RR\quad(\text{resp. } \Div(X)_{\QQ} := \Div(X) \otimes_{\ZZ} \QQ)
\]
is called an {\em $\RR$-divisor} (resp. {\em $\QQ$-divisor}) on $X$.
Let $D$ be an $\RR$-divisor on $X$ and 
let $D = a_1 D_1 + \cdots + a_l D_l$ be the unique decomposition of $D$ such that
$D_i$'s are prime divisors on $X$ and $a_1, \ldots, a_l \in \RR$.
Note that $D_i$'s are not necessarily Cartier divisors on $X$.
The support $\Supp(D)$ of $D$ is defined by $\bigcup_{i \in \{ i \mid a_i \not= 0\}} D_i$.
If $a_i \geq 0$ for all $i$, then $D$ is said to be {\em effective} and it is denoted by $D \geq 0$.
More generally, for $D, E \in \Div(X)_{\RR}$, if $D - E \geq 0$, then it is denoted by $D \geq E$ or $E \leq D$.
We define $H^0(X, D)$ to be
\[
H^0(X, D) = \{ \phi \in \Rat(X)^{\times} \mid (\phi) + D \geq 0 \} \cup \{ 0 \},
\]
where $\Rat(X)$ is the field of rational functions on $X$.
Let $F_{\infty} : X(\CC) \to X(\CC)$ be the complex conjugation map on $X(\CC)$.
Let $g$ be a locally integrable function on $X(\CC)$.
We say $g$ is {\em $F_{\infty}$-invariant} if $F_{\infty}^*(g) = g\ \aew$ on $X(\CC)$.
Note that we do not require that $F_{\infty}^*(g)$ is identically equal to  $g$ on $X(\CC)$.
A pair $\overline{D} = (D, g)$ is called an {\em arithmetic $\RR$-divisor on $X$} if
$g$ is $F_{\infty}$-invariant.
If $D \in \Div(X)$ (resp. $D \in \Div(X)_{\QQ}$),
then $\overline{D}$ is called an {\rm arithmetic divisor on $X$} (resp. {\em arithmetic $\QQ$-divisor on $X$}).
For arithmetic $\RR$-divisors $\overline{D}_1 = (D_1, g_1)$ and $\overline{D}_2 = (D_2, g_2)$,
$\overline{D}_1 = \overline{D}_2$ and $\overline{D}_1 \leq \overline{D}_2$ (or $\overline{D}_2 \geq \overline{D}_1$)
are defined as follows:
\[
\begin{cases}
\overline{D}_1 = \overline{D}_2\quad\overset{\text{def}}{\Longleftrightarrow}\quad\text{$D_1 = D_2$ and $g_1 = g_2\ \aew$},\\
\overline{D}_1 \leq \overline{D}_2\quad\overset{\text{def}}{\Longleftrightarrow}\quad\text{$D_1 \leq D_2$ and $g_1 \leq g_2\ \aew$}.
\end{cases}
\]
If $\overline{D} \geq (0,0)$, then $\overline{D}$ is said to be {\em arithmetically effective} (or {\em effective} for simplicity).
For arithmetic $\RR$-divisors $\overline{D}$ and $\overline{E}$ on $X$, we set $(-\infty, \overline{D}]$, $[\overline{D}, \infty)$ and
$[\overline{D},\overline{E}]$ as follows:
\[
\begin{cases}
(-\infty, \overline{D}] := \{ \overline{M} \mid \text{$\overline{M}$ is an arithmetic $\RR$-divisor on $X$ and $\overline{M} \leq \overline{D}$} \}, \\
[\overline{D}, \infty)  := \{ \overline{M} \mid \text{$\overline{M}$ is an arithmetic $\RR$-divisor on $X$ and $\overline{D} \leq \overline{M}$} \}, \\
[\overline{D}, \overline{E}]  := \{ \overline{M} \mid \text{$\overline{M}$ is an arithmetic $\RR$-divisor on $X$ and $\overline{D} \leq \overline{M} \leq \overline{E}$} \}.
\end{cases}
\]

Let $\TT$ be a type for Green functions on $X$, that is,
$\TT$ 
\query{add space after ``$\TT$''
(18/May/2010)}
is a type for Green functions on $X(\CC)$ together with the following extra $F_{\infty}$-compatibility condition:
if $u\in \TT(U)$ for an open set $U$ of $X(\CC)$, then $F_{\infty}^*(u) \in \TT(F_{\infty}^{-1}(U))$.
On arithmetic varieties, we always assume the above $F_{\infty}$-compatibility condition for a type for Green functions.
We denote 
\[
\{ u \in \TT(X(\CC)) \mid u = F_{\infty}^*(u) \}
\]
by $\TT(X)$. Note that $\TT(X)$ is different from $\TT(X(\CC))$.
Clearly $C^{0}$ and $C^{\infty}$ have $F_{\infty}$-compatibility. Moreover, by the following lemma,
$\Tpsh$ and $\Tpsh_{\RR}$ have also $F_{\infty}$-compatibility.
If two types $\TT$ and $\TT'$  for Green functions have $F_{\infty}$-compatibility, then
$\TT + \TT'$ and $\TT - \TT'$ have also $F_{\infty}$-compatibility.

\begin{Lemma}
\label{lem:plurisubharmonic:complex:conjugation}
Let $f_1, \ldots, f_r \in \RR[X_1, \ldots, X_N]$ and 
\[
V = \Spec(\CC[X_1, \ldots, X_N]/(f_1, \ldots, f_r)).
\]
We assume that $V$ is $e$-equidimensional and smooth over $\CC$.
Let $F_{\infty} : V \to V$ be the complex conjugation map.
If $u$ is a plurisubharmonic function on an open set $U$ of $V$,
then $F_{\infty}^*(u)$ is also a plurisubharmonic function on $F_{\infty}^{-1}(U)$.
\end{Lemma}

\begin{proof}
Fix $x \in U$ and choose $i_1 < \cdots < i_e$ such that the projection $p : V \to \CC^e$ given by
$(x_1, \ldots, x_N) \mapsto (x_{i_1}, \ldots, x_{i_e})$ is \'{e}tale at $x$.
Note that the following diagram is commutative:
\[
\begin{CD}
V @>{F_{\infty}}>> V \\
@V{p}VV @VV{p}V \\
\CC^e @>{F_{\infty}}>> \CC^e
\end{CD}
\]
Let $U_x$ be an open neighborhood of $x$ such that $\rest{p}{U_x} : U_x \to W_x = p(U_x)$
is an isomorphism as complex manifolds.
Then $\rest{p}{F_{\infty}^{-1}(U_x)} : F_{\infty}^{-1}(U_x) \to F_{\infty}^{-1}(W_x)$ is also an isomorphism
as complex manifolds. This observation indicates that we may assume $V = \CC^e$ in order to see our assertion.

Let $y \in F_{\infty}^{-1}(U) \subseteq \CC^e$ and $\xi \in \CC^e$ such that $y + \xi \exp({\sqrt{-1}\theta}) \in F_{\infty}^{-1}(U)$ for
all $0 \leq \theta \leq 2\pi$. Then
\begin{align*}
F_{\infty}^*(u)(y) & = u(\bar{y}) \leq \frac{1}{2\pi} \int_{0}^{2\pi} u(\bar{y} + \bar{\xi} \exp({\sqrt{-1}\theta})) d\theta \\
& =
\frac{1}{2\pi} \int_{0}^{2\pi} u(\bar{y} + \bar{\xi} \exp({-\sqrt{-1}\theta})) d\theta
\\ & =
 \frac{1}{2\pi} \int_{0}^{2\pi} u\left(\overline{y + \xi \exp({\sqrt{-1}\theta}})\right) d\theta \\
 & =
  \frac{1}{2\pi} \int_{0}^{2\pi} F_{\infty}^*(u)\left(y + \xi \exp({\sqrt{-1}\theta})\right) d\theta,
\end{align*}
which shows that $F_{\infty}^*(u)$ is plurisubharmonic on $F_{\infty}^{-1}(U)$.
\end{proof}

Let $D$ be an $\RR$-divisor on $X$ and let $g$ be a  $D$-Green function on $X(\CC)$.
By the following lemma,
$\frac{1}{2}( g + F_{\infty}^*(g))$ is an $F_{\infty}$-invariant $D$-Green function of $\TT$-type on $X(\CC)$.

\begin{Lemma}
\label{lem:D:Green:F:infty}
If $g$ is a $D$-Green function of $\TT$-type, then
$F_{\infty}^*(g)$ is also a $D$-Green function of $\TT$-type.
\end{Lemma}

\begin{proof}
Let $D = a_1 D_1 + \cdots + a_l D_l$ be a decomposition of $D$ such that $a_1, \ldots, a_l \in \RR$ and
$D_i$'s are Cartier divisors on $X$.
Let $U$ be a Zariski open set of $X$ over which $D_i$ can be written by a local equation $\phi_i$ for each $i$.
Let $g = u + \sum_{i=1}^l (-a_i) \log \vert \phi_i \vert^2\ \aew$ be the local expression of $g$ with respect to $\phi_1, \ldots, \phi_l$
over $U(\CC)$.
Note that $F_{\infty}^*(\phi_i) = \bar{\phi}_i$ as a function over $U(\CC)$.
Thus $F_{\infty}^*(g) = F_{\infty}^*(u) + \sum_{i=1}^l (-a_i) \log \vert \phi_i \vert^2\ \aew$ is a local expression of $F_{\infty}^*(g)$,
as required.
\end{proof}

We define $\aDiv_{\TT}(X)$, $\aDiv_{\TT}(X)_{\QQ}$ and $\aDiv_{\TT}(X)_{\RR}$ as follows:
\[
\begin{cases}
\aDiv_{\TT}(X) := \left\{ (D, g) \ \left|\  \begin{array}{l} \text{$D \in \Div(X)$ and $g$ is an $F_{\infty}$-invariant} \\
\text{$D$-Green function of $\TT$-type on $X(\CC)$.}
\end{array} \right\}\right., \\
 \\
 \aDiv_{\TT}(X)_{\QQ} := \left\{ (D, g) \ \left|\  \begin{array}{l} \text{$D \in \Div(X)_{\QQ}$ and $g$ is an $F_{\infty}$-invariant} \\
\text{$D$-Green function of $\TT$-type on $X(\CC)$.}
\end{array} \right\}\right., \\
 \\
\aDiv_{\TT}(X)_{\RR} := \left\{ (D, g) \ \left|\  \begin{array}{l} \text{$D \in \Div(X)_{\RR}$ and $g$ is an $F_{\infty}$-invariant}\\
\text{$D$-Green function of $\TT$-type on $X(\CC)$.}
\end{array} \right\}\right.. \\
\end{cases}
\]
An element of $\aDiv_{\TT}(X)_{\RR}$ (resp. $\aDiv_{\TT}(X)_{\QQ}$, $\aDiv_{\TT}(X)$) is called an {\em arithmetic $\RR$-divisor of $\TT$-type on $X$}
(resp. {\em arithmetic $\QQ$-divisor of $\TT$-type on $X$}, {\em arithmetic divisor of $\TT$-type on $X$}).
Let $\overline{D} = (D, g)$ be an arithmetic $\RR$-divisor of $\TT$-type.
Then, as $F_{\infty}^*(g) = g \ \aew$, we can see that $F_{\infty}^*(g_{\rm can}) = g_{\rm can}$ holds $X(\CC) \setminus \Supp(D)(\CC)$.

Here we recall $\aPic_{C^0}(X)$, $\aPic_{C^0}(X)_{\QQ}$ and $\aPic_{C^0}(X)_{\RR}$
(for details, see \cite{MoContExt}). First of all,
let $\aPic_{C^0}(X)$ be the group of isomorphism classes of $F_{\infty}$-invariant continuous hermitian invertible sheaves on $X$ and let
$\aPic_{C^0}(X)_{\QQ} := \aPic_{C^0}(X) \otimes_{\ZZ} \QQ$.
\query{$\otimes \QQ$\quad
$\Longrightarrow$\quad $\otimes_{\ZZ} \QQ$
(17/October/2010)}
For an $F_{\infty}$-invariant  continuous function $f$ on $X(\CC)$, $\overline{\OO}(f)$ is given by $(\OO_X, \exp(-f)\vert \cdot\vert_{\rm can})$.
Then $\aPic_{C^0}(X)_{\RR}$ is defined to be
\[
\aPic_{C^0}(X)_{\RR} 
:= \frac{\aPic_{C^0}(X) \otimes_{\ZZ} \RR}{
\left\{ \sum_i \overline{\OO}(f_i) \otimes a_i  \left| 
\begin{array}{l}
\text{$f_1, \ldots, f_r \in C^0(X)$ and} \\
\text{$a_1, \ldots, a_r \in \RR$ with $\sum_i a_i f_i  = 0$ }
\end{array} \hspace{-0.5em}\right\}\right.},
\]
\query{$\otimes \RR$\quad
$\Longrightarrow$\quad $\otimes_{\ZZ} \RR$
(17/October/2010)}
where $C^0(X) = \{ f \in C^0(X(\CC)) \mid F_{\infty}^*(f) = f \}$ as before.
Note that there is a natural surjective homomorphism $\overline{\OO} : \aDiv_{C^0}(X) \to \aPic_{C^0}(X)$
given by 
\[
\overline{\OO}(D,g) = (\OO_X(D), \vert\cdot\vert_g),
\]
where $\vert 1 \vert_g = \exp(-g/2)$.

\subsection{Volume function for arithmetic $\RR$-divisors}
\setcounter{Theorem}{0}
We assume that $X$ is projective.
Let $\overline{D} = (D, g)$ be an arithmetic $\RR$-divisor on $X$.
We set
\[
\aH(X, \overline{D}) = \{ \phi \in H^0(X, D) \mid \Vert \phi \Vert_g \leq 1 \}
\]
and 
\[
\ah(X,\overline{D}) =
\begin{cases} 
\log \# \aH(X, \overline{D}) & \text{if $\aH(X, \overline{D})$ is finite}, \\
\infty & \text{otherwise},
\end{cases}
\]
where $\Vert \phi \Vert_g$ is the essential supremum of $\vert \phi \vert_g = \vert \phi \vert \exp(-g/2)$.
Note that
\[
\aH(X, \overline{D}) = \{ \phi \in \Rat(X)^{\times} \mid \widehat{(\phi)} + \overline{D} \geq 0 \} \cup \{ 0 \}.
\]
The volume $\avol(\overline{D})$ of $\overline{D}$ is defined to be
\[
\avol(\overline{D}) = \limsup_{n\to\infty}\frac{\ah(X, n\overline{D})}{n^d/d!}.
\]
For arithmetic $\RR$-divisors $\overline{D}$ and $\overline{D}'$ on $X$, if $\overline{D} \leq \overline{D}'$, then
$\aH(X, \overline{D}) \subseteq \aH(X, \overline{D}')$ and $\avol(\overline{D}) \leq \avol(\overline{D}')$ hold.

\begin{Proposition}
Let $\TT$ be a type for Green functions on $X$ and let $\overline{D} = (D, g)$ be an arithmetic $\RR$-divisor of $\TT$-type on $X$.
If $g$ is either of upper bounded type or of lower bounded type,
then $\aH(X, \overline{D})$ is finite. Moreover, if $g$ is of upper bounded type,
then $\avol(\overline{D}) < \infty$.
\end{Proposition}

\begin{proof}
First we assume that $g$ is of lower bounded type.
Then, by Lemma~\ref{lem:norm:locally:integrable},
$\Vert\cdot\Vert_g$ yields a norm of $H^0(X, D)$, and hence the assertion follows.

Next we assume that $g$ is of upper bounded type.
Then, by Proposition~\ref{prop:upper:bounded:Green:C:infty},
there is an $F_{\infty}$-invariant $D$-Green function $g'$ of $C^{\infty}$-type such that $g \leq g'\ \aew$.
By Proposition~\ref{prop:green:irreducible:decomp}, we can 
choose $a_1, \ldots, a_l \in \RR$ and $\overline{D}_1, \ldots, \overline{D}_l \in \aDiv_{C^{\infty}}(X)$ such that
$(D, g') = a_1 \overline{D}_1 + \cdots + a_l \overline{D}_l$. For each $i$,
by using Lemma~\ref{lem:Weil:leq:Cartier} and Lemma~\ref{lem:diff:effective:arith:div}, 
we can find effective arithmetic divisors $\overline{A}_i$ and $\overline{B}_i$ of $C^{\infty}$-type such that
$\overline{D}_i = \overline{A}_i - \overline{B}_i$.
As 
\[
(D, g') = a_1 \overline{A}_1 + \cdots + a_l \overline{A}_l + (-a_1) \overline{B}_1 + \cdots + (-a_l) \overline{B}_l,
\]
if we set $\overline{D}'' = \lceil a_1 \rceil \overline{A}_1 + \cdots + \lceil a_l \rceil  \overline{A}_l + \lceil (-a_1) \rceil  \overline{B}_1 + \cdots + \lceil (-a_l) \rceil  \overline{B}_l$,
then $(D, g') \leq \overline{D}''$ and $\overline{D}'' \in \aDiv_{C^{\infty}}(X)$.
Note that 
\[
\aH(X, n\overline{D}) \subseteq \aH(X, n(D, g')) \subseteq \aH(X, n\overline{D}'') = \aH(X, \overline{\OO}(\overline{D}'')^{\otimes n})
\] 
for all $n \geq 1$.
Thus our assertion follows from \cite[Lemma~3.3]{MoCont}.
\end{proof}

Here we consider the fundamental properties of $\avol$ on $\aDiv_{C^0}(X)_{\RR}$.

\begin{Theorem}
\label{thm:aDiv:aPic:R}
There is a natural surjective  homomorphism
\[
\overline{\OO}_{\RR} : \aDiv_{C^0}(X)_{\RR} \to \aPic_{C^0}(X)_{\RR}
\]
such that the following diagram is commutative:
\[
\begin{CD}
\aDiv_{C^0}(X) \otimes_{\ZZ} \RR @>{\overline{\OO} \otimes \operatorname{id}} >> \aPic_{C^0}(X) \otimes_{\ZZ} \RR \\
@VVV @VVV \\
\aDiv_{C^0}(X)_{\RR} @>{\overline{\OO}_{\RR}}>> \aPic_{C^0}(X)_{\RR}. \\
\end{CD}
\]
Moreover, we have the following:
\begin{enumerate}
\renewcommand{\labelenumi}{(\arabic{enumi})}
\item
For all $\overline{D} \in \aDiv_{C^0}(X)_{\RR}$,
\[
\avol(\overline{D}) = \lim_{t\to\infty} \frac{\ah(t\overline{D})}{t^d/d!} =\avol(\overline{\OO}_{\RR}(\overline{D})),
\]
where $t \in \RR_{>0}$ and $\avol(\overline{\OO}_{\RR}(\overline{D}))$ is the volume defined in \cite[Section~4]{MoContExt}.

\item
$\avol(a\overline{D}) = a^d \avol(\overline{D})$ for all $a \in \RR_{\geq 0}$ and $\overline{D} \in \aDiv_{C^0}(X)_{\RR}$.

\item
\rom{(}Continuity of $\avol$\rom{)}
\query{Change the statement of (3)
(11/November/2010)}
Let $\overline{D}_1, \ldots, \overline{D}_r, \overline{A}_1, \ldots, \overline{A}_{r'} \in \aDiv_{C^0}(X)_{\RR}$.
For a compact set $B$ in $\RR^r$ and a positive number $\epsilon$, there are positive numbers $\delta$ and $\delta'$ such that,
for all $a_1, \ldots, a_r, \delta_1, \ldots, \delta_{r'} \in \RR$ and $\phi \in C^0(X)$
with $(a_1, \ldots, a_r) \in B$, $\sum_{j=1}^{r'} \vert \delta_j \vert \leq \delta$ and $\Vert \phi \Vert_{\sup} \leq \delta'$,
we have
\[
\hskip3em
\left\vert \avol\left(\sum_{i=1}^r a_i \overline{D}_i + \sum_{j=1}^{r'} \delta_j \overline{A}_j + (0, \phi) \right) -
\avol\left(\sum_{i=1}^r a_i \overline{D}_i \right) \right\vert
\leq \epsilon.
\]
Moreover, if  $\overline{D}_1, \ldots, \overline{D}_r$, $\overline{A}_1, \ldots, \overline{A}_{r'}$ are $C^{\infty}$, then
there is a positive constant $C$ depending only on
$X$ and $\overline{D}_1, \ldots, \overline{D}_r, \overline{A}_1, \ldots, \overline{A}_{r'}$ such that
\begin{multline*}
\hskip3em
\left\vert \avol\left(\sum_{i=1}^r a_i \overline{D}_i + \sum_{j=1}^{r'} \delta_j \overline{A}_j + (0, \phi) \right) -
\avol\left(\sum_{i=1}^r a_i \overline{D}_i \right) \right\vert \\
\leq C\left( \sum_{i=1}^r \vert a_i \vert + \sum_{j=1}^{r'} \vert \delta_j \vert \right)^{d-1} \left(
\Vert \phi \Vert_{\sup} + \sum_{j=1}^{r'} \vert \delta_j \vert \right)
\end{multline*}
for all $a_1, \ldots, a_r, \delta_1, \ldots, \delta_{r'} \in \RR$ and $\phi \in C^0(X)$.

\item
Let $\overline{D}_1$ and $\overline{D}_2$ be arithmetic $\RR$-divisors of $C^0$-type.
If $\overline{D}_1$ and $\overline{D}_2$ are pseudo-effective 
\rom{(}for the definition of pseudo-effectivity, see SubSection~\rom{\ref{subsec:def:positivity:arith:div}}\rom{)},%
\query{$\ah(X, n_1\overline{D}_1) \not= 0$ and $\ah(X, n_2\overline{D}_2) \not= 0$ for some $n_1, n_2 \in \ZZ_{>0}$
\quad$\Longrightarrow$\quad
$\overline{D}_1$ and $\overline{D}_2$ are pseudo-effective 
(for the definition of pseudo-effectivity, see SubSection~\ref{subsec:def:positivity:arith:div})\\
(28/February/2010)}
then
\[
\avol(\overline{D}_1 + \overline{D}_2)^{1/d} \geq \avol(\overline{D}_1)^{1/d} + \avol(\overline{D}_2)^{1/d}.
\]

\item \rom{(}Fujita's approximation theorem for arithmetic $\RR$-divisors\rom{)}
If  $\overline{D}$ is an arithmetic $\RR$-divisor of $C^0$-type and $\avol(\overline{D}) > 0$, then, for any positive number $\epsilon$, there are a birational morphism $\mu : Y \to X$ of
generically smooth and normal projective arithmetic varieties and an ample arithmetic $\QQ$-divisor $\overline{A}$ of $C^{\infty}$-type on $Y$ \rom{(}cf. Section~\rom{\ref{sec:positivity:arith:div}}\rom{)}
such that $\overline{A} \leq \mu^*(\overline{D})$ and $\avol(\overline{A}) \geq \avol(\overline{D}) - \epsilon$.
\end{enumerate}
\end{Theorem}

Let us begin with the following lemmas.

\begin{Lemma}
\label{lem:Weil:leq:Cartier}
Let $Y$ be a normal projective arithmetic variety. Then we have the following:
\begin{enumerate}
\renewcommand{\labelenumi}{(\arabic{enumi})}
\item
Let $Z$ be a Weil divisor on $Y$.
Then there is an effective Cartier divisor $A$ on $Y$ such that $Z \leq A$.

\item
Let $D$ be a Cartier divisor on $Y$.
Then there are effective Cartier divisors $A$ and $B$ on $Y$ such that $D = A - B$.

\item
Let $x_1, \ldots, x_l$ be points of $Y$ and let $D$ be a Cartier divisor on $Y$.
Then there are effective Cartier divisors $A$ and $B$, and a non-zero rational function $\phi$ on $Y$ such that
$D + (\phi) = A - B$ and $x_1, \ldots, x_l \not\in \Supp(A) \cup \Supp(B)$.
\end{enumerate}
\end{Lemma}

\begin{proof}
(1) Let $Z = a_1 \Gamma_1 + \cdots + a_l \Gamma_l$ be the decomposition such that
$\Gamma_i$'s are prime divisors on $Y$ and $a_1, \ldots, a_l \in \ZZ$.
Let $L$ be an ample invertible sheaf on $Y$.
Then we can choose a positive integer $n$ and a non-zero section $s \in H^0(Y, L^{\otimes n})$
such that $\mult_{\Gamma_i}(s) \geq a_i$ for all $i$.
Thus, if we set $A = \zeros(s)$, then $A$ is a Cartier divisor and $Z \leq A$.

(2) First of all, we can find effective Weil divisors $A'$ and $B'$ on $Y$ such that $D = A' - B'$.
By the previous (1), there is an effective Cartier divisor $A$ such that $A' \leq A$.
We set $B = B' + (A - A')$. Then $B$ is effective and $D = A - B$.
Moreover, since $B = A - D$, $B$ is a Cartier divisor.

(3) Let $L$ be an ample invertible sheaf on $Y$ as before. Then there are a positive integer $n_1$ and
a non-zero $s_1 \in H^0(Y, L^{\otimes n_1})$ such that $s_1(x_i) \not= 0$ for all $i$.
We set $A' = \zeros(s_1)$. Similarly we can find a positive integer $n_2$ and  a non-zero $s_2 \in H^0(Y, \OO_Y(n_2A' - D))$
such that $s_2(x_i) \not= 0$ for all $i$. Therefore, if we set $A = n_2A'$ and $B = \zeros(s_2)$, then there is a non-zero rational function $\phi$ on $Y$
such that $A - D  = B + (\phi)$, as required.
\end{proof}

\begin{Lemma}
\label{lem:diff:effective:arith:div}
Let $\TT$ be either $C^0$ or $C^{\infty}$.
Let $A'$ and $A''$ be effective $\RR$-divisors on $X$ and $A = A' - A''$.
Let $g_A$ be an $F_{\infty}$-invariant $A$-Green function of $\TT$-type on $X(\CC)$.
Then there are effective arithmetic $\RR$-divisors $(A', g_{A'})$ and $(A'', g_{A''})$ of $\TT$-type
such that $(A, g_A) = (A', g_{A'}) - (A'', g_{A''})$.
\end{Lemma}

\begin{proof}
Let $g_{A''}$ be an $F_{\infty}$-invariant $A''$-Green function of $\TT$-type such that $g_{A''} \geq 0 \ \aew$.
We put $g_{A'} = g_A + g_{A''}$. Then $g_{A'}$ is an $F_{\infty}$-invariant $A'$-Green function of $\TT$-type.
Replacing $g_{A''}$ with $g_{A''} + (\text{positive constant})$ if necessarily,
we have $g_{A'} \geq 0\ \aew$.
\end{proof}

\begin{Lemma}
\label{lem:kernel:Div:R}
Let $\TT$ be a type for Green functions such that $-\TT \subseteq \TT$ and $C^{\infty} \subseteq \TT$.
Then the kernel of the natural homomorphism $\aDiv_{\TT}(X) \otimes_{\ZZ} \RR \to \aDiv_{\TT}(X)_{\RR}$
\query{$\otimes \RR$\quad
$\Longrightarrow$\quad $\otimes_{\ZZ} \RR$
(17/October/2010)}
coincides with
\[
\left\{ \sum_{i=1}^l (0, \phi_i) \otimes a_i \left| 
\begin{array}{l}
\text{$a_1, \ldots, a_l \in \RR$, $\phi_1, \ldots, \phi_l \in \TT(X)$} \\
\text{and $a_1 \phi_1 + \cdots + a_l \phi_l = 0$}
\end{array}
\right\}\right..
\]
\end{Lemma}

\begin{proof}
It is sufficient to show that, for $\sum_{i=1}^l (D_i, g_i) \otimes a_i \in \aDiv_{\TT}(X) \otimes_{\ZZ} \RR$,
if
\[
\sum_{i=1}^l a_i D_i = 0\quad\text{and}\quad
\sum_{i=1}^l a_i g_i = 0\ \aew,
\]
then there are $\phi_1, \ldots, \phi_l \in \TT(X)$ such that $\sum_{i=1}^l (D_i, g_i) \otimes a_i = \sum_{i=1}^l (0, \phi_i) \otimes a_i$ and
$a_1 \phi_1 + \cdots + a_l \phi_l = 0$.
Let $E_1, \ldots, E_r$ be a free basis of the $\ZZ$-submodule of $\Div(X)$ generated by $D_1, \ldots, D_l$.
We set $D_i = \sum_{j=1}^r b_{ij} E_j$. Since
\[
0 = \sum_{i=1}^l a_i D_i = \sum_{j=1}^r \left(\sum_{i=1}^l a_i b_{ij}\right) E_j,
\]
we have $\sum_{i=1}^l a_i b_{ij} = 0$ for all $j=1, \ldots, r$.
Let $h_j$ be an $F_{\infty}$-invariant $E_j$-Green function of $C^{\infty}$-type. 
Note that $\sum_{j=1}^r b_{ij}h_j$ is an $F_{\infty}$-invariant $D_i$-Green function of $\TT$-type.
Thus we can find $\phi_1, \ldots, \phi_l \in \TT(X)$ such that
\[
g_i = \sum_{j=1}^r b_{ij} h_j +\phi_{i}\quad\aew
\]
for each $i$.
Then
\[
0 \overset{\aew}{=} \sum_{i=1}^l a_i g_i \overset{\aew}{=} \sum_{j=1}^r \left( \sum_{i=1}^l a_i b_{ij} \right) h_j + \sum_{i=1}^l a_i \phi_i  =  \sum_{i=1}^l a_i \phi_{i}.
\]
Note that $\sum_{i} a_{i} \phi_{i} \in \TT(X)$, so that $\sum_{i} a_{i} \phi_{i} = 0$ over $X(\CC)$.
On the other hand,
\begin{multline*}
\sum_{i=1}^l (D_i, g_i) \otimes a_i = \sum_{i=1}^l \sum_{j=1}^r(E_j, h_j) \otimes a_i b_{ij} + \sum_{i=1}^l (0, \phi_i) \otimes a_i  \\
=
\sum_{j=1}^r (E_j, h_j) \otimes \left( \sum_{i=1}^l a_i b_{ij} \right)  + \sum_{i=1}^l (0, \phi_{i}) \otimes a_{i} 
= \sum_{i=1}^l (0, \phi_{i}) \otimes a_{i},
\end{multline*}
as required.
\end{proof}

\begin{proof}[The proof of Theorem~\ref{thm:aDiv:aPic:R}]
By Proposition~\ref{prop:green:irreducible:decomp},  the natural homomorphism 
\[
\aDiv_{C^0}(X) \otimes_{\ZZ} \RR \to \aDiv_{C^0}(X)_{\RR}
\]
\query{$\otimes \RR$\quad
$\Longrightarrow$\quad $\otimes_{\ZZ} \RR$
(17/October/2010)}
is surjective.
Thus the first assertion follows from Lemma~\ref{lem:kernel:Div:R}.

\medskip
(1) We set $\overline{D} = a_1 \overline{D}_1+ \cdots + a_l\overline{D}_l$,
where 
$a_1, \ldots, a_l \in \RR$ and $\overline{D}_1, \ldots, \overline{D}_l \in \aDiv_{C^0}(X)$.
For each $\overline{D}_i$, by using Lemma~\ref{lem:Weil:leq:Cartier} and Lemma~\ref{lem:diff:effective:arith:div},
we can find effective arithmetic divisors $\overline{D}'_i$ and $\overline{D}''_i$ of $C^0$-type such that
$\overline{D}_i = \overline{D}'_i - \overline{D}''_i$.
Then
\[
\overline{D} = a_1 \overline{D}'_1 + \cdots + a_l \overline{D}'_l + (-a_1) \overline{D}''_1 + \cdots + (-a_l) \overline{D}''_l.
\]
Thus, in order to see our assertion, we may assume that $\overline{D}_i$ is effective for every $i$.
We set $I = \{ i \mid a_i \geq 0 \}$ and $J = \{ i \mid a_i < 0 \}$.
Moreover, we set
\[
\begin{cases}
\overline{A}_n = \sum_{i \in I} \lfloor na_i \rfloor \overline{D}_i  + \sum_{j \in J} \lfloor (n+1) a_j \rfloor \overline{D}_j, \\
\overline{B}_n = \sum_{i \in I} \lceil na_i \rceil \overline{D}_i + \sum_{j \in J}  \lceil (n-1) a_j \rceil  \overline{D}_j
\end{cases}
\]
for $n \in \ZZ_{\geq 1}$. Then,  as $\lim_{n\to\infty} \overline{A}_n/n = \lim_{n\to\infty} \overline{B}_n/n = \overline{D}$,
by virtue of \cite[Theorem~5.1]{MoContExt},
\[
\lim_{n\to\infty} \frac{\ah(X, \overline{A}_n)}{n^d/d!} = \lim_{n\to\infty}\frac{\ah(X, \overline{B}_n)}{n^d/d!} = \avol(\overline{\OO}_{\RR}(\overline{D})).
\]
Note that
\[
\begin{cases}
\lfloor \lfloor t \rfloor a \rfloor \leq t  a \leq \lceil \lceil t \rceil a \rceil & \text{if $a \geq 0$}, \\
\lfloor  (\lfloor t \rfloor +1 ) a \rfloor \leq t  a \leq  \lceil  ( \lceil t \rceil - 1)a \rceil & \text{if $a < 0$}
\end{cases}
\]
for $a \in \RR$ and $t \in \RR_{\geq 1}$, which yields
$\overline{A}_{\lfloor t \rfloor} \leq t \overline{D} \leq \overline{B}_{\lceil t \rceil}$
for $t \in \RR_{\geq 1}$. Therefore,
\[
\frac{(\lfloor t \rfloor)^d}{t^d}\cdot \frac{h^0(X, \overline{A}_{\lfloor t \rfloor})}{(\lfloor t \rfloor)^d/d!} \leq \frac{h^0(X, t \overline{D})}{t^d/d!} 
\leq \frac{(\lceil t \rceil)^d}{t^d}\cdot \frac{h^0(X,  \overline{B}_{\lceil t \rceil})}{(\lceil t \rceil)^d/d!},
\]
and hence (1) follows.

\medskip
(2) follows from (1).

\medskip
\query{Change the proof of (3).
(11/November/2010)}
(3) The first assertion follows from \cite[(4) in Proposition~4.6]{MoContExt}.
Let us see the second assertion.
We choose $\overline{E}_1, \ldots, \overline{E}_{m}, \overline{B}_1, \ldots, \overline{B}_{m'} \in \aDiv_{C^{\infty}}(X)$
such that $\overline{D}_i = \sum_{k=1}^{m} \alpha_{ik} \overline{E}_k$ and
$\overline{A}_j = \sum_{l=1}^{m'} \beta_{jl} \overline{B}_l$ for some
$\alpha_{ik}, \beta_{jl} \in \RR$.
Then
\[
\sum_{i=1}^r a_i \overline{D}_i = \sum_{k=1}^m \left( \sum_{i=1}^r a_i \alpha_{ik} \right) \overline{E}_k
\quad\text{and}\quad
\sum_{j=1}^{r'} \delta_j \overline{A}_j = \sum_{l=1}^{m'} \left( \sum_{j=1}^{r'} \delta_j \beta_{jl} \right) \overline{B}_l.
\]
Moreover, if we set $C' = \max \left( \{ \alpha_{ik} \} \cup \{ \beta_{jl}\} \right)$,
then
\[
\left| \sum_{i=1}^r a_i \alpha_{ik} \right| \leq C' \sum_{i=1}^r \vert a_i \vert\quad\text{and}\quad
\left| \sum_{j=1}^{r'} \delta_j \beta_{jl} \right| \leq C' \sum_{j=1}^{r'} \vert \delta_j \vert.
\]
Thus we may assume that $\overline{D}_1, \ldots, \overline{D}_r, \overline{A}_1, \ldots, \overline{A}_{r'} \in \aDiv_{C^{\infty}}(X)$.
Therefore, the second assertion of (3) follows from \cite[Lemma~3.1, Theorem~4.4 and Proposition~4.6]{MoContExt}. 

\medskip
(4) If $\avol(\overline{D}_1) > 0$ and $\avol(\overline{D}_2) >0$, then
(4) follows from (3) and \cite[Theorem~B]{YuanVol} (or \cite[Theorem~6.2]{MoArLin}).
Let us fix an ample arithmetic divisor $\overline{A}$ (for the definition of ampleness, see SubSection~\ref{subsec:def:positivity:arith:div}). 
Then $\avol(\overline{D}_1 + \epsilon \overline{A}) > 0$ and
$\avol(\overline{D}_2 + \epsilon \overline{A}) > 0$ for all $\epsilon > 0$ by Proposition~\ref{prop:equiv:pseudo:effective}.
\query{an arithmetic divisor $\overline{A}$ of $C^{\infty}$-type such that
$\avol(\overline{A}) > 0$. Then $\avol(\overline{D}_1 + \epsilon \overline{A}) > 0$ and
$\avol(\overline{D}_2 + \epsilon \overline{A}) > 0$ for all $\epsilon > 0$
\quad$\Longrightarrow$\quad
an ample arithmetic divisor $\overline{A}$ (for the definition of ampleness, see SubSection~\ref{subsec:def:positivity:arith:div}). 
Then $\avol(\overline{D}_1 + \epsilon \overline{A}) > 0$ and
$\avol(\overline{D}_2 + \epsilon \overline{A}) > 0$ for all $\epsilon > 0$ by Proposition~\ref{prop:equiv:pseudo:effective}\\
(28/February/2010)}
Thus, by using (3) and the previous observation, we obtain (4).

\medskip
(5) By using the continuity of $\avol$ and the Stone-Weierstrass theorem,
we can find an arithmetic $\QQ$-divisor $\overline{D}'$ of $C^{\infty}$-type such that
$\overline{D}' \leq \overline{D}$ and 
\[
\avol(\overline{D}') > \max \{ \avol(\overline{D}) - \epsilon/2, 0 \}.
\]
Then, by virtue of \cite{HChenFujita}, \cite{YuanVol} or \cite{MoArLin},
there are a birational morphism $\mu : Y \to X$ of
generically smooth and normal projective arithmetic varieties and an ample arithmetic $\QQ$-divisor $\overline{A}$ of $C^{\infty}$-type on $Y$
such that $\overline{A} \leq \mu^*(\overline{D}')$ and $\avol(\overline{A}) \geq \avol(\overline{D}') - \epsilon/2$.
Thus (5) follows.
\end{proof}

\subsection{Intersection number of arithmetic $\RR$-divisors with a $1$-dimensional subscheme}
\label{subsec:intersection:arithmetic:R:divisor:curve}
\setcounter{Theorem}{0}
We assume that $X$ is projective.
Let $C$ be a $1$-dimensional closed integral subscheme of $X$.
Let $\overline{\mathcal{L}} = (\mathcal{L}, h)$ be an $F_{\infty}$-invariant  continuous hermitian invertible sheaf on $X$.
Then it is well-known that $\adeg(\rest{\overline{\mathcal{L}}}{C})$ is defined and it has the following property:
if $s$ is not zero element  of $H^0(X, \mathcal{L})$ with $\rest{s}{C} \not= 0$,
then
\[
\adeg(\rest{\overline{\mathcal{L}}}{C}) = \log \# \left( \frac{\rest{\mathcal{L}}{C}}{ \OO_C \cdot s} \right) - \frac{1}{2} \sum_{x \in C(\CC)} \log (h( s,s)(x)).
\]
In addition, the map
\[
\aPic_{C^0}(X) \to \RR \quad (\overline{\mathcal{L}} \mapsto \adeg(\rest{\overline{\mathcal{L}}}{C}))
\]
is a homomorphism of abelian groups, so that it extends to a homomorphism
\[
\adeg(\rest{-}{C}) : \aPic_{C^0}(X) \otimes_{\ZZ} \RR \to \RR
\]
\query{$\otimes \RR$\quad
$\Longrightarrow$\quad $\otimes_{\ZZ} \RR$
(17/October/2010)}
given by 
\[
\adeg(\rest{(\overline{\mathcal{L}}_1 \otimes a_1 + \cdots + \overline{\mathcal{L}}_r \otimes a_r)}{C}) =
a_1 \adeg(\rest{\overline{\mathcal{L}}_1}{C}) + \cdots + a_r \adeg(\rest{\overline{\mathcal{L}}_r}{C}).
\]
If $f_1, \ldots, f_r \in C^0(X)$,  $a_1, \ldots, a_r \in \RR$ and $a_1 f_1 + \cdots + a_r f_r = 0$, then
\begin{multline*}
\adeg(\rest{(\overline{\OO}(f_1) \otimes a_1 + \cdots + \overline{\OO}(f_r) \otimes a_r)}{C}) \\
=
a_1 \adeg(\rest{\overline{\OO}(f_1)}{C}) + \cdots + a_r \adeg(\rest{\overline{\OO}(f_r)}{C}) \\
 =
\sum_{i=1}^r a_i \left( \sum_{x \in C(\CC)} f_i(x) \right)= 0.
\end{multline*}
Therefore, $\adeg(\rest{-}{C}) :  \aPic_{C^0}(X) \otimes_{\ZZ} \RR \to \RR$
\query{$\otimes \RR$\quad
$\Longrightarrow$\quad $\otimes_{\ZZ} \RR$
(17/October/2010)}
descents to a homomorphism
$\aPic_{C^0}(X)_{\RR} \to \RR$.
By abuse of notation, we use the same symbol $\adeg(\rest{-}{C})$ to denote
the homomorphism $\aPic_{C^0}(X)_{\RR} \to \RR$.
Using this homomorphism, we define 
\[
\adeg(\rest{-}{C}) : \aDiv_{C^0}(X)_{\RR} \to \RR
\]
to be $\adeg(\rest{\overline{D}}{C}) := \adeg(\rest{\overline{\OO}_{\RR}(\overline{D})}{C})$ for $\overline{D}  \in \aDiv_{C^0}(X)_{\RR}$.
If there are effective Cartier divisors $D_1, \ldots, D_l$ and $a_1, \ldots, a_l \in \RR$ such that $D =a_1 D_1 + \cdots + a_l D_l$ and
$C \not\subseteq \Supp(D_i)$ for all $i$, then we can see that
\[
\adeg(\rest{\overline{D}}{C}) = \sum_{i=1}^l a_i \log \#(\OO_C(D_i)/\OO_C)  + \frac{1}{2} \sum_{x \in C(\CC)} g_{\rm can}(x).
\]

Let $\TT$ be a type for Green functions on $X$ such that $C^0 \subseteq \TT$, $\TT$ is real valued and $-\TT \subseteq \TT$.
Let $\overline{D} = (D, g)$ be an arithmetic $\RR$-divisor of $\TT$-type on $X$.
There is $h \in \TT(X)$ such that $g - h$ is an $F_{\infty}$-invariant  $D$-Green function of $C^0$-type.
We would like to define $\adeg(\rest{\overline{D}}{C})$ by the following quantity:
\[
\adeg\left(\rest{(D, g - h)}{C}\right) +\frac{1}{2} \sum_{x \in C(\CC)} h(x).
\]
Indeed, it does not depend on the choice of $h$.
Let $h'$ be another element of $\TT(X)$ such that 
$g - h'$ is  an $F_{\infty}$-invariant  $D$-Green function of $C^0$-type.
We can find $u \in C^0(X)$ such that $g - h = g - h' + u\ \aew$, so that
$h' = h+ u$ over $X(\CC)$. Therefore,
\begin{multline*}
\adeg\left(\rest{(D, g - h')}{C}\right) +\frac{1}{2} \sum_{x \in C(\CC)} h'(x)  \\
\hspace{-4em}= \adeg(\rest{(D, (g -h) - u)}{C}) +\frac{1}{2} \sum_{x \in C(\CC)} (h+u)(x) \\
\hspace{6em} = \adeg(\rest{(D, (g -h))}{C}) - \frac{1}{2} \sum_{x \in C(\CC)} u(x) + \frac{1}{2} \sum_{x \in C(\CC)} (h+u)(x) \\
 = \adeg(\rest{(D, g - h)}{C}) +\frac{1}{2} \sum_{x \in C(\CC)} h(x).
\end{multline*}
Note that if there are effective Cartier divisors $D_1, \ldots, D_l$ and $a_1, \ldots, a_l \in \RR$ such that
$D =a_1 D_1 + \cdots + a_l D_l$ and 
$C \not\subseteq \Supp(D_i)$ for all $i$, then
\[
\adeg\left(\rest{\overline{D}}{C}\right) = \sum_{i=1}^l a_i \log \#(\OO_C(D_i)/\OO_C)  + \frac{1}{2} \sum_{x \in C(\CC)} g_{\rm can}(x).
\]
Moreover, $\adeg(\rest{-}{C}) : \aDiv_{\TT}(X)_{\RR} \to \RR$ is a homomorphism.

\bigskip
Let $Z_1(X)$ be the group of $1$-cycles on $X$ and $Z_1(X)_{\RR} = Z_1(X) \otimes_{\ZZ} \RR$.
\query{$\otimes \RR$\quad
$\Longrightarrow$\quad $\otimes_{\ZZ} \RR$
(17/October/2010)}
Let $Z$ be an element of $Z_1(X)_{\RR}$.
There is a unique expression 
$Z = a_1 C_1 + \cdots + a_l C_l$ such that $a_1, \ldots, a_l \in \RR$ and $C_1, \ldots, C_l$ are $1$-dimensional closed integral schemes on $X$.
For $\overline{D} \in \aDiv_{\TT}(X)_{\RR}$, we define $\adeg \left( \overline{D} \mid Z \right)$ to be
\[
\adeg \left( \overline{D} \mid Z \right) := \sum_{i=1}^l a_i \adeg \left( \rest{\overline{D}}{C_i} \right).
\]
Note that $\adeg \left( \overline{D} \mid C \right) = \adeg \left( \rest{\overline{D}}{C} \right)$ for a $1$-dimensional closed integral scheme $C$ on $X$.

\section{Positivity of arithmetic $\RR$-divisors}
\label{sec:positivity:arith:div}
In this section, we will introduce a lot of kinds of positivity for arithmetic $\RR$-divisors and
investigate their properties.
Throughout this section, let $X$ be a generically smooth projective and normal arithmetic variety. 
 
\subsection{Definitions}
\label{subsec:def:positivity:arith:div}
\setcounter{Theorem}{0}
Let $\overline{D} = (D, g)$ be an arithmetic $\RR$-divisor on $X$, that is,
$D \in \Div(X)_{\RR}$ and $g$ is an $F_{\infty}$-invariant  locally integrable function on $X(\CC)$.
The ampleness, adequateness, nefness, bigness and pseudo-effectivity of $\overline{D}$ are defined as follows:

\medskip
\noindent
$\bullet$ {\bf ample} : 
We say $\overline{D}$ is {\em ample} if there are $a_1, \ldots, a_r \in \RR_{>0}$ and
ample arithmetic $\QQ$-divisors $\overline{A}_1, \ldots, \overline{A}_r$ of $C^{\infty}$-type
(i.e., $\overline{\OO}(n_i \overline{A}_i)$ is an ample $C^{\infty}$-hermitian invertible sheaf for some $n_i \in \ZZ_{>0}$ in
the sense of \cite{MoHeight})
such that 
\[
\overline{D} = a_1 \overline{A}_1 + \cdots + a_r \overline{A}_r.
\]
Note that an ample arithmetic $\RR$-divisor is of $C^{\infty}$-type.
The set of all ample arithmetic $\RR$-divisors on $X$ is denoted by $\aAmp(X)_{\RR}$.
By applying  \cite[Lemma~1.1.3]{MoArLin} to the case where $P = \aDiv_{C^{\infty}}(X)_{\QQ}$, $m=1$, $b_1 = 0$,
$A = {}^t(0, \ldots, 0)$ and $x_1 = \overline{A}_1, \ldots, x_r = \overline{A}_r$, we can see that
\[
\aAmp(X)_{\RR}\cap \aDiv_{C^{\infty}}(X)_{\QQ} =
\left\{ \overline{D}  
\left| 
\begin{array}{l}
\text{$\overline{\OO}(n\overline{D})$ is an ample $C^{\infty}$-hermitian} \\
\text{invertible sheaf on $X$ for some $n \in \ZZ_{>0}$} 
\end{array}
\right\}\right..
\]

\medskip
\noindent
$\bullet$ {\bf adequate} : 
$\overline{D}$ is said to be {\em adequate} if there are an ample arithmetic $\RR$-divisor $\overline{A}$
and a non-negative $F_{\infty}$-invariant  continuous function $f$ on $X(\CC)$ such that $\overline{D} = \overline{A} + (0, f)$.
Note that an adequate arithmetic $\RR$-divisor is of $C^0$-type.

\medskip 
\noindent
$\bullet$ {\bf nef} :
We say 
 $\overline{D}$ is {\em nef} if the following properties holds:
\begin{enumerate}
\renewcommand{\labelenumi}{(\alph{enumi})}
\item
$\overline{D}$ is of $\Tpsh_{\RR}$-type.

\item
$\adeg(\rest{\overline{D}}{C}) \geq 0$ for
all $1$-dimensional closed integral subschemes $C$ of $X$.
\end{enumerate}
The cone of all nef arithmetic $\RR$-divisors on $X$ is denoted by $\aNef(X)_{\RR}$.
Moreover, the cone of all nef arithmetic $\RR$-divisors of $C^{\infty}$-type (resp. $C^0$-type) on $X$ is denoted by $\aNef_{C^{\infty}}(X)_{\RR}$
(resp. $\aNef_{C^{0}}(X)_{\RR}$).

\medskip
\noindent
$\bullet$ {\bf big} :
Let us fix a type $\TT$ for Green functions.
We say $\overline{D}$ is a {\em big arithmetic $\RR$-divisor of $\TT$-type} if
$\overline{D} \in \aDiv_{\TT^b}(X)_{\RR}$ (i.e. $\overline{D} \in \aDiv_{\TT}(X)_{\RR}$ and
$g$ is of bounded type) and $\avol(\overline{D}) > 0$.

\medskip
\noindent
$\bullet$ {\bf pseudo-effective} :
$\overline{D}$ is said to be {\em pseudo-effective} if $\overline{D}$ is of $C^0$-type and
there are arithmetic $\RR$-divisors $\overline{D}_1, \ldots, \overline{D}_r$ of $C^0$-type and
sequences $\{ a_{n1} \}_{n=1}^{\infty}, \ldots, \{ a_{nr} \}_{n=1}^{\infty}$ in $\RR$ such that
$\lim_{n\to\infty} a_{ni} = 0$ for all $i =1, \ldots, r$ and $\avol(\overline{D} + a_{n1} \overline{D}_1 + \cdots + a_{nr} \overline{D}_r) > 0$
for all $n \gg 1$.

\subsection{Properties of ample arithmetic $\RR$-divisors}
\setcounter{Theorem}{0}
In this subsection, we consider several properties of ample arithmetic $\RR$-divisors.
Let us begin with the following proposition.

\begin{Proposition}
\label{prop:ample:RR:div}
\begin{enumerate}
\renewcommand{\labelenumi}{(\arabic{enumi})}
\item
If $\overline{A}$ and $\overline{B}$ are ample \rom{(}resp. adequate\rom{)} arithmetic $\RR$-divisors  and
$a \in \RR_{>0}$,
then $\overline{A} + \overline{B}$ and $a \overline{A}$ are also ample  \rom{(}resp. adequate\rom{)}.

\item
If $\overline{A}$ is an ample arithmetic $\RR$-divisor, then
there are an ample arithmetic $\QQ$-divisor $\overline{A}'$  and an
ample arithmetic $\RR$-divisor $\overline{A}''$ 
such that $\overline{A} = \overline{A}' + \overline{A}''$.

\item
Let $\overline{A}$ be an ample  \rom{(}resp. adequate\rom{)} arithmetic $\RR$-divisor  and let
$\overline{L}_1, \ldots, \overline{L}_n$ be arithmetic $\RR$-divisors of $C^{\infty}$-type  \rom{(}resp. of $C^{0}$-type\rom{)}.
Then there is $\delta \in \RR_{>0}$ such that
$\overline{A} + \delta_1 \overline{L}_1 + \cdots + \delta_n \overline{L}_n$ is ample  \rom{(}resp. adequate\rom{)}
for $\delta_1, \ldots, \delta_n \in \RR$ with
$\vert \delta_1 \vert + \cdots + \vert \delta_n \vert \leq \delta$.

\item
If $\overline{A}$ is an adequate arithmetic $\RR$-divisor, then $\avol(\overline{A}) > 0$.
\end{enumerate}
\end{Proposition}

\begin{proof}
(1) and (2) are obvious. 

\medskip
(3) First we assume that $\overline{A}$ is ample and that $\overline{L}_1, \ldots, \overline{L}_n$ are of $C^{\infty}$-type.
We set $\overline{L}_i = \sum_{j=1}^l b_{ij} \overline{M}_j$ such that 
$\overline{M}_1, \ldots, \overline{M}_l$ are arithmetic $\QQ$-divisors of $C^{\infty}$-type and
$b_{ij} \in \RR$.
Then, as
\[
\overline{A} + \sum_{i=1}^n \delta_i \overline{L}_i =
\overline{A} + \sum_{j=1}^l \left( \sum_{i=1}^n \delta_i b_{ij} \right) \overline{M}_j,
\]
we may assume that $\overline{L}_1, \ldots, \overline{L}_n$ are arithmetic $\QQ$-divisors of $C^{\infty}$-type.
Moreover, by (1) and (2), we may further assume that $\overline{A}$ is
an ample arithmetic $\QQ$-divisor.

Let us choose $\delta \in \QQ_{>0}$ such that $\overline{A} \pm \delta \overline{L}_i$ is ample for every $i=1, \ldots, n$.
Note that
\[
\sum_{i=1}^n \frac{\vert \delta_i \vert }{\delta}( \overline{A} + \operatorname{sign}(\delta_i) \delta\overline{L}_i ) = \left( \sum_{i=1}^n  \frac{\vert \delta_i \vert }{\delta} \right) \overline{A}
+ \sum_{i=1}^n \delta_i \overline{L}_i,
\]
where $\operatorname{sign}(a)$ for $a \in \RR$ is given by
\[
\operatorname{sign}(a) = \begin{cases}
1 & \text{if $a \geq 0$}, \\
-1 & \text{if $a < 0$}.
\end{cases}
\]
Hence, if $\sum_{i=1} \vert \delta_i \vert \leq \delta$, then $\overline{A} + \sum_{i=1}^n \delta_i \overline{L}_i$ is ample.

Next we assume that $\overline{A}$ is adequate and that $\overline{L}_1, \ldots, \overline{L}_n$ are of $C^{0}$-type.
Then there are an ample arithmetic $\RR$-divisor $\overline{A}'$ and $u \in C^0(X)$ such that
$u \geq 0$ and $\overline{A} = \overline{A}' + (0,u)$.
As $\overline{A}' - (0, \epsilon)$ is ample for $0 < \epsilon \ll 1$ by the previous observation, we may assume that $u \geq \epsilon$ for some
positive number $\epsilon$.
By virtue of  the Stone-Weierstrass theorem, we can find $v_1, \ldots, v_n \in C^0(X)$ such that $v_i \geq 0\ (\forall i)$,
$\epsilon \geq v_1 + \cdots + v_n$ and
$\overline{L}'_i := \overline{L}_i + (0, v_i)$ is of $C^{\infty}$-type for all $i$.
By the previous case, we can find $0 < \delta < 1$ such that
$\overline{A}' + \delta_1 \overline{L}'_1 + \cdots + \delta_n \overline{L}'_n$ is ample
for $\delta_1, \ldots, \delta_n \in \RR$ with
$\vert \delta_1 \vert + \cdots + \vert \delta_n \vert \leq \delta$.
Note that
\[
\overline{A} + \delta_1 \overline{L}_1 + \cdots + \delta_n \overline{L}_n =
\overline{A}' + \delta_1 \overline{L}'_1 + \cdots + \delta_n \overline{L}'_n + (0, u - \delta_1 v_1 - \cdots - \delta_n v_n)
\]
and
\[
u - \delta_1 v_1 - \cdots - \delta_n v_n \geq u - v_1 - \cdots -  v_n \geq 0,
\]
as required.

\medskip
(4) Clearly we may assume that $\overline{A}$ is ample, so that the assertion follows from (2) and (4) in Theorem~\ref{thm:aDiv:aPic:R}.
\end{proof}

Next we consider the following proposition.

\begin{Proposition}
\label{prop:ample:plus:nef:ample}
\begin{enumerate}
\renewcommand{\labelenumi}{(\arabic{enumi})}
\item
If $\overline{A}$ is an ample arithmetic $\RR$-divisor and
$\overline{B}$ is a nef arithmetic $\RR$-divisor of $C^{\infty}$-type, then
$\overline{A} + \overline{B}$ is ample.

\item
If $\overline{A}$ is an adequate arithmetic $\RR$-divisor and
$\overline{B}$ is a nef arithmetic $\RR$-divisor of $C^{0}$-type, then
$\overline{A} + \overline{B}$ is adequate.
\end{enumerate}
\end{Proposition}

\begin{proof}
(1)
We set $\overline{B} = b_1 \overline{B}_1 + \cdots + b_n \overline{B}_n$ such that
$b_1, \ldots, b_n \in \RR$ and
$\overline{B}_1, \ldots, \overline{B}_n$ are arithmetic $\QQ$-divisors of $C^{\infty}$-type.
We choose an ample arithmetic $\QQ$-divisor $\overline{A}_1$  and an
ample arithmetic $\RR$-divisor $\overline{A}_2$ 
such that $\overline{A} = \overline{A}_1+ \overline{A}_2$.
Then, by (3) in Proposition~\ref{prop:ample:RR:div},
there are $\delta_1, \ldots, \delta_n \in \RR_{>0}$ such that
\[
\overline{A}_1 + \sum_{i=1}^n \delta_i \overline{B}_i\quad\text{and}\quad
\overline{A}_2 - \sum_{i=1}^n \delta_i \overline{B}_i
\]
are ample and $b_i + \delta_i \in \QQ$ for all $i$.
Moreover, we can take
an ample arithmetic $\QQ$-divisor $\overline{A}_3$  and an
ample arithmetic $\RR$-divisor $\overline{A}_4$ 
such that 
\[
\overline{A}_2 - \sum_{i=1}^n \delta_i \overline{B}_i = \overline{A}_3+ \overline{A}_4.
\]
Then, since
\[
\overline{A}_1 + \sum_{i=1}^n \delta_i \overline{B}_i + \overline{B} = \overline{A}_1 + \sum_{i=1}^n (b_i + \delta_i) \overline{B}_i
\]
is a nef arithmetic $\QQ$-divisor of $C^{\infty}$-type,
$\overline{A}_3 + \overline{A}_1 + \sum_{i=1}^n \delta_i \overline{B}_i + \overline{B}$ is an ample arithmetic $\QQ$-divisor
by \cite[Lemma~5.6]{MoCont}.
Therefore,
\[
\overline{A} + \overline{B} = \overline{A}_4 + \overline{A}_3 + \overline{A}_1 + \sum_{i=1}^n \delta_i \overline{B}_i + \overline{B}
\]
is an ample arithmetic $\RR$-divisor.

(2)
Clearly we may assume that $\overline{A}$ is ample.
By (3) in Proposition~\ref{prop:ample:RR:div}, there is a positive real number $\delta$ such that
$\frac{1}{2}\overline{A} - (0, \delta)$ is ample.
Note that $\frac{1}{2} A +  B$ is ample, that is, $\frac{1}{2} A + B$ is a linear combination of ample divisors with positive coefficients,
which can be checked in the same way as above.
Thus, by (2) in Theorem~\ref{thm:cont:upper:envelope},
there is $u \in C^0(X)$ (i.e., $u$ is an $F_{\infty}$-invariant continuous function in $X(\CC)$)
such that $0 \leq u < \delta$ on $X(\CC)$ and $\frac{1}{2}\overline{A} + \overline{B} + (0,u)$ is a nef
$\RR$-divisor of $C^{\infty}$-type. Then, by (1), 
\[
\frac{1}{2} \overline{A} - (0, \delta) + \frac{1}{2}\overline{A} + \overline{B} + (0, u)
\]
is ample. Thus
\[
\overline{A} + \overline{B} = \frac{1}{2} \overline{A} - (0, \delta) + \frac{1}{2}\overline{A} + \overline{B} + (0, u) + (0, \delta - u)
\]
is adequate.
\end{proof}

Finally let us observe the following lemma.

\begin{Lemma}
\label{lem:nef:D1:leq:D2}
Let $\overline{D}_1 = (D_1, g_1)$ and $\overline{D}_2 = (D_2, g_2)$ be arithmetic $\RR$-divisors of $\Tpsh_{\RR}$-type on $X$.
If $D_1 = D_2$, $g_1 \leq g_2\ \aew$ and $\overline{D}_1$ is nef,
then $\overline{D}_2$ is also nef.
\end{Lemma}

\begin{proof}
Since $D_1 = D_2$, there is a $\phi \in (\Tpsh_{\RR} -\Tpsh_{\RR})(X(\CC))$ 
such that
$g_2 = g_1 + \phi\ \aew$ and $\phi \geq 0 \ \aew$.
Note that $\phi \geq 0$ by Lemma~\ref{lem:fqpssh:ineq:ae}.
Let $C$ be a $1$-dimensional closed integral  subscheme of $X$.
Then
\[
\adeg(\rest{\overline{D}_2}{C}) = \deg(\rest{\overline{D}_1}{C}) + \frac{1}{2} \sum_{y \in C(\CC)} \phi(y) \geq  \deg(\rest{\overline{D}_1}{C})  \geq 0.
\]
\end{proof}

\subsection{Criterions of bigness and pseudo-effectivity}
\setcounter{Theorem}{0}

The purpose of this subsection is to prove the following propositions.

\begin{Proposition}
\label{prop:equiv:big:regular}
For $\overline{D} = (D, g)\in \aDiv_{C^0}(X)_{\RR}$, the following are equivalent:
\begin{enumerate}
\renewcommand{\labelenumi}{(\arabic{enumi})}
\item $\overline{D}$ is big, that is, $\avol(\overline{D}) > 0$.

\item
For any $\overline{A} \in \aDiv_{C^0}(X)_{\RR}$, there are a positive integer $n$ and
a non-zero rational function $\phi$ such that $\overline{A} \leq n \overline{D} + \widehat{(\phi)}$.
\end{enumerate}
\end{Proposition}

\begin{proof}
``(2) $\Longrightarrow$ (1)'' is obvious.

Let us consider ``(1) $\Longrightarrow$ (2)''.
By using Lemma~\ref{lem:Weil:leq:Cartier} and Lemma~\ref{lem:diff:effective:arith:div},
we can find effective arithmetic $\RR$-divisors $\overline{A}'$ and $\overline{A}''$ of $C^0$-type such that $\overline{A} = \overline{A}' - \overline{A}''$.
Note that $\overline{A} \leq \overline{A}'$.
Thus we may assume $\overline{A}$ is effective in order to see our assertion.
By virtue of the continuity of $\avol$ (cf. Theorem~\ref{thm:aDiv:aPic:R}), there is a positive integer $m$ such that
\[
\avol(\overline{D} - (1/m)\overline{A}) > 0,
\] that is, $\avol(m \overline{D} - \overline{A}) > 0$, so that
there is a positive integer $n$ and a non-zero rational function $\phi$ such that
\[
n(m\overline{D} - \overline{A}) + \widehat{(\phi)} \geq 0.
\]
Thus $mn\overline{D} + \widehat{(\phi)} \geq n \overline{A} \geq \overline{A}$.
\end{proof}

\begin{Proposition}
\label{prop:equiv:pseudo:effective}
For $\overline{D} = (D, g)\in \aDiv_{C^0}(X)_{\RR}$, the following are equivalent:
\begin{enumerate}
\renewcommand{\labelenumi}{(\arabic{enumi})}
\item $\overline{D}$ is pseudo-effective.

\item For any big arithmetic $\RR$-divisor $\overline{A}$ of $C^0$-type, 
$\avol(\overline{D} + \overline{A}) > 0$.
\query{ample $\Longrightarrow$ big\\ (19/July/2010)}

\item
There is a big arithmetic $\RR$-divisor $\overline{A}$ of $C^0$-type such that
$\avol(\overline{D} + (1/n)\overline{A}) > 0$ for 
all $n \geq 1$.
\end{enumerate}
\end{Proposition}

\begin{proof}
It is sufficient to see that (1) implies (2).
As $\overline{D}$ is pseudo-effective,
there are arithmetic $\RR$-divisors $\overline{D}_1, \ldots, \overline{D}_r$ of $C^0$-type and
sequences $\{ a_{m1} \}_{m=1}^{\infty}, \ldots, \{ a_{mr} \}_{m=1}^{\infty}$ in $\RR$ such that
$\lim_{m\to\infty} a_{mi} = 0$ for all $i =1, \ldots, r$ and $\avol(\overline{D} + a_{m1} \overline{D}_1 + \cdots + a_{mr} \overline{D}_r) > 0$
for all $m \gg 1$.
By the continuity of $\avol$,
there is a sufficiently large positive integer $m$ such that
$\overline{A} - (a_{m1} \overline{D}_1 + \cdots + a_{mr} \overline{D}_r)$ is big.
Thus
\[
\avol(\overline{D} + \overline{A}) \geq \avol(\overline{D} + a_{m1} \overline{D}_1 + \cdots + a_{mr} \overline{D}_r) > 0.
\]
\query{Rewrite the proof\\ (19/July/2010)}
\end{proof}

\query{Add Proposition~\ref{prop:pseudo:effective:big:on:generic}\\ (21/June/2010)}
\begin{Proposition}
\label{prop:pseudo:effective:big:on:generic}
If $\overline{D} = (D, g)$ is a pseudo-effective arithmetic $\RR$-divisor of $C^0$-type
such that $D$ is big on the generic fiber $X_{\QQ}$ 
\rom{(}i.e., $\operatorname{vol}(D_{\QQ}) > 0$ on $X_{\QQ}$\rom{)},
then $\overline{D} + (0, \epsilon)$ is big for all $\epsilon \in \RR_{>0}$.
\end{Proposition}

\begin{proof}
Let $\overline{A}$ be an ample arithmetic divisor on $X$.
Since $D$ is big on $X_{\QQ}$, 
by using the continuity of the volume function over $X_{\QQ}$ (cf. \cite[I, Corollary~2.2.45]{LazPositivity}),
we can see that
there are a positive integer $m$ and
a non-zero rational function $\phi$ such that
\[
m D - A + (\phi) \geq 0.
\]
If we set $(L, h) = m \overline{D} - \overline{A} + \widehat{(\phi)}$,
then $h$ is an $L$-Green function of $C^{0}$-type and $L$ is effective.
Thus there is a positive number $\lambda$ such that
\[
m \overline{D} - \overline{A} + \widehat{(\phi)} \geq (0, -\lambda),
\]
that is, $m \overline{D} + (0, \lambda) \geq \overline{A} - \widehat{(\phi)}$.
We choose a sufficiently large positive integer $n$ such that
\[
\frac{\lambda}{n + m} \leq \epsilon.
\]
Then
\begin{align*}
\overline{D} + \frac{1}{n}(\overline{A} - \widehat{(\phi)}) & \leq \overline{D} + \frac{1}{n}(m\overline{D} + (0,\lambda)) \\
& = \left( 1 + \frac{m}{n} \right) \left( \overline{D} + \left( 0, \frac{\lambda}{n+m} \right) \right) \\
& \leq \left( 1 + \frac{m}{n} \right) \left( \overline{D} + ( 0, \epsilon ) \right).
\end{align*}
Note that $\overline{A} - \widehat{(\phi)}$ is ample, so that
$\overline{D} + (1/n)(\overline{A} - \widehat{(\phi)})$ is big by Proposition~\ref{prop:equiv:pseudo:effective}, and
hence $\overline{D} + ( 0, \epsilon )$ is also big.
\end{proof}

\query{Add Remark~\ref{remark:pseudo:effective:non:big}\\ (21/June/2010)}
\begin{Remark}
\label{remark:pseudo:effective:non:big}
It is very natural to ask whether
$\aH(X, n(\overline{D} + (0, \epsilon))) \not= \{ 0 \}$ for some $n \in \ZZ_{>0}$
in the case where $D$ is not necessarily big on $X_{\QQ}$.
This does not hold in general. For example,
let $\PP^1_{\ZZ} = \Proj(\ZZ[T_0, T_1])$ be the projective line over $\ZZ$ and
$\overline{D} = a\widehat{(T_1/T_0)}$ for $a \in \RR \setminus \QQ$.
It is easy to see that $\overline{D}$ is pseudo-effective and
$H^0(\PP^1_{\ZZ}, nD) = \{ 0 \}$ for all $n \in \ZZ_{>0}$.
Thus
$\aH(\PP^1_{\ZZ}, n(\overline{D} + (0, \epsilon))) = \{ 0 \}$ for $\epsilon \in \RR_{>0}$ and $n \in \ZZ_{>0}$.
\end{Remark}

\subsection{Intersection number of arithmetic $\RR$-divisors of $C^0$-type}
\setcounter{Theorem}{0}
\label{subsec:intersection:nef:C:0}
Let $\aDiv_{C^{\infty}}(X) \times \cdots \times \aDiv_{C^{\infty}}(X) \to \RR$ be a symmetric multi-linear map over $\ZZ$ given by
\[
(\overline{D}_1, \ldots, \overline{D}_d) \mapsto \adeg(\overline{D}_1 \cdots \overline{D}_d) := \adeg(\acherncl_1(\overline{\OO}(\overline{D}_1)) \cdots
\acherncl_1(\overline{\OO}(\overline{D}_d))),
\]
which extends  to the symmetric multi-linear map 
\[
(\aDiv_{C^{\infty}}(X) \otimes_{\ZZ} \RR) \times \cdots \times (\aDiv_{C^{\infty}}(X) \otimes_{\ZZ} \RR) \to \RR
\]
\query{$\otimes \RR$\quad
$\Longrightarrow$\quad $\otimes_{\ZZ} \RR$
(17/October/2010)}
over $\RR$.

\begin{PropDef}
\label{propdef:intersection:Div:C:infty}
The above  multi-linear map 
\[
(\aDiv_{C^{\infty}}(X) \otimes_{\ZZ} \RR) \times \cdots \times (\aDiv_{C^{\infty}}(X) \otimes_{\ZZ} \RR) \to \RR
\]
\query{$\otimes \RR$\quad
$\Longrightarrow$\quad $\otimes_{\ZZ} \RR$
(17/October/2010)}
descents to
the symmetric multi-linear map
\[
\aDiv_{C^{\infty}}(X)_{\RR} \times \cdots \times \aDiv_{C^{\infty}}(X)_{\RR} \to \RR
\]
over $\RR$, whose value at $(\overline{D}_1, \ldots, \overline{D}_d) \in \aDiv_{C^{\infty}}(X)_{\RR} \times \cdots \times \aDiv_{C^{\infty}}(X)_{\RR}$ is also
denoted by $\adeg(\overline{D}_1 \cdots \overline{D}_d)$ by abuse of notation.
\end{PropDef}

\begin{proof}
Let $a_1, \ldots, a_l \in \RR$ and $\phi_1, \ldots, \phi_l \in C^{\infty}(X)$ such that
$a_1 \phi_1 + \cdots + a_l \phi_l = 0$.
By Lemma~\ref{lem:kernel:Div:R}, it is sufficient to show that
\[
\adeg \left( ((0, \phi_1) \otimes a_1 + \cdots + (0, \phi_l) \otimes a_l ) \cdot \overline{D}_2 \cdots \overline{D}_d \right) = 0
\]
for all $\overline{D}_2, \ldots, \overline{D}_d \in \aDiv_{C^{\infty}}(X)$.
First of all, note that there are $1$-dimensional closed integral subschemes $C_1, \ldots, C_r$, $c_1, \ldots, c_r \in \ZZ$ and
a current $T$ of $(d-2, d-2)$-type such that
\[
\overline{D}_2 \cdots \overline{D}_d \sim (c_1 C_1 + \cdots + c_r C_r, T).
\]
Then
\begin{multline*}
\adeg \left( ((0, \phi_1) \otimes a_1 + \cdots + (0, \phi_l) \otimes a_l ) \cdot \overline{D}_2 \cdots \overline{D}_d \right)  \\
\hspace{-9em}= \sum_{i=1}^l a_i \adeg\left( (0, \phi_i) \cdot (c_1 C_1 + \cdots + c_r C_r, T) \right) \\
= 
\sum_{i=1}^l a_i \left( \sum_{j=1}^r c_j \sum_{y \in C_j(\CC)} \phi_i(y) + (1/2) \int_{X(\CC)} dd^c(\phi_i) \wedge T\right) \\
= 
\sum_{j=1}^r c_j \sum_{y \in C_j(\CC)} \sum_{i=1}^l a_i  \phi_i(y) + (1/2) \int_{X(\CC)} dd^c\left(\sum_{i=1}^l a_i  \phi_i\right) \wedge T = 0,
\end{multline*}
as required.
\end{proof}

Let $\aDiv_{C^0}^{\Nef}(X)_{\RR}$ be the vector subspace of $\aDiv_{C^0}(X)_{\RR}$ generated by
$\aNef_{C^0}(X)_{\RR}$. 
The purpose of this subsection is to show the following proposition,
which gives a generalization of \cite[Lemma~6.5]{ZhPos} for $\RR$-divisors.

\begin{Proposition}
\label{prop:intersection:Div:nef:C:0}
\begin{enumerate}
\renewcommand{\labelenumi}{(\arabic{enumi})}
\item
\query{Update the statement of (1).
The original is ``$\aDiv_{C^{\infty}}(X)_{\RR} + \aDiv_{C^0 \cap \Tpsh}(X)_{\RR}
\subseteq \aDiv_{C^0}^{\Nef}(X)_{\RR}$''.
(12/July/2010)}
\[
\aDiv_{C^0 \cap \Tpsh + C^{\infty}}(X)_{\RR} \subseteq \aDiv_{C^0}^{\Nef}(X)_{\RR}
\subseteq \aDiv_{C^0 \cap \Tpsh - C^0 \cap \Tpsh}(X)_{\RR}.
\]

\item
The above symmetric multi-linear map 
\[
\aDiv_{C^{\infty}}(X)_{\RR} \times \cdots \times  \aDiv_{C^{\infty}}(X)_{\RR} \to \RR
\]
given in
Proposition-Definition~\rom{\ref{propdef:intersection:Div:C:infty}}
extends to a unique symmetric multi-linear map 
\[
\aDiv_{C^0}^{\Nef}(X)_{\RR} \times \cdots \times  \aDiv_{C^0}^{\Nef}(X)_{\RR} \to \RR
\]
such that
$(\overline{D}, \ldots, \overline{D}) \mapsto \avol(\overline{D})$ for $\overline{D} \in \aNef_{C^0}(X)_{\RR}$.
By abuse of notation, for 
\[
(\overline{D}_1, \ldots, \overline{D}_d) \in \aDiv_{C^0}^{\Nef}(X)_{\RR} \times \cdots \times  \aDiv_{C^0}^{\Nef}(X)_{\RR},
\]
the image of  $(\overline{D}_1, \ldots, \overline{D}_d)$ by the above extension is also denoted by 
\[
\adeg(\overline{D}_1 \cdots \overline{D}_d).
\]
\end{enumerate}
\end{Proposition}

\begin{proof}
\query{Change the proof of (1).
(12/July/2010)}
(1) 
It is obvious that
\[
\aDiv_{C^0}^{\Nef}(X)_{\RR}
\subseteq \aDiv_{C^0 \cap \Tpsh - C^0 \cap \Tpsh}(X)_{\RR}.
\]
Let $\overline{D} \in \aDiv_{C^0 \cap \Tpsh + C^{\infty}}(X)_{\RR}$.
By  Proposition~\ref{prop:psh:plus:positive}, there is an ample arithmetic divisor $\overline{A}$ with
$\overline{D} + \overline{A} \in \aDiv_{C^0 \cap \Tpsh}(X)_{\RR}$. Thus 
it is sufficient to show the following claim:

\begin{Claim}
\label{claim:prop:intersection:Div:nef:C:0:0}
For $\overline{D} \in \aDiv_{C^0 \cap \Tpsh}(X)_{\RR}$, there is an ample arithmetic divisor $\overline{B}$ such that
$\overline{D} + \overline{B} \in \aNef_{C^0}(X)_{\RR}$ and $D + B$ is ample.
\end{Claim}

\begin{proof}
By virtue of  the Stone-Weierstrass theorem,
there is an $F_{\infty}$-invariant non-negative continuous function $u$ on $X(\CC)$ such that
$\overline{D} - (0, u) \in \aDiv_{C^{\infty}}(X)_{\RR}$.
Thus, by Proposition~\ref{prop:ample:RR:div},
we can find an ample arithmetic divisor $\overline{B}$ such that 
\[
\overline{D} - (0, u) + \overline{B}
\]
is ample.
In particular, $\overline{D} + \overline{B} \in \aNef_{C^0}(X)_{\RR}$ and $D + B$ is ample.
\end{proof}

\medskip
(2) Let us begin with the following claim.

\begin{Claim}
\label{claim:prop:intersection:Div:nef:C:0:1}
\begin{enumerate}
\renewcommand{\labelenumi}{(\alph{enumi})}
\item
For $\overline{D} \in \aNef_{C^{\infty}}(X)_{\RR}$, $\adeg(\overline{D}^d) = \avol(\overline{D})$.

\item
${\displaystyle d! X_1 \cdots X_d = \sum_{I \subseteq \{ 1, \ldots, d\}} (-1)^{d-\#(I)} \left(\sum_{i \in I} X_i \right)^d}$ in $\ZZ[X_1, \ldots, X_d]$.

\item
For $\overline{D}_1, \ldots, \overline{D}_d \in \aNef_{C^{\infty}}(X)_{\RR}$,
\[
\adeg(\overline{D}_1 \cdots \overline{D}_d) = \frac{1}{d!} \sum_{I \subseteq \{ 1, \ldots, d\}} (-1)^{d - \#(I)} \avol\left(\sum_{i \in I} \overline{D}_i \right).
\]
\end{enumerate}
\end{Claim}

\begin{proof}
(a) First we assume that $\overline{D}$ is ample.
We set $\overline{D} = a_1 \overline{A}_1 + \cdots + a_l \overline{A}_l$ such that
$a_1, \ldots, a_l \in \RR_{>0}$ and $\overline{A}_i$'s are ample arithmetic divisors.
Let us choose sufficient small positive numbers $\delta_1, \ldots, \delta_l$ such that $a_i + \delta_i \in \QQ$ for all $i$.
Then, by \cite[Corollary~5.5]{MoCont}, 
\[
\adeg(((a_1+ \delta_1) \overline{A}_1 + \cdots + (a_l + \delta_l)\overline{A}_l)^d) = \avol((a_1+ \delta_1) \overline{A}_1 + \cdots + (a_l + \delta_l)\overline{A}_l).
\]
Thus, using the continuity of $\avol$, the assertion follows.

Next we consider a general case. Let $\overline{A}$ be an ample arithmetic divisor of $C^{\infty}$-type.
Then, by Proposition~\ref{prop:ample:plus:nef:ample}, $\overline{D} + \epsilon \overline{A}$ is ample for all $\epsilon > 0$.
Thus the assertion follows from the previous observation and the continuity of $\avol$.

\medskip
(b)
In general, let us see that
\[
\sum_{I \subseteq \{ 1, \ldots, d\}} (-1)^{\#(I)} \left(\sum_{i \in I} X_i \right)^l = \begin{cases}
0 & \text{if $l < d$}, \\
(-1)^d d! X_1 \cdots X_d & \text{if $l = d$}
\end{cases}
\]
holds for integers $d$ and $l$ with $1 \leq l \leq d$.
This assertion for $d$ and $l$ is denoted by $A(d,l)$.
$A(1,1)$ is obvious. Moreover, it is easy to see $A(d,1)$. Note that
\begin{multline*}
\int_{0}^{X_d} \left( \sum_{I \subseteq \{ 1, \ldots, d\}} (-1)^{\#(I)} \left(\sum_{i \in I} X_i \right)^{l-1} \right) dX_{d} \\
=
\frac{1}{l} \sum_{I \subseteq \{ 1, \ldots, d\}} (-1)^{\#(I)} \left(\sum_{i \in I} X_i \right)^{l} + X_d 
\sum_{J \subseteq \{ 1, \ldots, d-1\}} (-1)^{\#(J)} \left(\sum_{j \in J} X_j \right)^{l-1},
\end{multline*}
which shows that $A(d-1, l-1)$ and $A(d, l-1)$ imply $A(d,l)$.
Thus (b) follows by double induction on $d$ and $l$.

\medskip
(c) follows from (a) and (b). 
\end{proof}

The uniqueness of the symmetric multi-linear map follows from (b) in the previous claim.
For $(\overline{D}_1, \ldots, \overline{D}_d) \in \aNef_{C^0}(X)_{\RR} \times \cdots \times \aNef_{C^0}(X)_{\RR}$, we define
$\alpha(\overline{D}_1, \ldots, \overline{D}_d)$ to be
\addtocounter{Claim}{1}
\begin{equation}
\label{eqn:prop:intersection:Div:nef:C:0:1}
\alpha(\overline{D}_1, \ldots, \overline{D}_d) := \frac{1}{d!} \sum_{I \subseteq \{ 1, \ldots, d\}} (-1)^{d - \#(I)} \avol\left(\sum_{i \in I} \overline{D}_i \right).
\end{equation}

\begin{Claim}
\begin{enumerate}
\renewcommand{\labelenumi}{(\roman{enumi})}
\item
$\alpha$ is symmetric and
\[
\alpha(a \overline{D}_1 + b \overline{D}'_1, \overline{D}_2, \ldots, \overline{D}_d) = a \alpha(\overline{D}_1,  \overline{D}_2, \ldots, \overline{D}_d)
+ b \alpha(\overline{D}'_1, \overline{D}_2, \ldots, \overline{D}_d)
\]
holds for $a, b \in \RR_{\geq 0}$ and $\overline{D}_1, \overline{D}'_1, \overline{D}_2, \ldots, \overline{D}_d \in \aNef_{C^0}(X)_{\RR}$
with 
\[
D_1, D'_1, D_2, \ldots, D_d
\]
ample.

\item
If $\overline{A}_{1,1},\overline{A}_{1,-1}, \ldots, \overline{A}_{d,1},\overline{A}_{d,-1},\overline{B}_{1,1},\overline{B}_{1,-1},\ldots,\overline{B}_{d,1},\overline{B}_{d,-1}
\in \aNef_{C^0}(X)_{\RR}$,
\[
A_{1,1}, A_{1,-1}, \ldots, A_{d,1}, A_{d,-1}, B_{1,1}, B_{1,-1},\ldots,B_{d,1},B_{d,-1}
\]
are ample
and $\overline{A}_{i,1} - \overline{A}_{i,-1} = \overline{B}_{i,1} - \overline{B}_{i,-1}$ for all $i=1,\ldots,d$, then
\[
\sum_{\epsilon_1, \ldots, \epsilon_d \in \{ \pm 1 \}} \epsilon_1 \cdots \epsilon_d \alpha(\overline{A}_{1,\epsilon_1}, \ldots, \overline{A}_{d,\epsilon_d}) =
\sum_{\epsilon_1, \ldots, \epsilon_d \in \{ \pm 1 \}} \epsilon_1 \cdots \epsilon_d \alpha(\overline{B}_{1,\epsilon_1}, \ldots, \overline{B}_{d,\epsilon_d}).
\]
\end{enumerate}
\end{Claim}

\begin{proof}
(i) Clearly $\alpha$ is symmetric.
By Theorem~\ref{thm:cont:upper:envelope}, for any $\epsilon > 0$,
there are non-negative $F_{\infty}$-invariant continuous functions $u_1, u'_1, u_2, \ldots, u_d$ such that
\[
\Vert u_1 \Vert_{\sup} \leq \epsilon, \Vert u'_1 \Vert_{\sup} \leq \epsilon,  \Vert u_2 \Vert_{\sup} \leq \epsilon, \ldots, \Vert u_d \Vert_{\sup} \leq \epsilon
\]
and that 
$\overline{D}_1(\epsilon) := \overline{D}_1 + (0, u_1),$ $\overline{D}'_1(\epsilon)  := \overline{D}'_1 + (0, u'_1),$ $\overline{D}_2(\epsilon)  := \overline{D}_2 + (0, u_2),$ 
$\ldots,$ $\overline{D}_d(\epsilon)  := \overline{D}_d + (0, u_d)$
are elements of $\aNef_{C^{\infty}}(X)_{\RR}$.
Then, by virtue of Claim~\ref{claim:prop:intersection:Div:nef:C:0:1},
\begin{multline*}
\alpha(a \overline{D}_1(\epsilon)+ b \overline{D}'_1(\epsilon), \overline{D}_2(\epsilon), \ldots, \overline{D}_d(\epsilon)) \\
= a \alpha(\overline{D}_1(\epsilon),  \overline{D}_2(\epsilon), \ldots, \overline{D}_d(\epsilon))
+ b \alpha(\overline{D}'_1(\epsilon), \overline{D}_2(\epsilon), \ldots, \overline{D}_d(\epsilon)).
\end{multline*}
Thus, using the continuity of $\avol$, we have the assertion of (i).

\medskip
(ii) We would like to show the following assertion by induction on $l$:
if $\overline{A}_{1,1},$  $\overline{A}_{1,-1},$  $\ldots,$ $\overline{A}_{l,1},$ $\overline{A}_{l,-1},$ $\overline{B}_{1,1},$ $\overline{B}_{1,-1},$ $\ldots,$ $\overline{B}_{l,1},$ 
$\overline{B}_{l,-1},$ $\overline{D}_{l+1},$ $\ldots,$ $\overline{D}_d
\in \aNef_{C^0}(X)_{\RR}$,
\[
A_{1,1}, A_{1,-1}, \ldots, A_{l,1}, A_{l,-1}, B_{1,1}, B_{1,-1},\ldots,B_{l,1},B_{l,-1}, D_{l+1},\ldots, D_d
\]
are ample
and $\overline{A}_{i,1} - \overline{A}_{i,-1} = \overline{B}_{i,1} - \overline{B}_{i,-1}$ for all $i=1,\ldots,l$, then
\begin{multline*}
\sum_{\epsilon_1, \ldots, \epsilon_l \in \{ \pm 1 \}} \epsilon_1 \cdots \epsilon_l \alpha(\overline{A}_{1,\epsilon_1}, \ldots, \overline{A}_{l,\epsilon_l}, \overline{D}_{l+1}, \ldots, \overline{D}_{d}) = \\
\sum_{\epsilon_1, \ldots, \epsilon_l \in \{ \pm 1 \}} \epsilon_1 \cdots \epsilon_l \alpha(\overline{B}_{1,\epsilon_1}, \ldots, \overline{B}_{l,\epsilon_l},\overline{D}_{l+1}, \ldots, \overline{D}_{d}).
\end{multline*}
First we consider the case where $l=1$.
As $\overline{A}_{1,1} + \overline{B}_{1,-1}  = \overline{A}_{1,-1}  + \overline{B}_{1,1}$, by (i),
we have
\begin{multline*}
\alpha(\overline{A}_{1,1},\overline{D}_{2}, \ldots, \overline{D}_{d}) + \alpha(\overline{B}_{1,-1},\overline{D}_{2}, \ldots, \overline{D}_{d}) \\
=
\alpha(\overline{A}_{1,-1},\overline{D}_{2}, \ldots, \overline{D}_{d}) + \alpha(\overline{B}_{1,1},\overline{D}_{2}, \ldots, \overline{D}_{d}),
\end{multline*}
as required. In general, by using the case $l=1$,
\begin{multline*}
\alpha(\overline{A}_{1,\epsilon_1},\ldots,\overline{A}_{l,\epsilon_l}, \overline{A}_{l+1,1}, \overline{D}_{l+2}, \ldots, \overline{D}_{d}) \\
\hspace{-5em} - \alpha(\overline{A}_{1,\epsilon_1},\ldots,\overline{A}_{l,\epsilon_l}, \overline{A}_{l+1,-1},\overline{D}_{l+2}, \ldots, \overline{D}_{d}) \\
\hspace{5em} =
\alpha(\overline{A}_{1,\epsilon_1},\ldots,\overline{A}_{l,\epsilon_l}, \overline{B}_{l+1,1}, \overline{D}_{l+2}, \ldots, \overline{D}_{d}) \\
- \alpha(\overline{A}_{1,\epsilon_1},\ldots,\overline{A}_{l,\epsilon_l}, \overline{B}_{l+1,-1},\overline{D}_{l+2}, \ldots, \overline{D}_{d}).
\end{multline*}
Thus, by the hypothesis of induction, if we set
\[
\begin{cases}
A = \sum_{\epsilon_1, \ldots, \epsilon_{l+1} \in \{ \pm 1 \}} \epsilon_1 \cdots \epsilon_{l+1} \alpha(\overline{A}_{1,\epsilon_1}, \ldots, \overline{A}_{l+1,\epsilon_{l+1}}, \overline{D}_{l+2}, \ldots, \overline{D}_{d}), \\
B = \sum_{\epsilon_1, \ldots, \epsilon_{l+1} \in \{ \pm 1 \}} \epsilon_1 \cdots \epsilon_{l+1} \alpha(\overline{B}_{1,\epsilon_1}, \ldots, \overline{B}_{l+1,\epsilon_{l+1}},\overline{D}_{l+2}, \ldots, \overline{D}_{d}),
\end{cases}
\]
then
{\footnotesize
\begin{align*}
A & =  \sum \epsilon_1 \cdots \epsilon_{l} \left(\alpha(\overline{A}_{1,\epsilon_1}, \ldots, \overline{A}_{l+1,1}, \overline{D}_{l+2}, \ldots, \overline{D}_{d}) 
- \alpha(\overline{A}_{1,\epsilon_1}, \ldots, \overline{A}_{l+1,-1}, \overline{D}_{l+2}, \ldots, \overline{D}_{d}) \right) \\
& = \sum \epsilon_1 \cdots \epsilon_{l} \left(\alpha(\overline{A}_{1,\epsilon_1}, \ldots, \overline{B}_{l+1,1}, \overline{D}_{l+2}, \ldots, \overline{D}_{d}) 
- \alpha(\overline{A}_{1,\epsilon_1}, \ldots, \overline{B}_{l+1,-1}, \overline{D}_{l+2}, \ldots, \overline{D}_{d}) \right) \\
& = \sum \epsilon_1 \cdots \epsilon_{l} \left(\alpha(\overline{B}_{1,\epsilon_1}, \ldots, \overline{B}_{l+1,1}, \overline{D}_{l+2}, \ldots, \overline{D}_{d}) 
- \alpha(\overline{B}_{1,\epsilon_1}, \ldots, \overline{B}_{l+1,-1}, \overline{D}_{l+2}, \ldots, \overline{D}_{d}) \right) \\
& = B.
\end{align*}
}
\end{proof}

Note that any element of $\aDiv^{\Nef}_{C^0}(X)_{\RR}$ can be written by a form $\overline{A} - \overline{B}$ such that
$\overline{A}, \overline{B} \in \aNef_{C^0}(X)_{\RR}$ and $A, B$ are ample. Thus
the existence of the symmetric multi-linear map follows from the above claim.
\end{proof}

\query{Add Remark~6.4.3\\ (21/April/2010)}
\begin{Remark}
\label{rem:nef:and:big:self:intersection}
By our construction,
$\avol(\overline{D}) = \adeg(\overline{D}^{d})$ for $\overline{D} \in \aNef_{C^0}(X)_{\RR}$.
In particular, $\overline{D}$ is big if and only if
$\adeg(\overline{D}^{d}) > 0$.
This is however a non-trivial fact for $\overline{D} \in \aNef_{C^{\infty}}(X)_{\RR}$
(cf. \cite[Corollary~5.5]{MoCont} and Claim~\ref{claim:prop:intersection:Div:nef:C:0:1}).
\end{Remark}

\subsection{Asymptotic multiplicity}
\setcounter{Theorem}{0}
\label{subsec:asym:mult}
\query{Add explanations of the  multiplicity of Cartier divisors (28/April/2010)}%
First we recall the multiplicity of Cartier divisors.
Let $(R,\mathfrak{m})$ be a $d$-dimensional noetherian local domain with $d \geq 1$.
For a non-zero element $a$ of $R$, we denote 
the multiplicity of a local ring $(R/aR, \mathfrak{m}(R/aR))$ by $\mult_{\mathfrak{m}}(a)$,
that is,
\[
\mult_{\mathfrak{m}}(a) := 
\begin{cases} {\displaystyle \lim_{n\to\infty} \frac{ \length_R((R/aR)/\mathfrak{m}^{n+1}(R/aR))}{n^{d-1}/(d-1)!}} & \text{if $a \not\in R^{\times}$}, \\
0 & \text{if $a \in R^{\times}$}.
\end{cases}
\]
Note that $\mult_\mathfrak{m}(a) \in \ZZ_{\geq 0}$.
Moreover, if $R$ is regular, then 
\[
\mult_\mathfrak{m}(a) = \max \{ i \in \ZZ_{\geq 0} \mid a \in \mathfrak{m}^i \}.
\]

Let $a$ and $b$ be non-zero elements of $R$.
By applying  \cite[Theorem~14.6]{MatsumuraComRingJ} to the following exact sequence:
\[
0 \longrightarrow R/aR \overset{\times b}{\longrightarrow}  R/abR \longrightarrow R/bR \longrightarrow 0,
\]
we can see that
\[
\mult_\mathfrak{m}(ab) = \mult_\mathfrak{m}(a) +\mult_\mathfrak{m}(b).
\]

Let $K$ be the quotient field of $R$. For $\alpha \in K^{\times}$, we set
$\alpha = a/b$ ($a, b \in R \setminus \{ 0 \}$).
Then $\mult_\mathfrak{m}(a) - \mult_\mathfrak{m}(b)$ does not depend on the expression $\alpha = a/b$.
Indeed, if $\alpha = a/b = a'/b'$, then, by the previous formula,
\[
\mult_\mathfrak{m}(a) + \mult_\mathfrak{m}(b') = \mult_\mathfrak{m}(ab') = \mult_\mathfrak{m}(a'b) = \mult_\mathfrak{m}(a') + \mult_\mathfrak{m}(b).
\]
Thus we define $\mult_\mathfrak{m}(\alpha)$ to be 
\[
\mult_\mathfrak{m}(\alpha) := \mult_\mathfrak{m}(a) - \mult_\mathfrak{m}(b).
\]
Note that the map 
\[
\mult_\mathfrak{m} : K^{\times} \to \ZZ
\]
is a homomorphism, that is,
$\mult_\mathfrak{m}(\alpha\beta) = \mult_\mathfrak{m}(\alpha) + \mult_\mathfrak{m}(\beta)$ for $\alpha, \beta \in K^{\times}$.

\bigskip
For $x \in X$, we define a homomorphism
\[
\mult_x : \Div(X) \to \ZZ
\]
to be $\mult_x(D) := \mult_{\mathfrak{m}_x}(f_x)$,
where $\mathfrak{m}_x$ is the maximal ideal of $\OO_{X,x}$ and $f_x$ is a local equation of $D$ at $x$.
Note that this definition does not depend on the choice of the local equation $f_x$.
By abuse of notation,
the natural extension
\[
\mult_x \otimes \operatorname{id}_{\RR} : \Div(X)_{\RR} \to \RR
\]
is also denoted by $\mult_x$.

Let $\overline{D}$ be an arithmetic $\RR$-divisor of $C^0$-type.
For $x \in X$, we define $\nu_x(\overline{D})$
to be
\[
\nu_x(\overline{D}) := \begin{cases}
\inf \{ \mult_x(D + (\phi)) \mid \phi \in \aH(X, \overline{D}) \setminus \{ 0 \} \} & \text{if $\aH(X, \overline{D}) \not= \{ 0 \}$}, \\
\infty & \text{if $\aH(X, \overline{D}) = \{ 0 \}$}
\end{cases}
\]
%
%
We call $\nu_x(\overline{D})$ the {\em multiplicity at $x$} of the complete arithmetic linear series of
$\overline{D}$.
First let us see the following lemma.

\begin{Lemma}
\label{lem:sum:mu:arith:div}
Let $\overline{D}$ and $\overline{E}$ be arithmetic $\RR$-divisors of $C^0$-type.
Then we have the following:
\begin{enumerate}
\renewcommand{\labelenumi}{(\arabic{enumi})}
\item
If $\overline{D}$ is effective, then $\nu_x(\overline{D}) \leq \mult_x(D)$.

\item
$\nu_x(\overline{D} + \overline{E}) \leq \nu_x(\overline{D}) + \nu_x(\overline{E})$.

\item
If $\overline{D} \leq \overline{E}$, then $\nu_x(\overline{E}) \leq \nu_x(\overline{D}) + \mult_{x}(E - D)$.

\item
For $\phi \in \Rat(X)^{\times}$, $\nu_x(\overline{D} + \widehat{(\phi)}) = \nu_x(\overline{D})$.
\end{enumerate}
\end{Lemma}

\begin{proof}
(1) is obvious.

(2) If either $\aH(X, \overline{D}) = \{ 0 \}$ or $\aH(X, \overline{E}) = \{ 0 \}$,
then the assertion is obvious, so that we may assume that
$\aH(X, \overline{D}) \not= \{ 0 \}$ and  $\aH(X, \overline{E}) \not= \{ 0 \}$.
Let $\phi \in \aH(X, \overline{D}) \setminus \{ 0 \}$ and $\psi \in \aH(X, \overline{E}) \setminus \{ 0 \}$.
Then, as
\[
\widehat{(\phi\psi)} + \overline{E} + \overline{D} = \widehat{(\phi)} + \overline{D} + \widehat{(\psi)} + \overline{E} \geq 0,
\]
we have $\phi\psi \in \aH(X, \overline{D} + \overline{E}) \setminus \{ 0 \}$.
Thus
\[
\nu_x(\overline{D} + \overline{E}) \leq \mult_{x} ((\phi\psi) + D + E) = \mult_x((\phi) + D) + \mult_x((\psi) + E),
\]
which implies (2).

(3)  If we set $\overline{F} = \overline{E} - \overline{D}$, then, by (1) and (2),
\[
\nu_x(\overline{E}) = \nu_x(\overline{D} + \overline{F}) \leq \nu_x(\overline{D}) + \nu_x(\overline{F}) \leq \nu_x(\overline{D}) + \mult_x(F).
\]

(4) Let 
$\alpha : H^0(X, D + (\phi)) \to H^0(X, D)$
be the natural isomorphism 
given by $\alpha(\psi) = (\phi\psi)$. 
Note that $(\overline{D} + \widehat{(\phi)}) + \widehat{(\psi)} = \overline{D} + \widehat{(\alpha(\psi))}$.
Thus we have (4).
\end{proof}

\bigskip
\query{Rewrite several parts from here to the end of this subsection (06/June/2010)}%
We set
\[
N(\overline{D}) = \left\{ n \in \ZZ_{>0} \mid \aH(X, n \overline{D}) \not= \{ 0 \} \right\}.
\]
Note that $N(\overline{D})$ is a sub-semigroup of $\ZZ_{>0}$, that is,
if $n, m \in N(\overline{D})$, then $n + m \in N(\overline{D})$.
We assume that $N(\overline{D}) \not= \emptyset$.
For $x \in X$, we define $\mu_x(\overline{D})$ to be 
\[
\mu_x(\overline{D}) := \inf \left\{ \mult_x(D + (1/n)(\phi)) \mid n \in N(\overline{D}), \ \phi \in
\aH(X, n\overline{D}) \setminus \{ 0 \} \right\},
\] 
which is called the {\em asymptotic multiplicity at $x$}
of the complete arithmetic $\QQ$-linear series of $\overline{D}$.

We can see that
\[
\mu_x(\overline{D}) = \inf \left\{\left.  \frac{\nu_x(n\overline{D})}{n} \ \right|\  n \in N(\overline{D}) \right\}.
\]
Indeed, an inequality
$\mu_x(\overline{D}) \leq \nu_x(n\overline{D})/n$ for $n \in N(\overline{D})$ is obvious, so that
$\mu_x(\overline{D}) \leq \inf \left\{ \nu_x(n\overline{D})/n \mid  n \in N(\overline{D}) \right\}$. 
Moreover, for $n \in N(\overline{D})$ and $\phi \in \aH(X, n\overline{D}) \setminus \{ 0 \}$,
\[
\inf \left\{\left.  \frac{\nu_x(n\overline{D})}{n} \ \right|\  n \in N(\overline{D}) \right\}
\leq \frac{\nu_x(n\overline{D})}{n} \leq \mult_x(D + (1/n)(\phi))
\]
holds, and hence we have the converse inequality.

By the above lemma, 
\[
\nu_x((n+m)\overline{D}) \leq \nu_x(n\overline{D}) + \nu_x(m\overline{D})
\]
for all $n, m \in N(\overline{D})$.
Thus, if $\ah(\overline{D}) \not= \{ 0 \}$ (i.e., $N(\overline{D}) = \ZZ_{>0}$), then
\[
\lim_{n\to\infty} \frac{\nu_x(n\overline{D})}{n} = \inf \left\{\left.  \frac{\nu_x(n\overline{D})}{n} \ \right|\  n > 0 \right\}.
\]

\begin{Proposition}
\label{prop:mu:basic}
Let $\overline{D}$ and $\overline{E}$ be arithmetic $\RR$-divisors of $C^0$-type such that
$N(\overline{D}) \not= \emptyset$ and $N(\overline{E}) \not= \emptyset$.
Then we have the following:
\begin{enumerate}
\renewcommand{\labelenumi}{(\arabic{enumi})}
\item
$\mu_x(\overline{D} + \overline{E}) \leq \mu_x(\overline{D}) + \mu_x(\overline{E})$.

\item
If $\overline{D} \leq \overline{E}$, then $\mu_x(\overline{E}) \leq \mu_x(\overline{D}) + \mult_{x}(E - D)$.

\item
$\mu_x(\overline{D} + \widehat{(\phi)}) = \mu_x(\overline{D})$ for $\phi \in \Rat(X)^{\times}$.

\item
$\mu_x(a\overline{D}) = a \mu_x(\overline{D})$ for $a \in \QQ_{>0}$.
\end{enumerate}
\end{Proposition}

\begin{proof}
First let us see (4).
We assume that $a \in \ZZ_{>0}$.
Let $n \in N(\overline{D})$ and $\phi \in \aH(n\overline{D}) \setminus \{ 0 \}$.
Then $\phi^a \in \aH(n(a \overline{D})) \setminus \{ 0 \}$. Thus
\[
\mu_{x}(a\overline{D}) \leq \mult_{x} (aD + (1/n)(\phi^a)) = a \mult_{x} (D + (1/n)(\phi)),
\]
which yields $\mu_{x}(a\overline{D}) \leq a \mu_{x}(\overline{D})$.
Conversely let $n \in N(a\overline{D})$ and $\psi \in \aH(n(a\overline{D})) \setminus \{ 0 \}$.
Then
\[
\mu_x(\overline{D}) \leq \mult_{x} (D + (1/na)(\psi)) = (1/a) \mult_{x} (a D + (1/n)(\psi)),
\]
and hence $\mu_x(\overline{D}) \leq (1/a) \mu_{x}(a \overline{D})$.
Thus (4) follows in the case where $a \in \ZZ_{>0}$.

In general, we choose a positive integer $m$ such that $ma \in \ZZ_{>0}$. Then, by the previous observation,
\[
m \mu_{x}(a \overline{D}) = \mu_x(ma\overline{D}) = ma \mu_x(\overline{D}),
\]
as required.

By (4), we may assume that $\ah(\overline{D}) \not= 0$ and $\ah(\overline{E}) \not= 0$ in order to see
(1), (2) and (3), so that 
(1), (2) and (3) follow from (2), (3) and (4) in Lemma~\ref{lem:sum:mu:arith:div} respectively.
\end{proof}

Finally we consider the vanishing result of the asymptotic multiplicity for a nef and big arithmetic $\RR$-divisor.

\begin{Proposition}
\label{prop:vanish:mu:nef:big}
If $\overline{D}$ is a nef and big arithmetic $\RR$-divisor of $C^{0}$-type, 
then $\mu_x(\overline{D}) = 0$ for all $x \in X$.
\end{Proposition}

\begin{proof}
Step 1 (the case where $\overline{D}$ is an ample arithmetic $\RR$-divisor) : 
Note that if $\overline{D}$ is an ample arithmetic $\QQ$-divisor,
then the assertion is obvious.
By using Lemma~\ref{lem:Weil:leq:Cartier} and Lemma~\ref{lem:diff:effective:arith:div},
there are $a_1, \ldots, a_l \in \RR$ and effective arithmetic $\QQ$-divisors
\[
\overline{A}_1, \ldots, \overline{A}_l, \overline{B}_1, \ldots, \overline{B}_l
\]
of $C^{\infty}$-type
such that
\[
\overline{D} = a_1 \overline{A}_1 + \cdots + a_l \overline{A}_l - a_1 \overline{B}_1 - \cdots - a_l \overline{B}_l.
\] 
Let us choose sufficiently small arbitrary positive numbers $\delta_1, \ldots, \delta_l, \delta'_1, \ldots, \delta'_l$ such that
$a_i - \delta_i, a_i + \delta'_i \in \QQ$ for all $i$. We set
\[
\overline{D}' = (a_1 - \delta_1) \overline{A}_1 + \cdots + (a_l - \delta_l) \overline{A}_l - (a_1+\delta'_1)  \overline{B}_1 - \cdots - (a_l + \delta'_l)\overline{B}_l.
\]
Then, $\overline{D}' \leq \overline{D}$ and $\overline{D}'$ is an ample arithmetic $\QQ$-divisor by Proposition~\ref{prop:ample:RR:div}.
By (2) in Proposition~\ref{prop:mu:basic},
\[
0 \leq \mu_x(\overline{D}) \leq \mu_x(\overline{D}') +  \mult_x(D - D') =  \sum (\delta_i\mult_x(A_i) + \delta'_i\mult_x(B_i))
\]
because $\mu_x(\overline{D}') = 0$.
Therefore, 
\[
0 \leq \mu_x(\overline{D}) \leq  \sum (\delta_i\mult_x(A_i) + \delta'_i\mult_x(B_i)),
\]
and hence $\mu_x(\overline{D}) = 0$.

Step 2  (the case where $\overline{D}$ is an adequate arithmetic $\RR$-divisor) : 
In this case, there is an ample arithmetic $\RR$-divisor $\overline{A}$ and a non-negative $F_{\infty}$-invariant continuous function $\phi$ on $X(\CC)$
such that $\overline{D} = \overline{A} + (0,\phi)$. 
By (2) in Proposition~\ref{prop:mu:basic},
\[
0 \leq \mu_x(\overline{D}) \leq \mu_x(\overline{A}) =0,
\]
as required.

Step 3 (general case) : 
Let $\overline{A}$ be an ample arithmetic $\QQ$-divisor.
Since $\overline{D}$ is big, by Proposition~\ref{prop:equiv:big:regular},
there are a positive integer $m$ and $\phi \in \Rat(X)^{\times}$ such that
$\overline{A} \leq m\overline{D} + \widehat{(\phi)}$. We set $\overline{E} = m \overline{D} +  \widehat{(\phi)}$.
Then $\overline{E}$ is nef.
Moreover, for $\delta \in (0,1]$, by Proposition~\ref{prop:ample:plus:nef:ample},
$\delta \overline{A} + (1 - \delta) \overline{E}$ is adequate 
and  $\delta \overline{A} + (1 - \delta) \overline{E} \leq \overline{E}$.
Hence
\[
\mu_x(\overline{E}) \leq \mu_x(\delta \overline{A} + (1 - \delta) \overline{E}) + \delta \mult_x(E - A) \leq \delta \mult_x(E - A),
\]
which implies that $\mu_x(\overline{E}) = 0$. Therefore, using (3) and (4) in Proposition~\ref{prop:mu:basic},
\[
\mu_x(\overline{D}) = \frac{1}{m} \mu_x(m\overline{D}) =  \frac{1}{m} \mu_x(\overline{E}) = 0.
\] 
\end{proof}

\subsection{Generalized Hodge index theorem for an arithmetic $\RR$-divisor}
\setcounter{Theorem}{0}

In this subsection, let us consider the following theorem, which is an $\RR$-divisor version of \cite[Corollary~6.4]{MoCont}:

\begin{Theorem}
\label{thm:gen:Hodge:index}
Let $\overline{D}$ be an arithmetic $\RR$-divisor of $(C^0 \cap \Tpsh)$-type on $X$.
If $D$ is nef on every fiber of  $X \to \Spec(\ZZ)$ 
\rom{(}i.e., $\deg(\rest{D}{C}) \geq 0$ for all $1$-dimensional closed vertical integral subschemes $C$ on $X$\rom{)},
then $\avol(\overline{D}) \geq \adeg (\overline{D}^d)$.
\end{Theorem}

\begin{proof}
Let us begin with the following claim:

\begin{Claim}
We set $\overline{D} = (D, g)$.
If $\overline{D}$ is of $C^{\infty}$-type,
$D$ is ample \rom{(}that is, there are $a_1, \ldots, a_l \in \RR_{>0}$ and ample Cartier divisors $A_1, \ldots, A_l$ such that
$D = a_1 A_1 + \cdots + a_l A_l$\rom{)} and $dd^c([g]) + \delta_D$ is positive, then the assertion of the theorem holds.
\end{Claim}

\begin{proof}
By virtue of Proposition~\ref{prop:green:irreducible:decomp}, we can find  $F_{\infty}$-invariant locally integrable functions $h_1, \ldots, h_l$ such that
$h_i$ is an $A_i$-Green function $h_i$ of $C^{\infty}$-type for each $i$ and
$g = a_1 h_1 + \cdots + a_l h_l\ \aew$.
Let $\delta_1, \ldots, \delta_l$ be sufficiently small positive real numbers such that $a_1 + \delta_1, \ldots, a_l + \delta_l \in \QQ$.
We set 
\[
(D', g') = (a_1 + \delta_1)(A_1, h_1) + \cdots + (a_l + \delta_l)(A_l, h_l).
\]
Then  $D'$ is an ample $\QQ$-divisor and
\[
dd^c([g']) + \delta_{D'} = dd^c([g]) + \delta_{D} + \sum_{i=1}^l \delta_i (dd^c([h_i]) + \delta_{A_i}).
\]
is positive because 
$\delta_1, \ldots, \delta_l$ are sufficiently small. Therefore,
by \cite[Corollary~6.4]{MoCont}, we have $\avol(\overline{D}') \geq \adeg({\overline{D}'}^d)$, which implies the claim by using
the continuity of $\avol$ (cf. Theorem~\ref{thm:aDiv:aPic:R}).
\end{proof}

First we assume that $\overline{D}$ is of $C^{\infty}$-type.
Let $\overline{A} = (A, h)$ be an arithmetic divisor of $C^{\infty}$-type such that $A$ is ample and
$dd^c([h]) + \delta_A$ is positive.
Then, by using the same idea as in the proofs of Proposition~\ref{prop:ample:RR:div} and Proposition~\ref{prop:ample:plus:nef:ample},
we can see that $D + \epsilon A$ is ample for all $\epsilon > 0$,
Thus, by the above claim, $\avol(\overline{D} + \epsilon(A, h)) \geq \adeg ((\overline{D} + \epsilon(A, h))^d)$, and hence
the assertion follows by taking $\epsilon \to 0$.

\medskip
Finally we consider a general case. By Claim~\ref{claim:prop:intersection:Div:nef:C:0:0}, there is an ample arithmetic divisor $\overline{B}$ such that
$\overline{A} := \overline{D} + \overline{B} \in \aNef_{C^0}(X)_{\RR}$ and $A$ is ample.
Let $\epsilon$ be an arbitrary positive number. Then, by virtue of Theorem~\ref{thm:cont:upper:envelope},
we can find an $F_{\infty}$-invariant continuous function $u$ on $X(\CC)$ such that
$0 \leq u(x) \leq \epsilon$ for all $x \in X(\CC)$ and $\overline{A}'  := \overline{A} + (0, u) \in \aDiv_{C^{\infty} \cap \Tpsh}(X)_{\RR}$,
which means that
$\overline{A}' \in \aNef_{C^{\infty}}(X)_{\RR}$.
Note that
\[
\begin{cases}
{\displaystyle \adeg(\overline{D}^d) = \sum_{i=0}^d (-1)^{d-i} \binom{d}{i} \adeg(\overline{A}^i\cdot\overline{B}^{d-i})},\\
{\displaystyle \adeg({\overline{D}'}^d) = \sum_{i=0}^d (-1)^{d-i} \binom{d}{i} \adeg({\overline{A}'}^i\cdot\overline{B}^{d-i})},
\end{cases}
\]
where $\overline{D}'  := \overline{D} + (0, u)$.
By \eqref{eqn:prop:intersection:Div:nef:C:0:1},
$\adeg(\overline{A}^i\cdot\overline{B}^{d-i})$ and $\adeg({\overline{A}'}^i\cdot\overline{B}^{d-i})$ are given by an alternative sum of
volumes, so that, by the continuity of $\avol$, there is a constant $C$ such that $C$ does not depend on $\epsilon$ and that
\[
\left\vert \adeg({\overline{A}'}^i\cdot\overline{B}^{d-i}) - \adeg(\overline{A}^i\cdot\overline{B}^{d-i}) \right\vert \leq C\epsilon
\]
for all $i=0, \ldots, d$, and hence
\[
\left\vert \adeg({\overline{D}'}^d) - \adeg(\overline{D}^d) \right\vert \leq 2^d C\epsilon.
\]
On the other hand, by the continuity of $\avol$ again,
there is a constant $C'$ such that $C'$ does not depend on $\epsilon$ and that
\[
\left\vert \avol(\overline{D}') - \avol(\overline{D}) \right\vert \leq C'\epsilon.
\]
Therefore, by using the previous case,
\begin{align*}
\avol(\overline{D}) - \adeg(\overline{D}^d) & \geq \left(\avol(\overline{D}') - C' \epsilon\right) - \left(\adeg({\overline{D}'}^d) + 2^d C \epsilon\right) \\
& = \left(\avol(\overline{D}')  - \adeg({\overline{D}'}^d)\right) - (C' + 2^dC)\epsilon \geq - (C' + 2^dC)\epsilon.
\end{align*}
Thus the theorem follows because $\epsilon$ is an arbitrary positive number.
\end{proof}

\renewcommand{\theTheorem}{\arabic{section}.\arabic{Theorem}}
\renewcommand{\theClaim}{\arabic{section}.\arabic{Theorem}.\arabic{Claim}}
\renewcommand{\theequation}{\arabic{section}.\arabic{Theorem}.\arabic{Claim}}

\section{Limit of nef arithmetic $\RR$-divisors on arithmetic surfaces}
Let $X$ be a regular projective arithmetic surface and let $\TT$ be a type for Green functions on $X$ such that
$\Tpsh$ is a subjacent type of $\TT$.
The purpose of this section is to prove the following theorem.

\begin{Theorem}
\label{thm:limit:nef:div:surface}
Let $\{ \overline{M}_n = (M_n, h_n) \}_{n=0}^{\infty}$ be a sequence of
nef arithmetic $\RR$-divisors on $X$ with the following properties:
\begin{enumerate}
\renewcommand{\labelenumi}{(\alph{enumi})}
\item
There is an arithmetic divisor $\overline{D} = (D, g)$ of $\TT$-type such that $g$ is of upper bounded type and that
$\overline{M}_n \leq \overline{D}$ for all $n \geq 1$.

\item
There is a proper closed subset $E$ of $X$ such that $\Supp(D) \subseteq E$ and
$\Supp(M_n) \subseteq E$ for all $n \geq 1$.

\item
$\lim_{n\to\infty} \mult_{C}(M_n)$ exists for all $1$-dimensional closed integral  subschemes $C$ on $X$.

\item
$\limsup_{n\to\infty} (h_n)_{\rm can}(x)$ exists in $\RR$ for all $x \in X(\CC) \setminus E(\CC)$.
\end{enumerate}
Then there is a nef arithmetic $\RR$-divisor $\overline{M} = (M, h)$ on $X$ such that $\overline{M} \leq \overline{D}$,
\[
M = \sum_{C}\left( \lim_{n\to\infty} \mult_{C}(M_n)\right) C
\]
and that $\rest{h_{\rm can}}{X(\CC) \setminus E(\CC)}$ is the upper semicontinuous regularization of the function
given by 
$x \mapsto  \limsup_{n\to\infty}(h_n)_{\rm can}(x)$
over $X(\CC) \setminus E(\CC)$.
Moreover,
\[
\limsup_{n\to\infty} \adeg\left( \rest{\overline{M}_n}{C} \right) \leq \adeg\left( \rest{\overline{M}}{C} \right)
\]
holds for all $1$-dimensional closed integral subschemes $C$ on $X$.
\end{Theorem}

\begin{proof}
Let $C_1, \ldots, C_l$ be $1$-dimensional irreducible components of $E$.
Then there are $a_1, \ldots, a_l, a_{n1}, \ldots, a_{nl} \in \RR$ such that
\[
D = a_1 C_1 + \cdots + a_l C_l\quad\text{and}\quad M_n = a_{n1} C_1 + \cdots + a_{nl}C_l.
\]
We set $p_i = \lim_{n\to\infty} a_{ni}$ for $i=1, \ldots, l$ and
$M = p_1 C_1 + \cdots + p_l C_l$.

Let $U$ be a Zariski open set of $X$ over which we can find local equations $\phi_1, \ldots, \phi_l$ of
$C_1, \ldots, C_l$ respectively. Let
\[
h_n = u_n + \sum_{i=1}^l (-a_{ni}) \log  \vert \phi_i \vert^2 \ \aew\quad\text{and}\quad
g = v +  \sum_{i=1}^l (-a_{i}) \log  \vert \phi_i \vert^2\ \aew
\]
be the local expressions of $h_n$ and $g$ with respect to $\phi_1, \ldots, \phi_l$,
where $u_n \in \Tpsh_{\RR}$ and $v$ is locally bounded above.

\begin{Claim}
$\{ u_n \}_{n=0}^{\infty}$ is locally uniformly bounded above, that is, for each point $x \in U(\CC)$,
there are an open neighborhood $V_x$ of $x$ and a constant $M_x$ such that
$u_n(y) \leq M_x$ for all $y \in V_x$ and $n \geq 0$.
\end{Claim}

\begin{proof}
Since $h_n \leq g\ \aew$,  
we have 
\[
u_n \leq v - \sum_{i=1}^n (a_i - a_{ni}) \log \vert \phi_i \vert^2\ \aew
\]
over $U(\CC)$.
If $x \not\in C_1(\CC) \cup \cdots \cup C_l(\CC)$, then
$\phi_i(x) \not= 0$ for all $i$. Thus, as 
\[
v- \sum_{i=1}^n (a_i - a_{ni}) \log \vert \phi_i \vert^2
\]
is locally bounded  above, 
the assertion follows from
Lemma~\ref{lem:fqpssh:ineq:ae}.

Next we assume that $x \in C_1(\CC) \cup \cdots \cup C_l(\CC)$.
Clearly we may assume $x \in C_1(\CC)$.
Note that $C_i(\CC) \cap C_j(\CC) = \emptyset$ for $i \not= j$.
Thus $\phi_1(x) = 0$ and $\phi_i(x) \not= 0$ for all $i \geq 2$.
Therefore, we can find an open neighborhood $V_x$ of $x$ and
a constant $M'_x$ such that $\vert \phi_1 \vert < 1$ on $V_x$ and
\[
u_n \leq M'_x - (a_1 - a_{n1}) \log \vert \phi_1 \vert^2\quad\aew
\]
over $V_x$ for all $n \geq 1$. 
Moreover, we can also find a positive constant $M''$ such that
$a_1 - a_{n1} \leq M''$ for all $n \geq 1$, so that
\[
u_n \leq M'_x - M'' \log \vert \phi_1 \vert^2\quad\aew
\]
holds over $V_x$.
Thus the claim follows from
Lemma~\ref{lem:subharmonic:local:bound}. 
\end{proof}

We set $u(x) := \limsup_{n\to\infty} u_n(x)$ for $x \in U(\CC)$. Note that $u(x) \in \{-\infty\} \cup \RR$.
Let $\tilde{u}$ be the upper semicontinuous regularization of $u$.
Then, as $u_n$ is subharmonic for all $n \geq 1$, by the above claim,
$\tilde{u}$ is also subharmonic on $U(\CC)$  (cf. Subsection~\ref{subsec:pluri:subharmonic}). 

\begin{Claim}
\label{claim:thm:limit:nef:div:surface:2}
$\tilde{u}(x) \not= -\infty$ for all $x \in U(\CC)$.
\end{Claim}

\begin{proof}
If $x \not\in C_1(\CC) \cup \cdots \cup C_l(\CC) = E(\CC)$,
then $\phi_i(x) \not= 0$ for all $i$.
Note that  $\limsup_{n\to\infty} (h_n)_{\rm can}(x)$ exists in $\RR$ and that
\[
(h_n)_{\rm can}(x) =  u_n(x) + \sum_{i=1}^l (-a_{ni}) \log  \vert \phi_i(x) \vert^2.
\]
Thus $\limsup_{n\to\infty} u_n(x)$ exists in $\RR$ and 
\[
\limsup_{n\to\infty} u_n(x) = \limsup_{n\to\infty} (h_n)_{\rm can}(x) + \sum_{i=1}^l p_i \log  \vert \phi_i(x) \vert^2.
\]
Hence the assertion follows in this case.

Next we assume that $x \in C_1(\CC) \cup \cdots \cup C_l(\CC)$.
We may assume $x \in C_1(\CC)$.
As before, $\phi_1(x) = 0$ and $\phi_i(x) \not= 0$ for $i\geq 2$.
By using Lemma~\ref{lem:Weil:leq:Cartier},
let us choose a rational function $\psi$ and effective divisors $A$ and $B$ such that
$C_1 + (\psi) = A - B$ and $C_1 \not\subseteq \Supp(A) \cup \Supp(B)$.
We set 
\[
M'_n = M_n + a_{n1} (\psi), \quad h'_n = h_n - a_{n1} \log \vert \psi \vert^2\quad\text{and}\quad
\overline{M}'_n =(M'_n, h'_n).
\]
Then $\overline{M}'_n = \overline{M}_n + a_{n1}\widehat{(\psi)}$ and
\begin{multline*}
0 \leq \adeg(\rest{\overline{M}_n}{C_1}) = \adeg(\rest{\overline{M}'_n}{C_1}) \\
= a_{n1}\left(
\log \#(\OO_{C_1}(A)/\OO_{C_1}) - \log \# (\OO_{C_1}(B)/\OO_{C_1}) \right) \\
+
\sum_{i=2}^l a_{ni} \log \#(\OO_{C_1}(C_i)/\OO_{C_1} ) + \frac{1}{2} \sum_{y \in C_1(\CC)} (h'_n)_{\rm can}(y).
\end{multline*}
Thus we can find a constant $T$ such that
\[
\sum_{y \in C_1(\CC)} (h'_n)_{\rm can}(y) \geq T
\]
for all $n \geq 1$, which yields
\[
\sum_{y \in C_1(\CC)} \limsup_{n\to\infty} (h'_n)_{\rm can}(y) \geq \limsup_{n\to \infty} \left( \sum_{y \in C_1(\CC)} (h'_n)_{\rm can} (y)\right) \geq T.
\]
In particular, $\limsup_{n\to\infty} (h'_n)_{\rm can}(x) \not= -\infty$.
On the other hand,
\[
h'_n = u_n - a_{n1} \log \vert \phi_1 \psi \vert^2 - \sum_{i=2}^l a_{ni} \log \vert \phi_i \vert^2 \quad\aew.
\]
Note that $(\phi_1 \psi)(x) \in \CC^{\times}$. Thus 
\[
 \limsup_{n\to\infty} (u_n(x))  = \limsup_{n\to\infty} (h'_n)_{\rm can}(x) + p_{1} \log \vert (\phi_1 \psi)(x) \vert^2 + \sum_{i=2}^l p_{i} \log \vert \phi_i(x) \vert^2.
\]
Therefore we have the assertion of the claim in this case.
\end{proof}

\begin{Claim}
$\tilde{u} + \sum_{i=1}^l (-p_i) \log \vert \phi_i \vert^2$ does not depend on the choice of $\phi_1, \ldots, \phi_l$.
\end{Claim}

\begin{proof}
Let $\phi'_1, \ldots, \phi'_l$ be another local equations of $C_1, \ldots, C_l$.
Then there are $e_1, \ldots, e_l \in \OO_U^{\times}(U)$ such that
$\phi'_i = e_i \phi_i$ for all $i$.
Let $g_n = u'_n - \sum_{i=1}^l a_{ni} \log \vert \phi'_i \vert^2\ \aew$ be the local expression of $g_n$ with respect to $\phi'_1, \ldots, \phi'_l$.
Then $u'_n = u_n + \sum_{i=1}^l a_{ni} \log \vert e_i \vert^2$ by Lemma~\ref{lem:fqpssh:ineq:ae}.
Thus 
\[
\tilde{u}' = \tilde{u} + \sum_{i=1}^l p_i \log \vert e_i \vert^2,
\]
which implies that
\[
\tilde{u} + \sum_{i=1}^l (-p_i) \log \vert \phi_i \vert^2 = \tilde{u}' + \sum_{i=1}^l (-p_i) \log \vert \phi'_i \vert^2.
\]
\end{proof}

By the above claim, there is an $M$-Green function $h$ of $\Tpsh_{\RR}$-type on $X(\CC)$
such that 
\[
\rest{h}{U(\CC)} = \tilde{u} + \sum_{i=1}^l (-p_i) \log \vert \phi_i \vert^2.
\]
By our construction, $\rest{h_{\rm can}}{X(\CC) \setminus E(\CC)}$ is the upper semicontinuous regularization of the function
given by  $h^{\sharp}(x) = \limsup_{n\to\infty}(h_n)_{\rm can}(x)$
over $X(\CC) \setminus E(\CC)$.

\begin{Claim}
 $h$ is $F_{\infty}$-invariant and $h \leq g\ \aew$. 
\end{Claim}

\begin{proof}
As $\Tpsh$ is a subjacent type of $\TT$, we have $(h_n)_{\rm can} \leq g_{\rm can}$
over $X(\CC) \setminus E(\CC)$, so that
$h^{\sharp} \leq g_{\rm can}$
over $X(\CC) \setminus E(\CC)$.
Note that $h^{\sharp} = h\ \aew$ (cf. Subsection~\ref{subsec:pluri:subharmonic}). 
Thus the claim follows because $h^{\sharp}$ is $F_{\infty}$-invariant.
\end{proof}

Finally let us check that
\[
\adeg\left( \rest{\overline{M}}{C} \right) \geq \limsup_{n\to\infty} \adeg\left( \rest{\overline{M}_n}{C} \right) \geq 0
\]
holds for all $1$-dimensional closed integral subschemes $C$ on $X$.

By Lemma~\ref{lem:Weil:leq:Cartier} again,
we can choose non-zero rational functions $\psi_1, \ldots, \psi_l$ on $X$ and
effective divisors 
\[
A_1, \ldots, A_l, B_1, \ldots, B_l
\]
such that
$C_i + (\psi_i) = A_i - B_i$ for all $i$ and $C \not\subseteq \Supp(A_i) \cup \Supp(B_i)$ for all $i$.
We set 
\[
\left\{\hspace{-0.5em}
\begin{array}{lll}
M''_n = M_n + \sum_{i=1}^l a_{ni}(\psi_i),&
h''_n = h_n + \sum_{i=1}^l (-a_{ni}) \log \vert \psi_i \vert^2, &
\overline{M}''_n = (M''_n, h''_n) \\
M'' = M + \sum_{i=1}^l p_{i}(\psi_i),&
h'' = h+ \sum_{i=1}^l (-p_{i}) \log \vert \psi_i \vert^2, &
\overline{M}'' = (M'', h'') 
\end{array}\right.
\]

First we assume that $C$ is not flat over $\ZZ$.
Then
\[
\adeg(\rest{\overline{M}_n}{C}) = \adeg(\rest{\overline{M}''_n}{C})
=
\sum_{i=1}^l a_{in} \left( \log \#(\OO_C(A_i)/\OO_C) - \log \#(\OO_C(B_i)/\OO_C)\right) 
\]
and
\[
\adeg(\rest{\overline{M}}{C}) = \adeg(\rest{\overline{M}''}{C})
=
\sum_{i=1}^l p_{i} \left( \log \#(\OO_C(A_i)/\OO_C) - \log \#(\OO_C(B_i)/\OO_C)\right).
\]
Thus
\[
\adeg(\rest{\overline{M}}{C}) = \lim_{n\to\infty} \adeg(\rest{\overline{M}_n}{C}) \geq 0
\]

Next we assume that $C$ is  flat over $\ZZ$.
Then
\begin{multline*}
\adeg(\rest{\overline{M}_n}{C}) = \adeg(\rest{\overline{M}''_n}{C}) \\
=
\sum_{i=1}^l a_{in} \left( \log \#(\OO_C(A_i)/\OO_C) - \log \#(\OO_C(B_i)/\OO_C)\right) + \frac{1}{2} \sum_{y \in C(\CC)} (h''_n)_{\rm can}(y)
\end{multline*}
and
\begin{multline*}
\adeg(\rest{\overline{M}}{C}) = \adeg(\rest{\overline{M}''}{C}) \\
=
\sum_{i=1}^l p_{i} \left( \log \#(\OO_C(A_i)/\OO_C) - \log \#(\OO_C(B_i)/\OO_C)\right) + \frac{1}{2} \sum_{y \in C(\CC)} (h'')_{\rm can}(y).
\end{multline*}
Let us consider a Zariski open set $U$ of $X$ with $C \cap U \not= \emptyset$.
Let 
\[
h_n = u_n + \sum (-a_{ni}) \log \vert \phi_i \vert^2\ \aew\quad\text{and}\quad
h = \tilde{u} + \sum (-p_i) \log \vert \phi_i \vert^2\ \aew
\]
be
the local expressions of $h_n$ and $h$ as before. Then
\[
h''_n = u_n + \sum (-a_{ni}) \log \vert \phi_i \psi_i \vert^2\ \aew\quad\text{and}\quad
h'' = \tilde{u} + \sum (-p_i) \log \vert \phi_i \psi_i\vert^2\ \aew.
\]
Moreover, $(\phi_i \psi_i)(y) \in \CC^{\times}$ for all $y \in C(\CC)$ and $i$.
Thus
\[
\limsup_{n\to\infty} (h''_n)_{\rm can}(y) \leq (h'')_{\rm can}(y).
\]
Therefore,
\[
\limsup_{n\to\infty} \sum_{y \in C(\CC)} (h''_n)_{\rm can}(y) \leq \sum_{y \in C(\CC)} \limsup_{n\to\infty} (h''_n)_{\rm can}(y) \leq \sum_{y \in C(\CC)} (h'')_{\rm can}(y),
\]
which yields
\[
0 \leq \limsup_{n\to\infty} \adeg(\rest{\overline{M}_n}{C}) \leq \adeg(\rest{\overline{M}}{C}).
\]
\end{proof}

\renewcommand{\theTheorem}{\arabic{section}.\arabic{Theorem}}
\renewcommand{\theClaim}{\arabic{section}.\arabic{Theorem}.\arabic{Claim}}
\renewcommand{\theequation}{\arabic{section}.\arabic{Theorem}.\arabic{Claim}}

\section{$\sigma$-decompositions on arithmetic surfaces}
\label{sec:sigma:decomp}
Let $X$ be a regular projective arithmetic surface.
We fix an $F_{\infty}$-invariant  continuous volume form $\Phi$ on $X(\CC)$ with $\int_{X(\CC)} \Phi = 1$.
Let $\overline{D} = (D, g)$ be an effective arithmetic $\RR$-divisor of $C^0$-type on $X$. 
For a $1$-dimensional closed integral  subscheme $C$ on $X$, we set
\[
\nu_{C}(\overline{D}) := \min \left\{ \mult_{C}(D + (\phi)) \mid
\phi \in \aH(X, \overline{D}) \setminus \{ 0 \} \right\}
\]
as in Subsection~\ref{subsec:asym:mult}.
Moreover, we set
\[
F(\overline{D}) = \Fx(\overline{D}) = \sum_{C} \nu_{C}(\overline{D})C \quad\text{and}\quad
M(\overline{D}) =\Mv(\overline{D}) = D - \Fx(\overline{D}).
\]
Let $V(\overline{D})$ be the complex vector space generated by $\aH(X, \overline{D})$ in $H^0(X, D) \otimes_{\ZZ} \CC$,
that is, $V(\overline{D}) := \langle \aH(X, \overline{D}) \rangle_{\CC}$.

\begin{Lemma}
$\dist(V(\overline{D}); g)$ is $F_{\infty}$-invariant.
\end{Lemma}

\begin{proof}
First of all, note that, for $\phi \in \Rat(X)$, $F_{\infty}^*(\phi) = \overline{\phi}$ as a function on $X(\CC)$.
Let us see $\langle \phi, \psi \rangle_{g} \in \RR$ for all $\phi, \psi \in  \langle \aH(X, \overline{D}) \rangle_{\RR}$.
Indeed,
\begin{align*}
\langle \phi, \psi \rangle_{g} & = \int_{X(\CC)} \phi \bar{\psi} \exp(-g) \Phi = -\int_{X(\CC)} F_{\infty}^*(\phi \bar{\psi} \exp(-g) \Phi) \\
& = -\int_{X(\CC)} F_{\infty}^*(\phi)  F_{\infty}^*(\bar{\psi}) F_{\infty}^*(\exp(-g))F_{\infty}^*(\Phi) \\
& = \int_{X(\CC)} \bar{\phi} \psi \exp(-g) \Phi =
\langle \psi, \phi \rangle_{g} = \overline{\langle \phi, \psi \rangle_{g}}.
\end{align*}
Thus $\langle \phi, \psi \rangle_{g}$ yields an inner product of  $\langle \aH(X, \overline{D}) \rangle_{\RR}$, so that
let $\phi_1, \ldots, \phi_N$ be an orthonormal basis of  $\langle \aH(X, \overline{D}) \rangle_{\RR}$ over $\RR$.
These give rise to an orthonormal basis of $\langle \aH(X, \overline{D}) \rangle_{\CC}$. Therefore,
\[
\dist(V(\overline{D}); g) = \vert \phi_1 \vert^2_g + \cdots + \vert \phi_N \vert^2_g.
\]
Note that $F_{\infty}^*(\vert \phi_i \vert_g) = \vert \bar{\phi}_i \vert_g = \vert \phi_i \vert_g$, and hence the lemma follows.
\end{proof}

Here we define $g_{F(\overline{D})}$, $g_{M(\overline{D})}$, $\overline{M}(\overline{D})$ and
$\overline{F}(\overline{D})$ as follows:
\[
\begin{cases}
g_{F(\overline{D})} = - \log \dist\left( V(\overline{D}); g\right),&
g_{M(\overline{D})} = g - g_{F(\overline{D})} = g +  \log \dist\left( V(\overline{D}); g\right), \\
\overline{M}(\overline{D}) = \left(M(\overline{D}), g_{M(\overline{D})}\right),&
\overline{F}(\overline{D}) = \left(F(\overline{D}), g_{F(\overline{D})}\right).
\end{cases}
\]
Let us check the following proposition:

\begin{Proposition}
\label{prop:decomp:surface:no:limit}
\begin{enumerate}
\renewcommand{\labelenumi}{(\arabic{enumi})}
\item
$\aH(X, \overline{D}) \subseteq \aH(X, \overline{M}(\overline{D}))$.

\item
$g_{M(\overline{D})}$ is an  $M(\overline{D})$-Green function of $(C^{\infty} \cap \Tpsh)$-type on $X(\CC)$.

\item
$g_{F(\overline{D})}$ is an $F(\overline{D})$-Green function of $(C^0 - C^{\infty} \cap \Tpsh)$-type  over $X(\CC)$.

\item
$\overline{M}(\overline{D})$ is nef.
\end{enumerate}
\end{Proposition}

\begin{proof}
(1)  If $\phi \in \aH(X, \overline{D}) \setminus \{ 0 \}$,
then 
$(\phi) + D \geq F(\overline{D})$,
and hence $(\phi) + M(\overline{D}) \geq 0$.
Note that $\vert \phi \vert_g^2 = \dist(V(\overline{D}); g) \vert \phi\vert^2_{g_{M(\overline{D})}}$ for $\phi \in \aH(X, M(\overline{D}))$.
Thus, as $\Vert \phi \Vert_g \leq 1$, by Proposition~\ref{prop:decomposition:hermitian:metric},
\[
\vert \phi \vert^2_{g_{M(\overline{D})}} = \vert \phi \vert^2_g/\dist(V(\overline{D}); g) \leq \Vert \phi \Vert^2_{g} \leq 1. 
\]
Therefore, $\phi \in \aH(X, \overline{M}(\overline{D}))$.

(2), (3) 
Let us fix $x \in X(\CC)$. 
We set 
\[
\nu_x := \min \{ \mult_x(D + (\phi)) \mid \phi \in V(\overline{D}) \setminus \{ 0 \} \}.
\]
Note that $ \mult_x(D + (\phi)) =  \mult_x(D) + \operatorname{ord}_x(\phi)$.
First let us see the following claim:

\begin{Claim}
\begin{enumerate}
\renewcommand{\labelenumi}{(\alph{enumi})}
\item
If $\phi_1, \ldots, \phi_n \in V(\overline{D}) \setminus \{ 0 \}$ and $V(\overline{D})$ is generated by $\phi_1, \ldots, \phi_n$,
then
$\nu_x = \min \{ \mult_x(D + (\phi_1)), \ldots, \mult_x(D + (\phi_n)) \}$.

\item
$\nu_x = \mult_x(F(\overline{D}))$.
\end{enumerate}
\end{Claim}

\begin{proof}
(a) is obvious. 
Let us consider the natural homomorphism
\[
\langle \aH(X, \overline{D}) \rangle_{\ZZ} \otimes_{\ZZ} \OO_X \to \OO_{X}(\lfloor D \rfloor),
\]
which is surjective on $X \setminus \Supp(D)$ because $0 \leq \lfloor D \rfloor \leq D$. In particular,
\[
V(\overline{D}) \otimes_{\CC} \OO_{X(\CC)} \to \OO_{X(\CC)}(\lfloor D \rfloor),
\]
is surjective on $X(\CC) \setminus \Supp(D)(\CC)$, so that if $x \in X(\CC) \setminus \Supp(D)(\CC)$,
then $\nu_x = 0$. On the other hand, if $x \in X(\CC) \setminus \Supp(D)(\CC)$, then
$\mult_x(F(\overline{D})) = 0$ because $0 \leq F(\overline{D}) \leq D$.
Therefore, we may assume that $x \in \Supp(D)(\CC)$, so that
there is a $1$-dimensional closed integral  subscheme $C$ of $X$ with $x \in C(\CC)$.
Let $\psi_1, \ldots, \psi_n$ be all elements of $\aH(X, \overline{D})\setminus \{ 0 \}$.
Let $\eta$ be the generic point of $C$. 
Then
\[
\mult_C(F(\overline{D})) =  \min \{ \mult_C(D) + \operatorname{ord}_{\eta}(\psi_1), \ldots, \mult_C(D) + \operatorname{ord}_{\eta}(\psi_n) \}.
\]
Thus, by using (a),
\begin{align*}
\mult_x(F(\overline{D}))  & = \mult_C(F(\overline{D})) \\
& =  
\min \{ \mult_C(D) + \operatorname{ord}_{\eta}(\psi_1), \ldots, \mult_C(D) + \operatorname{ord}_{\eta}(\psi_n) \} \\
& =   \min \{ \mult_x(D + (\psi_1)), \ldots, \mult_x(D + (\psi_n)) \} = \nu_x.
\end{align*}
\end{proof}

Let $\phi_1, \ldots, \phi_N$ be an orthonormal basis
of $V(\overline{D})$ with respect to $\langle\ ,\ \rangle_g$.
Let $g = u_x + (-a) \log \vert z \vert^2\ \aew$ be a local expression of $g$ around $x$,
where $z$ is a local chart around $x$ with $z(x) = 0$.
For every $i$,  we set $\phi_i = z^{a_i} v_i$ around $x$ with
$v_i \in \OO_{X(\CC),x}^{\times}$.
Then $\vert \phi_i \vert_g^2 = \vert z \vert^{2(a_i + a)} \exp(-u_x) \vert v_i \vert^2$.
By the above claim,
\[
\nu_x = \min \{ a_1 + a,\ldots, a_N + a\} = \mult_x(F(\overline{D})).
\]
Thus
\[
\dist(V(\overline{D}); g) = \vert z \vert^{2\nu_x}\exp(-u_x)
\sum_{i=1}^N \vert z \vert^{2(a_i + a - \nu_x)} \vert v_i \vert^2.
\]
Therefore,
\[
\begin{cases}
g_{F(\overline{D})} = u_x -\log \left( \sum_{i=1}^N \vert z \vert^{2(a_i + a - \nu_x)} \vert v_i \vert^2 \right) - \nu_x \log \vert z \vert^2, \\
g_{M(\overline{D})} = \log \left( \sum_{i=1}^N \vert z \vert^{2(a_i + a - \nu_x)} \vert v_i \vert^2 \right) - (a - \nu_x)\log \vert z \vert^2.
\end{cases}
\]
Note that $\log \left( \sum_{i=1}^N \vert z \vert^{2(a_i + a - \nu_x)} \vert v_i \vert^2 \right)$ is
a  subharmonic $C^{\infty}$-function.
Thus we get (2) and (3).

(4) For $\phi \in \aH(X, \overline{D}) \setminus \{ 0 \}$ and a $1$-dimensional closed integral subscheme $C$ on $X$,
as 
\[
\mult_{C}(M(\overline{D}) + (\phi)) = \mult_{C}(D + (\phi)) - \nu_{C}(\overline{D}),
\]
there is a $\psi \in \aH(X, \overline{D}) \setminus \{ 0 \}$ such that $\mult_{C}(M(\overline{D}) + (\psi)) = 0$.
This means that 
\[
C \not\subset \Supp(M(\overline{D}) + (\psi)).
\]
Then, by Proposition~\ref{prop:decomposition:hermitian:metric},
$0 < \vert \psi \vert_{g_{M(\overline{D})}}(x) \leq 1$ for all $x \in C(\CC)$ as before. Hence
\[
\adeg \left( \rest{\overline{M}(\overline{D})}{C} \right) = \log \# \OO_C((\psi) + M(\overline{D}))/\OO_C - \sum_{x \in C(\CC)} \log \vert \psi \vert_{g_{M(\overline{D})}}(x) \geq 0.
\]
\end{proof}

For $n \geq 1$, we set
\[
\begin{cases}
{\displaystyle M_n(\overline{D}) := \frac{1}{n} M(n\overline{D})}, & {\displaystyle g_{M_n(\overline{D})} := \frac{1}{n} g_{M(n\overline{D})}}, \\
\\
{\displaystyle F_n(\overline{D}) := \frac{1}{n} F(n\overline{D})}, & {\displaystyle g_{F_n(\overline{D})} := \frac{1}{n} g_{F(n\overline{D})}}.
\end{cases}
\]
In addition,
\[
\overline{M}_n(\overline{D}) := \left(M_n(\overline{D}), g_{M_n(\overline{D})}\right)\quad\text{and}\quad
\overline{F}_n(\overline{D}) := \left(F_n(\overline{D}),g_{F_n(\overline{D})}\right).
\]
Then we have the following proposition, which guarantees a decomposition
\[
\overline{D} = \overline{M}_{\infty}(\overline{D}) + \overline{F}_{\infty}(\overline{D})
\]
as described in the proposition.
This decomposition $\overline{D} = \overline{M}_{\infty}(\overline{D}) + \overline{F}_{\infty}(\overline{D})$ is called
the {\em $\sigma$-decomposition of $\overline{D}$}. Moreover,   $\overline{M}_{\infty}(\overline{D})$ \rom{(}resp. $\overline{F}_{\infty}(\overline{D})$\rom{)}
is called the {\em asymptotic movable part} \rom{(}resp. the {\em asymptotic fixed part}\rom{)} of $\overline{D}$.

\begin{Proposition}
\label{prop:sigma:decomp:surface}
There is a nef arithmetic $\RR$-divisor $\overline{M}_{\infty}(\overline{D}) = \left(M_{\infty}(\overline{D}), g_{M_{\infty}(\overline{D})}\right)$ on $X$
with the following properties:
\begin{enumerate}
\renewcommand{\labelenumi}{(\arabic{enumi})}
\item
$\mult_{C} (M_{\infty}(\overline{D})) = \lim_{n\to\infty} \mult_{C}(M_{n}(\overline{D}))$
for all $1$-dimensional closed integral  subschemes $C$ on $X$.

\item
$\left(g_{M_{\infty}(\overline{D})}\right)_{\rm can}$ is the upper semicontinuous regularization of the function given by
\[
x \mapsto \limsup_{n\to\infty} \left(g_{M_{n}(\overline{D})}\right)_{\rm can}(x)
\]
over $X(\CC) \setminus \Supp(D)(\CC)$. 
In particular,
\[
\left(g_{M_{\infty}(\overline{D})}\right)_{\rm can}(x) = \limsup_{n\to\infty} \left( \left(g_{M_{n}(\overline{D})}\right)_{\rm can}(x) \right) \quad \aew.
\]
Moreover, if $\overline{D}$ is of $C^{\infty}$-type, then
$\lim_{n\to\infty} \left( \left(g_{M_{n}(\overline{D})}\right)_{\rm can}(x) \right)$ exists.

\item
$\adeg\left( \rest{\overline{M}_{\infty}(\overline{D})}{C} \right) \geq \limsup_{n\to\infty} \adeg\left( \rest{\overline{M}_n(\overline{D})}{C} \right)$
holds for all $1$-dimensional closed integral subschemes $C$ on $X$.

\item
If we set $F_{\infty}(\overline{D})$, $g_{F_{\infty}(\overline{D})}$ and $\overline{F}_{\infty}(\overline{D})$ as follows:
\[
\begin{cases}
F_{\infty}(\overline{D}) := D -M_{\infty}(\overline{D}),\\
g_{F_{\infty}(\overline{D})} := g - g_{M_{\infty}(\overline{D})},\\
\overline{F}_{\infty}(\overline{D}) := \left(F_{\infty}(\overline{D}), g_{F_{\infty}(\overline{D})}\right) \ (=\overline{D}- \overline{M}_{\infty}(\overline{D})),
\end{cases}
\]
then $\mu_C(\overline{D}) = \mult_C(F_{\infty}(\overline{D}))$ for all $1$-dimensional closed integral subschemes $C$ on $X$ and
$\overline{F}_{\infty}(\overline{D})$ is an effective arithmetic $\RR$-divisor 
of $(C^{0} -\Tpsh_{\RR})$-type.
In addition, if $\overline{D}$ is of $C^{\infty}$-type, then there is a constant $e$ such that
\[
n g_{F_{\infty}(\overline{D})} \leq g_{F(n\overline{D})} + 3 \log(n+1) + e\quad\aew
\]
for all $n \geq 1$.

\item
If $\overline{D}$ is of $C^{\infty}$-type, then there is a constant $e'$ such that
\[
\ah(X, n \overline{M}_{\infty}(\overline{D})) \leq \ah(X, n\overline{D}) \leq
\ah(X, n \overline{M}_{\infty}(\overline{D})) + e' n \log(n+1)
\]
for all $n \geq 1$.
\end{enumerate}
\end{Proposition}

\begin{proof}
It is easy to see that 
\[
\mult_{C}(F((n+m)\overline{D})) \leq \mult_{C}(F(n\overline{D}))  + \mult_{C}(F(m\overline{D}))
\]
for
all $n, m \geq 1$ and $1$-dimensional closed  integral subschemes $C$. Thus 
\[
\lim_{n\to\infty} \mult_{C}(F_n(\overline{D}))
\]
exists
and  
\[
\lim_{n\to\infty} \mult_{C}(F_n(\overline{D})) = \inf_{n \geq 1} \{ \mult_{C}(F_n(\overline{D})) \}.
\]
Therefore $\lim_{n\to\infty} \mult_{C}(M_n(\overline{D}))$ exists because
$M_n(\overline{D}) = D - F_n(\overline{D})$.
Note that $\mu_C(\overline{D}) = \lim_{n\to\infty} \mult_C(F_n(\overline{D}))$ as
$\mult_C(F_n(\overline{D})) = \nu_C(\overline{D})/n$ (cf. Subsection~\ref{subsec:asym:mult}).

\begin{Claim}
Let $h$ be a $D$-Green function of $C^{\infty}$-type.
Then there is a positive constant $A$ such that, for $x \in X(\CC) \setminus \Supp(D)(\CC)$,
\[
\lim_{n\to\infty} \frac{ \log \left( \dist(V(n\overline{D}); nh)(x) \right)}{n}
\]
exists in $\RR_{\leq 0}$ and
\begin{multline*}
\lim_{n\to\infty} \frac{\log \left( \dist(V(n\overline{D}); nh)(x) \right) - \log(A(n+1)^3)}{n} \\
=
\sup_{n \geq 1} \left\{  \frac{ \log \left( \dist(V(n\overline{D}); nh)(x) \right) - \log(A(n+1)^3) }{n} \right\}.
\end{multline*}
\end{Claim}

\begin{proof}
First of all, note that $\bigoplus_{n=0}^{\infty} V(n\overline{D})$ is a graded subring of
$\bigoplus_{n=0}^{\infty} H^0(X, nD)$.
By Theorem~\ref{thm:dist:dist:dist:ineq}, there is a positive constant $A$ such that
\[
\dist(V(n\overline{D}); nh) \leq A (n+1)^3
\]
and 
\[
\frac{\dist(V(n\overline{D}); nh)}{A(n+1)^3} \cdot \frac{\dist(V(m\overline{D}); mh)}{A(m+1)^3} \leq
\frac{\dist(V((n+m)\overline{D}); (n+m)h)}{A(n+m+1)^3} 
\]
for all $n, m \geq 1$. Moreover, $\dist(V(n\overline{D}); nh)(x) \not= 0$ for $x \in X(\CC) \setminus \Supp(D)(\CC)$.
Thus the claim follows.
\end{proof}
By using the Stone-Weierstrass theorem, for a positive number $\epsilon$, we can find continuous functions
$u$ and $v$ with the following properties:
\[
\begin{cases}
u \geq 0, \ \Vert u \Vert_{\sup} \leq \epsilon, \ \text{$h := g + u$ is of $C^{\infty}$-type}, \\
v \geq 0, \ \Vert v \Vert_{\sup} \leq \epsilon, \ \text{$h' := g - v$ is of $C^{\infty}$-type}.
\end{cases}
\]
By Lemma~\ref{lem:dist:comp},
\[
\exp(-n\epsilon) \dist(V(n\overline{D}); nh') \leq \dist(V(n\overline{D}); ng) \leq \exp(n\epsilon) \dist(V(n\overline{D}); nh).
\]
Thus, by the above claim,  for $x \in X(\CC) \setminus \Supp(D)(\CC)$,
\[
\limsup_{n\to\infty} \frac{ \log \left( \dist(V(n\overline{D}); ng)(x) \right)}{n}
\]
exists in $\{ a \in \RR \mid a \leq \epsilon \}$.  Since $\epsilon$ is arbitrary positive number, we actually have 
\addtocounter{Claim}{1}
\begin{equation}
\label{eqn:prop:sigma:decomp:surface:1}
\limsup_{n\to\infty} \frac{ \log \left( \dist(V(n\overline{D}); ng)(x) \right)}{n} \leq 0.
\end{equation}
This observation shows that $\limsup_{n\to\infty} \left(g_{M_n(\overline{D})}\right)_{\rm can}(x)$ exists in $\RR$ for $x \in X(\CC) \setminus \Supp(D)(\CC)$.
Therefore, by Theorem~\ref{thm:limit:nef:div:surface}, there is a nef arithmetic $\RR$-divisor 
\[
\overline{M}_{\infty}(\overline{D})
=\left(M_{\infty}(\overline{D}), g_{M_{\infty}(\overline{D})}\right)
\]
satisfying (1), (2) and (3). 
Further the last assertion of (2) is a consequence of the above claim.

Let us see (4). 
Obviously $\mu_C(\overline{D}) = \mult_C(F_{\infty}(\overline{D}))$ because 
\[
\mu_C(\overline{D}) = \lim_{n\to\infty} \mult_C(F_n(\overline{D})).
\]
Note that
\[
\left(g_{F_{\infty}(\overline{D})}\right)_{\rm can}(x) = - \limsup_{n\to\infty} \frac{ \log \left( \dist(V(n\overline{D}); ng)(x) \right)}{n} \quad\aew.
\]
on $X(\CC) \setminus \Supp(D)(\CC)$.
Thus \eqref{eqn:prop:sigma:decomp:surface:1} yields
$\left(g_{F_{\infty}(\overline{D})}\right)_{\rm can}(x)  \geq 0\ \aew$.
Hence $\overline{F}_{\infty}(\overline{D})$ is effective.
Moreover, it is obvious that $g_{F_{\infty}(\overline{D})}$ is of $(C^{0} -\Tpsh_{\RR})$-type because
$g$ is of $C^{0}$-type and $g_{M_{\infty}(\overline{D})}$ is of $\Tpsh_{\RR}$-type.

We assume that $\overline{D}$ is of $C^{\infty}$-type.
By the above claim, there is a positive constant $A'$ such that
\begin{multline*}
 - \lim_{n\to\infty} \frac{ \log \left( \dist(V(n\overline{D}); ng)(x) \right)}{n}\\
  =
  \lim_{n\to\infty} \frac{ - \log \left( \dist(V(n\overline{D}); ng)(x)\right)  + \log (A'(n+1)^3)}{n} \\
  = \inf_{n \geq 1} \left\{  \frac{ - \log \left( \dist(V(n\overline{D}); ng)(x) \right) + \log(A'(n+1)^3) }{n} \right\}
\end{multline*}
on  $X(\CC) \setminus \Supp(D)(\CC)$.
Thus, for $n \geq 1$,
\[
g_{F_{\infty}(\overline{D})} \leq  \frac{ - \log \left( \dist(V(n\overline{D}); ng)(x) \right) + \log(A'(n+1)^3) }{n}\quad\aew.
\]
which implies the last assertion of (4).

Finally let us check (5). By (4), we have $\overline{M}_{\infty}(\overline{D}) \leq \overline{D}$, so that
\[
\ah(X, n \overline{M}_{\infty}(\overline{D}) ) \leq \ah(X, n\overline{D})
\]
holds for $n \geq 1$. Moreover, by (4) again,
\[
n\overline{M}_{\infty}(\overline{D}) + (0, 3\log(n+1) + \log(A')) \geq
\overline{M}(n\overline{D})
\]
for all $n \geq 1$. Thus, by using (1) in Proposition~\ref{prop:decomp:surface:no:limit},
\ifpxfont
\[
\ah\left(X, n\overline{D}\right) \leq \ah\left(X, \overline{M}(n\overline{D})\right)
 \leq \ah\left(X,n\overline{M}_{\infty}(\overline{D}) + (0, 3\log(n+1) + \log(A'))\right).
\]
\else
\begin{multline*}
\hspace{2em}\ah\left(X, n\overline{D}\right) \leq \ah\left(X, \overline{M}(n\overline{D})\right) \\
 \leq \ah\left(X,n\overline{M}_{\infty}(\overline{D}) + (0, 3\log(n+1) + \log(A'))\right).
\end{multline*}
\fi
Note that there is a positive constant $e'$ such that
\[
\ah\left(X,n\overline{M}_{\infty}(\overline{D}) + (0, 3\log(n+1) + \log(A'))\right) \leq \ah(X,n\overline{M}_{\infty}(\overline{D})) + e'n\log(n+1)
\]
for all $n \geq 1$ (cf. \cite[(3) in Proposition~2,1]{MoCont} and \cite[Lemma~1.2.2]{MoArLin}). Thus (5) follows.
\end{proof}

\renewcommand{\theTheorem}{\arabic{section}.\arabic{subsection}.\arabic{Theorem}}
\renewcommand{\theClaim}{\arabic{section}.\arabic{subsection}.\arabic{Theorem}.\arabic{Claim}}
\renewcommand{\theequation}{\arabic{section}.\arabic{subsection}.\arabic{Theorem}.\arabic{Claim}}

\section{Zariski decompositions and their properties on arithmetic surfaces}

Throughout this section, let $X$ be a regular projective arithmetic surface and
let $\TT$ be a type for Green functions on $X$.
We always assume that $\Tpsh$ is a subjacent type of $\TT$.

\subsection{Preliminaries}
\setcounter{Theorem}{0}
In this subsection, we prepare several lemmas for the proof of Theorem~\ref{thm:Zariski:decomp:arith:surface}.

\begin{Lemma}
\label{lem:subharmonic:max:Riemann:surface}
We assume that $\TT$ is either $C^0$ or $\Tpsh_{\RR}$.
Let $M$ be a $1$-equidimensional complex manifold and
let $D_1, \ldots, D_n$ be $\RR$-divisors on $M$.
Let $g_1, \ldots, g_n$ be locally integrable functions on $M$
such that $g_i$ is a $D_i$-Green functions of $\TT$-type for each $i$.
We set 
\[
g(x) = \max \{ g_1(x), \ldots, g_n(x) \}\quad(x \in M)
\]
and
\[
D = \sum_{x \in M} \max \{ \mult_{x}(D_1), \ldots, \mult_{x}(D_n) \} x.
\]
Then $g$ is a $D$-Green function of $\TT$-type.
\end{Lemma}

\begin{proof}
For $x \in M$, let $z$ be a local chart of an open neighborhood $U_x$ of $x$ with $z(x) = 0$, and
let 
\[
g_1 = u_1 - a_1 \log \vert z \vert^2\ \aew,\ \ldots, \ g_n = u_n - a_n \log \vert z \vert^2\ \aew
\]
be local expressions of $g_1, \ldots, g_n$ respectively over $U_x$,
where $a_i = \mult_{x}(D_i)$ and $u_i \in \TT(U_x)$ for $i=1, \ldots, n$.
Clearly we may assume that $a_1 = \max \{ a_1, \ldots, a_n \}$.
First of all, we have
\[
g = \max \{ u_i + (a_1 - a_i) \log \vert z \vert^2 \mid i=1, \ldots, n \} - a_1 \log \vert z \vert^2\quad\aew
\]
over $U_x$.
In addition, the value of
\[
u := \max \{ u_i + (a_1 - a_i) \log \vert z \vert^2 \mid i=1, \ldots, n \}
\]
at $y \in U_x$ is finite

First we consider the case where $\TT = \Tpsh_{\RR}$.
Then $u_1, \ldots, u_n$ are subharmonic over $U_x$, so that 
$u_i + (a_1 - a_i) \log \vert z \vert^2$ is also subharmonic over $U_x$ for every $i$.
Therefore,
$u$ is subharmonic over $U_x$.

Next let us see the case where $\TT = C^0$.
We set $I = \{ i \mid a_i = a_1\}$.
Then, shrinking $U_x$ if necessarily,
we may assume that $u_1 > u_j + (a_1 - a_j) \log \vert z \vert^2$ on $U_x$ for all $j \not\in I$.
Thus $u = \max \{ u_i \mid i \in I \}$, and hence $u$ is continuous.
\end{proof}

\begin{Lemma}
\label{lem:nef:max:arith:surface}
We assume that $\TT$ is either $C^0$ or $\Tpsh_{\RR}$.
Let 
\[
\overline{D}_1 = (D_1, g_1), \ldots, \overline{D}_n = (D_n, g_n)
\]
be arithmetic $\RR$-divisors of $\TT$-type on $X$. 
We set 
\[
\begin{cases}
\max\{ D_1, \ldots, D_n \} := \sum_{C} \max \{ \mult_{C}(D_1), \ldots, \mult_{C}(D_n) \} C, \\
\max\{ \overline{D}_1, \ldots, \overline{D}_n \} := (\max\{ D_1, \ldots, D_n \}, \max \{ g_1, \ldots, g_n \}).
\end{cases}
\]
Then we have the following:
\begin{enumerate}
\renewcommand{\labelenumi}{(\arabic{enumi})}
\item
$\max\{ \overline{D}_1, \ldots, \overline{D}_n \}$ is an arithmetic $\RR$-divisor of $\TT$-type for $D$.

\item
If $\TT = \Tpsh_{\RR}$ and
$\overline{D}_1, \ldots, \overline{D}_n$ are nef, then
$\max\{ \overline{D}_1, \ldots, \overline{D}_n \}$ is nef.
\end{enumerate}
\end{Lemma}

\begin{proof}
(1) It is obvious that $\max \{ g_1, \ldots, g_n \}$ is $F_{\infty}$-invariant, so that
(1) follows from Lemma~\ref{lem:subharmonic:max:Riemann:surface}.

(2) For simplicity, we set $D = \max \{ D_1, \ldots, D_n \}$,
$g = \max \{ g_1, \ldots, g_n \}$ and $\overline{D} = \max \{ \overline{D}_1, \ldots, \overline{D}_n \}$.
Let $C$ be a $1$-dimensional closed integral subscheme of $X$.
Let $\gamma$ be the generic point of $C$. Since the codimension of
\[
\Supp(D - D_1) \cap \cdots \cap \Supp(D - D_n)
\]
is greater than or equal to $2$, there is $i$ such that $\gamma \not\in \Supp(D - D_i)$.
By Proposition~\ref{prop:scalar:sum:Green:function},
$g- g_i$ is a $(D - D_i)$-Green function of $(\Tpsh_{\RR} -\Tpsh_{\RR})$-type and $g - g_i \geq 0\ \aew$.
Moreover, as $x \not\in \Supp(D-D_i)$ for $x \in C(\CC)$,
by Proposition~\ref{prop:scalar:sum:Green:function}, 
\[
(g-g_i)_{\rm can}(x) \geq 0.
\]
Therefore, $\adeg \left( \rest{\overline{D} - \overline{D}_i}{C} \right) \geq 0$, and hence
\[
\adeg \left( \rest{\overline{D}}{C}  \right) \geq \adeg \left( \rest{\overline{D}_i}{C} \right) \geq 0.
\]
\end{proof}

\begin{Lemma}
\label{lem:good:psh:arith:surface}
Let $(D, g)$ be an effective arithmetic $\RR$-divisor of $C^0$-type on $X$ and
let $E$ be an $\RR$-divisor on $X$ with $0 \leq E \leq D$.
Then there is an $F_{\infty}$-invariant $E$-Green function $h$ of $(C^0 \cap \Tpsh)$-type such that
\[
0 \leq (E, h) \leq (D, g).
\]
\end{Lemma}

\begin{proof}
Let $h_1$ be an $F_{\infty}$-invariant $E$-Green function of $(C^{\infty} \cap \Tpsh)$-type.
There is a constant $C_1$ such that $h_1 + C_1 \leq g \ \aew$.
We set $h = \max \{ h_1 + C_1, 0 \}$.
Then, by Lemma~\ref{lem:subharmonic:max:Riemann:surface}, $h$ is an $F_{\infty}$-invariant  $E$-Green function of $(C^0 \cap \Tpsh)$-type and
$0 \leq h \leq g \ \aew$.
\end{proof}

\subsection{The existence of Zariski decompositions}
\setcounter{Theorem}{0}

Let  $\overline{D} = (D, g)$ be an arithmetic $\RR$-divisor of $\TT$-type on $X$
such that $g$ is of upper bounded type. Let us consider 
\[
(-\infty, \overline{D}] \cap \aNef(X)_{\RR} = \{ \overline{M} \mid \text{$\overline{M}$ is nef and
$\overline{M} \leq \overline{D}$} \}.
\]
The following theorem is one of the main theorems of this paper, which guarantees the greatest element $\overline{P}$ of $(-\infty, \overline{D}] \cap \aNef(X)_{\RR}$
under the assumption $(-\infty, \overline{D}] \cap \aNef(X)_{\RR} \not= \emptyset$.
If we set $\overline{N} = \overline{D} - \overline{P}$, then we have 
a decomposition $\overline{D} = \overline{P} + \overline{N}$.
It is called
the {\em Zariski decomposition of $\overline{D}$}, and $\overline{P}$ \rom{(}resp. $\overline{N}$\rom{)}
is called the {\em positive part} \rom{(}resp. {\em negative part}\rom{)} of $\overline{D}$.

\begin{Theorem}[Zariski decomposition on an arithmetic surface]
\label{thm:Zariski:decomp:arith:surface}
If \[
(-\infty, \overline{D}] \cap \aNef(X)_{\RR} \not= \emptyset,
\]
 then
there is $\overline{P}  = (P,p) \in (-\infty, \overline{D}] \cap \aNef(X)_{\RR}$ such that
$\overline{P}$ is greatest in $(-\infty, \overline{D}] \cap \aNef(X)_{\RR}$, that is,
$\overline{M} \leq \overline{P}$ for all $\overline{M} \in (-\infty, \overline{D}] \cap \aNef(X)_{\RR}$.
Moreover, if $\overline{D}$ is of $C^0$-type, then $\overline{P}$ is also of $C^0$-type.
\end{Theorem}

\begin{proof}
For a  $1$-dimensional closed integral subscheme $C$ of $X$, we put 
\[
a(C) = \sup \{ \mult_C(M) \mid (M, g_M) \in (-\infty, \overline{D}] \cap \aNef(X)_{\RR} \}.\\
\]
We choose $\overline{M}_0 = (M_0, g_0) \in (-\infty, \overline{D}] \cap \aNef(X)_{\RR}$.
Then $\mult_{C}(M_0) \leq a(C) \leq \mult_C(D)$ by Lemma~\ref{lem:nef:max:arith:surface}.
Let $\{ C_1, \ldots, C_l \}$ be the set of all
$1$-dimensional closed integral subschemes in $\Supp(D) \cup \Supp(M_0)$.
Note that if $C \not\in \{ C_1, \ldots, C_l \}$, then
$a(C) = 0$. Thus we set $P = \sum_C a(C) C$.

\begin{Claim}
There is a sequence $\{ \overline{M}_n = (M_n, g_n) \}_{n=0}^{\infty}$ in $(-\infty, \overline{D}] \cap \aNef(X)_{\RR}$
such that $\overline{M}_n \leq \overline{M}_{n+1}$ for all $n \geq 0$  and that
\[
\lim_{n\to\infty} \mult_{C_i}(M_n) = a(C_i)
\]
for all $i=1, \ldots, n$.
\end{Claim}

\begin{proof}
For each $i$, let $\{ \overline{M}_{i,n} \}_{n=1}^{\infty}$ be a sequence in $(-\infty, \overline{D}] \cap \aNef(X)_{\RR}$
such that 
\[
\lim_{n\to\infty} \mult_{C_i}(M_{i,n}) = a(C_i).
\]
We set $\overline{M}_n = \max \left\{ \{ \overline{M}_0 \} \cup \{ \overline{M}_{i,j} \}_{1 \leq i \leq l, 1 \leq j \leq n} \right\}$ for $n \geq 1$.
By Lemma~\ref{lem:nef:max:arith:surface}, $\overline{M}_n \in (-\infty, \overline{D}] \cap \aNef(X)_{\RR}$. Moreover,
$\overline{M}_n \leq \overline{M}_{n+1}$
and
\[
\lim_{n\to\infty} \mult_{C_i}(M_{n}) = a(C_i)
\]
for all $i$.
\end{proof}

Since $\Tpsh$ is a subjacent type of $\TT$, by using Lemma~\ref{lem:fqpssh:ineq:ae},
\[
(g_0)_{\rm can} \leq \cdots \leq (g_n)_{\rm can} \leq (g_{n+1})_{\rm can} \leq \cdots \leq g_{\rm can}
\]
holds on $X(\CC) \setminus (\Supp(D) \cup \Supp(M_0))(\CC)$, which means that $\lim_{n\to\infty} (g_n)_{\rm can}(x)$ exists
for $x \in X(\CC) \setminus (\Supp(D) \cup \Supp(M_0))(\CC)$.
Therefore, by Theorem~\ref{thm:limit:nef:div:surface},
there is an $F_{\infty}$-invariant  $P$-Green function $h$ of $\Tpsh_{\RR}$-type on $X(\CC)$ such that
$(P, h) \leq \overline{D}$ and $(P, h)$ is nef.
Here we consider
\[
[(P,h), \overline{D}] \cap \aNef(X)_{\RR} = \{ (M, g_M) \mid \text{$(M, g_M)$ is nef and $(P, h) \leq (M, g_M) \leq \overline{D}$} \}.
\]
Note that $M = P$ for all $(M, g_M) \in [(P,h), \overline{D}] \cap \aNef(X)_{\RR}$.

\begin{Claim}
If $\overline{P} = (P, p)$ is the greatest element of $[(P,h), \overline{D}] \cap \aNef(X)_{\RR}$, then
$\overline{P}$ is also the greatest element of $(-\infty, \overline{D}] \cap \aNef(X)_{\RR}$.
\end{Claim}

\begin{proof}
For $(N, g_N) \in (-\infty, \overline{D}] \cap \aNef(X)_{\RR}$, we set $(M, g_M) = (\max\{ P, N \}, \max \{ h, g_N \})$.
Then 
\[
(M, g_M) \in [(P,h), \overline{D}] \cap \aNef(X)_{\RR}\quad\text{and}\quad(N, g_N) \leq (M, g_M).
\]
Thus the claim follows.
\end{proof}

By Proposition~\ref{prop:limit:Green:function:PSH},
there is a $P$-Green function $p$ of $\Tpsh_{\RR}$-type such that $p \leq g\ \aew$ and
$p_{\rm can}$
is the upper semicontinuous regularization of the function $p'$ given by
\[
p'(x) := \sup \{ (g_M)_{\rm can}(x) \mid \overline{M}  \in [(P,h), \overline{D}] \cap \aNef(X)_{\RR}\}
\]
over $X(\CC) \setminus \Supp(P)(\CC)$.
Since $(g_M)_{\rm can}$ is $F_{\infty}$-invariant on $X(\CC) \setminus \Supp(P)(\CC)$,
$p'$ is also $F_{\infty}$-invariant, and hence $p$ is $F_{\infty}$-invariant because
$p = p' \ \aew$ on $X(\CC) \setminus \Supp(P)(\CC)$ (cf. Subsection~\ref{subsec:pluri:subharmonic}).
We set $\overline{P} = (P, p)$. Then
$(P, h) \leq  \overline{P} \leq \overline{D}$ and hence
$\overline{P}$ is nef by Lemma~\ref{lem:nef:D1:leq:D2}.
In addition, $\overline{P}$ is the greatest element of $[(P,h), \overline{D}] \cap \aNef(X)_{\RR}$.

\medskip
Finally we assume that $\overline{D}$ is of $C^0$-type.
Let $e$ be the degree of $P$ on the generic fiber of $X \to \Spec(\ZZ)$.
As $\overline{P}$ is nef, we have $e \geq 0$.
Let $X(\CC) = X_1 \cup \cdots \cup X_r$ be the decomposition into connected components of $X(\CC)$.
We set $P = \sum_{i=1}^r \sum_{j} a_{ij} P_{ij}$ on $X(\CC)$,
where $P_{ij} \in X_i$ for all $i$ and $j$.
Note that $e = \sum_{j} a_{ij}$ for all $i$.
Let us fix a $C^{\infty}$-volume form $\omega_i$ on $X_i$ with $\int_{X_i} \omega_i = 1$.
Let $p_{ij}$ be a $P_{ij}$-Green function of $C^{\infty}$-type on $X_i$ such that
$dd^c([p_{ij}]) + \delta_{P_{ij}} = [\omega_i]$.
We set $p' = \sum_{i=1}^r \sum_j a_{ij} p_{ij}$.
Then $p'$ is a $P$-Green function of $C^{\infty}$-type and
\[
dd^c([p']) + \delta_{P} = \sum_{i=1}^r \left(\sum\nolimits_{j} a_{ij}\right) [\omega_i] = e \sum_{i=1}^r [\omega_i].
\]
Thus, if $e> 0$, then $dd^c([p']) + \delta_{P}$
is represented by a positive $C^{\infty}$-form $e \sum_{i=1}^r \omega_i$.
Moreover, if $e = 0$, then $dd^c([p']) + \delta_{P} = 0$.
Let us consider
\query{$\Tpsh_{\RR}$ $\Longrightarrow$ $\Tpsh$
(21/September/2010)}
\[
\left\{ \varphi \ \left|\  \begin{array}{l} \text{$\varphi$ is a $P$-Green function of $\Tpsh$-type} \\
\text{on $X(\CC)$ with
$\varphi \leq g\ \aew$} \end{array} \right\}\right..
\]
By Theorem~\ref{thm:cont:upper:envelope}, the above set has the greatest element $\tilde{p}$ modulo null functions
such that
$\tilde{p}$ is
a $P$-Green function of $(C^0 \cap \Tpsh)$-type.
Since $g$ is $F_{\infty}$-invariant, we have $F_{\infty}^*(\tilde{p}) \leq F_{\infty}^* (g) = g \ \aew$.
Moreover, by Lemma~\ref{lem:plurisubharmonic:complex:conjugation} and Lemma~\ref{lem:D:Green:F:infty}, 
$F_{\infty}^*(\tilde{p})$ is a $P$-Green function of $\Tpsh$-type.
Thus $F_{\infty}^*(\tilde{p}) \leq \tilde{p} \ \aew$, and hence
\[
\tilde{p} = F_{\infty}^*(F_{\infty}^*(\tilde{p})) \leq F_{\infty}^*(\tilde{p}) \quad \aew.
\]
Therefore, $\tilde{p}$ is $F_{\infty}$-invariant.
Note that $(P, \tilde{p})$ is nef because $p \leq \tilde{p}\ \aew$.
Hence $p = \tilde{p}\ \aew$.
\end{proof}

\subsection{Properties of Zariski decompositions}
\label{subsec:properties:Zariski:decomp}
\setcounter{Theorem}{0}

Let  $\overline{D} = (D, g)$ be an arithmetic $\RR$-divisor of $\TT$-type on $X$
such that $g$ is of upper bounded type. 
First of all, let us observe the following three properties of the Zariski decompositions:

\begin{Proposition}
\label{prop:easy:prop:Zariski:decomp}
We assume $(-\infty, \overline{D}] \cap \aNef(X)_{\RR} \not= \emptyset$.
Let $\overline{D} = \overline{P} + \overline{N}$ be the Zariski decomposition of $\overline{D}$.
Then we have the following:
\begin{enumerate}
\renewcommand{\labelenumi}{(\arabic{enumi})}
\item For a non zero rational function $\phi$ on $X$,
$\overline{D} +  \widehat{(\phi)}= (\overline{P} +  \widehat{(\phi)})+ \overline{N}$ is the Zariski decomposition of $\overline{D} + \widehat{(\phi)}$.

\item For $a \in \RR_{>0}$,
$a \overline{D} = a\overline{P} + a\overline{N}$ is the Zariski decomposition of $a\overline{D}$.
\end{enumerate}
\end{Proposition}

\begin{proof}
Note that $\pm\widehat{(\phi)}$ is nef and that
\[
\overline{D}_1 \leq \overline{D}_2 \quad\Longleftrightarrow\quad
\overline{D}_1 +  \widehat{(\phi)} \leq \overline{D}_2 +  \widehat{(\phi)}
\]
and
\[
\overline{D}_1 \leq \overline{D}_2 \quad\Longleftrightarrow\quad
a \overline{D}_1 \leq a \overline{D}_2
\]
for arithmetic $\RR$-divisors $\overline{D}_1$, $\overline{D}_2$, a non-zero rational function $\phi$ and $a \in \RR_{>0}$.
Thus the assertions of this proposition are obvious.
\end{proof}

\begin{Proposition}
\label{prop:non:empty:Zariski:decomp}
\begin{enumerate}
\renewcommand{\labelenumi}{(\arabic{enumi})}
\item 
If $\ah(X, a \overline{D}) \not= 0$ for some $a \in \RR_{>0}$,
then 
\[
(-\infty, \overline{D}] \cap \aNef(X)_{\RR} \not= \emptyset.
\]

\item
If $\overline{D}$ is of $C^0$-type and
$(-\infty, \overline{D}] \cap \aNef(X)_{\RR} \not= \emptyset$,
then $\overline{D}$ is pseudo-effective. 
\end{enumerate}
\end{Proposition}

\begin{proof}
(1)
We choose 
$\phi \in \aH(X, a\overline{D}) \setminus \{ 0 \}$.
Then $a\overline{D} + \widehat{(\phi)} \geq 0$, which implies 
$\overline{D} \geq (-1/a)  \widehat{(\phi)}$. Note that $(-1/a)  \widehat{(\phi)}$ is nef,
so that $(-1/a)  \widehat{(\phi)} \in (-\infty, \overline{D}] \cap \aNef(X)_{\RR}$, as required.

(2) Let $\overline{D} = \overline{P} + \overline{N}$ be the Zariski decomposition of $\overline{D}$ and
let $\overline{A}$ be an ample
arithmetic $\RR$-divisor. For $n \in \ZZ_{>0}$, by Proposition~\ref{prop:ample:plus:nef:ample},
$\overline{P} + (1/n) \overline{A}$ is adequate. In particular, $\avol(\overline{P} + (1/n) \overline{A}) > 0$, and hence
\[
\avol(\overline{D} + (1/n) \overline{A}) \geq \avol(\overline{P} + (1/n)\overline{A}) > 0,
\]
which shows that $\overline{D}$ is pseudo-effective.
\end{proof}

\begin{Proposition}
\label{prop:volume:P:Zariski:decomp:C:infty}
We assume that $\overline{D}$ is of $C^{\infty}$-type and $\overline{D}$ is effective.
Let $\overline{P}$ be the positive part of the Zariski decomposition of $\overline{D}$.
Then there is a constant $e'$ such that
\[
\ah(X, n \overline{P}) \leq \ah(X, n\overline{D}) \leq
\ah(X, n \overline{P}) + e' n \log(n+1)
\]
for all $n \geq 1$.
In particular, $\avol(\overline{P}) = \avol(\overline{D})$.
\end{Proposition}

\begin{proof}
The assertion is a consequence of Proposition~\ref{prop:sigma:decomp:surface} because $\overline{M}_{\infty}(\overline{D}) \leq \overline{P}$.
\end{proof}

The following theorem is also one of the main theorems of this paper.

\begin{Theorem}
\label{thm:Zariski:decomp:big}
We assume that $\overline{D}$ is of $C^{0}$-type and 
$(-\infty, \overline{D}] \cap \aNef(X)_{\RR} \not= \emptyset$.
Let $\overline{P}$ \rom{(}resp. $\overline{N}$\rom{)}
be the positive part \rom{(}resp. negative part\rom{)} of the Zariski decomposition of $\overline{D}$.
Then we have the following:
\begin{enumerate}
\renewcommand{\labelenumi}{(\arabic{enumi})}
\item
$\avol(\overline{P}) = \avol(\overline{D}) = \adeg(\overline{P}^2)$.

\item
$\adeg (\rest{\overline{P}}{C}) = 0$ for all $1$-dimensional  closed integral subschemes $C$ with $C \subseteq \Supp(N)$.

\item
If $\overline{M}$ is an arithmetic $\RR$-divisor of $\Tpsh_{\RR}$-type on $X$ such that
$0 \leq \overline{M} \leq \overline{N}$ and $\deg(\rest{\overline{M}}{C}) \geq 0$ 
for all $1$-dimensional closed  integral subschemes $C$ with $C \subseteq \Supp(N)$,
then $\overline{M} = 0$.

\item
We assume $N \not= 0$. Let $N = c_1 C_1 + \cdots + c_l C_l$ be the decomposition such that
$c_1, \ldots, c_l \in \RR_{>0}$ and $C_1, \ldots, C_l$ are $1$-dimensional closed integral subschemes on $X$.
Then the following hold:
\begin{enumerate}
\renewcommand{\labelenumii}{(\arabic{enumi}.\arabic{enumii})}
\item There are effective arithmetic divisors $(C_1, h_1), \ldots, (C_l, h_l)$ of $(C^0 \cap \Tpsh)$-type such that
$c_1(C_1, h_1) + \cdots + c_l(C_l, h_l) \leq \overline{N}$.

\item If $(C_1, k_1), \ldots, (C_l, k_l)$ are effective arithmetic divisors of $\Tpsh_{\RR}$-type such that
$\alpha_1 (C_1, k_1) + \cdots + \alpha_l (C_l, k_l) \leq \overline{N}$ for some $\alpha_1, \ldots, \alpha_l \in \RR_{>0}$,
then 
\[
(-1)^l \det \left(\adeg\left(\rest{(C_i, k_i)}{C_j}\right)\right) > 0.
\]
\end{enumerate}
\end{enumerate}
\end{Theorem}

\begin{proof}
(1) It follows from Proposition~\ref{prop:intersection:Div:nef:C:0} that
$\avol(\overline{P}) = \adeg(\overline{P}^2)$. We need to show $\avol(\overline{P}) = \avol(\overline{D})$.
If $\avol(\overline{D}) = 0$, then the assertion is obvious, so that we may assume that $\avol(\overline{D}) > 0$.

First we consider the case where $\overline{D}$ is of $C^{\infty}$-type.
We choose a positive integer $n$ and a non-zero rational function $\phi$ such that
$n \overline{D} + \widehat{(\phi)}$ is effective. By Proposition~\ref{prop:easy:prop:Zariski:decomp},
the positive part of
the Zariski decomposition $n \overline{D} + \widehat{(\phi)}$ is 
$n \overline{P} + \widehat{(\phi)}$.
Thus,
by using Proposition~\ref{prop:volume:P:Zariski:decomp:C:infty},
\[
n^2 \avol(\overline{P}) = \avol(n\overline{P}) = \avol(n \overline{P} + \widehat{(\phi)}) = \avol(n \overline{D} + \widehat{(\phi)}) =
\avol(n \overline{D}) = n^2\avol(\overline{D}),
\]
as required.

Let us consider a general case.
By the Stone-Weierstrass theorem,
there is a sequence $\{ u_n \}_{n=1}^{\infty}$ of non-negative $F_{\infty}$-invariant continuous functions such that
$\lim_{n\to\infty} \Vert u_n \Vert_{\sup} = 0$ and $\overline{D}_n := \overline{D} - (0,u_n)$ is of $C^{\infty}$-type
for every $n \geq 1$. 
By the continuity of $\avol$ (cf. Theorem~\ref{thm:aDiv:aPic:R}),
\[
\lim_{n\to\infty} \avol(\overline{D}_n) = \avol(\overline{D}).
\]
In particular, $\overline{D}_n$ is big for $n \gg 1$.
Let $\overline{P}_n$ be the positive part of the Zariski decomposition of $\overline{D}_n$.
Since $\overline{P}_n \leq \overline{D}_n \leq \overline{D}$ and $\overline{P}_n$ is nef, we have $\overline{P}_n \leq \overline{P}$, and hence
\[
\avol(\overline{D}_n) = \avol(\overline{P}_n) \leq \avol(\overline{P}) \leq \avol(\overline{D}).
\]
Thus the assertion follows by taking $n \to \infty$.

\medskip
(4.1) Before starting the proofs of (2), (3) and (4.2), let us see (4.1) first.
By Proposition~\ref{prop:green:irreducible:decomp},
there are effective arithmetic divisors $(C_1, h'_1), \ldots, (C_l, h'_l)$ of $C^0$-type
such that $c_1(C_1, h'_1) + \cdots + c_l (C_l, h'_l) = \overline{N}$.
For each $i$, by using Lemma~\ref{lem:good:psh:arith:surface}, we can find an effective arithmetic
divisor $(C_i, h_i)$ of $(C^0 \cap \Tpsh)$-type such that $(C_i, h_i) \leq (C_i, h'_i)$, as required.

\medskip
(2) We may assume $N \not= 0$.
We assume $\deg(\rest{\overline{P}}{C_i}) > 0$ for some $i$.
By (4.1), 
\[
0 \leq c_i (C_i, h_i) \leq \overline{N}.
\]
Note that if $C'$ is a $1$-dimensional closed  integral subscheme with $C' \not= C_i$, then
\[
\deg(\rest{(C_i, h_i)}{C'}) \geq 0.
\]
Thus, since  $\deg(\rest{\overline{P}}{C_i}) > 0$, we can find a sufficiently small positive number $\epsilon$ such that
$\overline{P} + \epsilon (C_i, h_i)$ is nef  and $\overline{P} + \epsilon (C_i, h_i) \leq \overline{D}$.
This is a contradiction.

\medskip
(3) Since $0 \leq \overline{M} \leq \overline{N}$, if $C'$ is a $1$-dimensional closed  integral subscheme with $C' \not\subseteq \Supp(N)$,
then $\adeg(\rest{\overline{M}}{C'}) \geq 0$. Thus $\overline{M}$ is nef, and hence $\overline{P} + \overline{M}$ is nef and
$\overline{P} + \overline{M} \leq \overline{D}$.
Therefore, $\overline{M} = 0$.

\medskip
(4.2) By Lemma~\ref{lem:negative:definite}, it is sufficient to see the following:
if $\beta_1, \ldots, \beta_l \in \RR_{\geq 0}$ and
\[
\adeg\left( \rest{(\beta_1 (C_1, k_1) + \cdots + \beta_l(C_l, k_l))}{C_i}\right) \geq 0
\]
for all $i$,
then $\beta_1 = \cdots = \beta_l = 0$.
Replacing $\beta_1, \ldots, \beta_l$ with $t \beta_1, \ldots, t \beta_l$ ($t > 0$),
we may assume that $0 \leq \beta_i \leq \alpha_i$ for all $i$.
Thus the assertion follows from (3).
\end{proof}

\begin{Theorem}[Asymptotic orthogonality of $\sigma$-decomposition]
\label{thm:asym:orthogonal}
If $\overline{D}$ is of $C^{0}$-type, effective and big, then
\[
\lim_{n\to\infty} \adeg \left( \overline{M}_n(\overline{D}) \mid F_n(\overline{D}) \right) = 0.
\]
\rom{(}For the definition of $\overline{M}_n(\overline{D})$ and $F_n(\overline{D})$, see Section~\rom{\ref{sec:sigma:decomp}}.\rom{)}
\end{Theorem}

\begin{proof}
Let us begin with the following claim:

\begin{Claim}
\label{claim:thm:asym:orthogonal:1}
$P = M_{\infty}(\overline{D})$ and $N = F_{\infty}(\overline{D})$.
\end{Claim}

\begin{proof}
First of all, note that $\overline{M}_{\infty}(\overline{D}) \leq \overline{P}$ and $\overline{F}_{\infty}(\overline{D}) \geq \overline{N}$.
Since $\overline{D}$ is effective, $(0,0) \in (-\infty, \overline{D}] \cap \aNef(X)_{\RR}$, so that
$\overline{P}$ is effective. Then, by (2) of Proposition~\ref{prop:mu:basic},
\[
\mu_{C}(\overline{D}) \leq \mu_C(\overline{P}) + \mult_C(N).
\]
Moreover, by Proposition~\ref{prop:vanish:mu:nef:big}, $\mu_C(\overline{P}) = 0$ because $\overline{P}$ is nef and big.
Thus we have
\[
\mult_C(F_{\infty}(\overline{D})) = \mu_C(\overline{D}) \leq \mult_C(N),
\]
which implies $F_{\infty}(\overline{D}) \leq N$. Therefore, $N = F_{\infty}(\overline{D})$, and hence
$P = M_{\infty}(\overline{D})$.
\end{proof}

\begin{Claim}
$\adeg\left( \rest{\overline{M}_{\infty}(\overline{D})}{C} \right) = 0$ for any $1$-dimensional closed integral subscheme $C$ with $C \subseteq \Supp(N)$.
\end{Claim}

\begin{proof}
Since $\overline{M}_{\infty}(\overline{D}) \leq \overline{P}$ and $P = M_{\infty}(\overline{D})$,
there is $\phi \in (C^{0} - \Tpsh_{\RR})(X(\CC))$ such that $\phi \geq 0$ and $\overline{P} = \overline{M}_{\infty}(\overline{D}) + (0, \phi)$.
Thus, for a $1$-dimensional closed integral subscheme $C$ with $C \subseteq \Supp(N)$, by (3) in Theorem~\ref{thm:Zariski:decomp:big},
\[
0 \leq \adeg\left( \rest{\overline{M}_{\infty}(\overline{D})}{C} \right) \leq \adeg\left( \rest{\overline{P}}{C} \right) = 0,
\]
as required.
\end{proof}

Let $C_1, \ldots, C_l$ be irreducible components of $\Supp(D)$.
We set $F_n(\overline{D}) = \sum_{i=1}^l a_{ni}C_i$ and $F_{\infty}(\overline{D}) = \sum_{i=1}^l a_i C_i$.
Then $\lim_{n\to\infty} a_{ni} = a_i$. Moreover, if we set $I = \{ i \mid a_i > 0 \}$, then
$\bigcup_{i \in I} C_i = \Supp(N)$. Therefore, by the above claim and (3) in Proposition~\ref{prop:sigma:decomp:surface},
\begin{align*}
0 & \leq \liminf_{n\to \infty} \adeg \left( \overline{M}_n(\overline{D}) \mid F_n(\overline{D}) \right) \leq
\limsup_{n\to \infty} \adeg \left( \overline{M}_n(\overline{D}) \mid F_n(\overline{D}) \right) \\
& \leq \sum_{i=1}^l \limsup_{n\to \infty} a_{ni} \adeg \left( \rest{\overline{M}_n(\overline{D})}{C_i}  \right) =
\sum_{i=1}^l a_i \limsup_{n\to \infty} \adeg \left( \rest{\overline{M}_n(\overline{D})}{C_i}  \right) \\
& = \sum_{i \in I } a_i \limsup_{n\to \infty} \adeg \left( \rest{\overline{M}_n(\overline{D})}{C_i}  \right) \leq
\sum_{i \in I } a_i \adeg \left( \rest{\overline{M}_{\infty}(\overline{D})}{C_i}  \right) = 0.
\end{align*}
Hence the theorem follows.
\end{proof}

Finally let us consider Fujita's approximation theorem on an arithmetic surface.

\begin{Proposition}
\label{prop:Fujita:app}
We assume that $\overline{D}$ is $C^0$-type  and  $\avol(\overline{D}) > 0$.
Then, for any $\epsilon > 0$, there is $\overline{A} \in \aDiv_{C^{\infty}}(X)_{\RR}$ such that
\[
\text{$\overline{A}$ is nef}, \quad \overline{A} \leq \overline{D}\quad\text{and}\quad \avol(\overline{A}) \geq \avol(\overline{D}) - \epsilon.
\]
\end{Proposition}

\begin{proof}
By using the continuity of $\avol$, we can find a sufficiently small positive number $\delta$ such that
\[
\avol(\overline{D} - (0, \delta)) > \max \{ \avol(\overline{D}) - \epsilon, 0 \}.
\]
Let $\overline{D} - (0, \delta) = \overline{P}_{\delta} + \overline{N}_{\delta}$ be the Zariski decomposition of $\overline{D} - (0, \delta)$.
Since $\overline{P}_{\delta}$ is a big arithmetic $\RR$-divisor of $C^0$-type, by Theorem~\ref{thm:cont:upper:envelope},
there is an $F_{\infty}$-invariant continuous function $u$ on $X(\CC)$ such that
$0 \leq u < \delta$ on $X(\CC)$ and $\overline{P}_{\delta} + (0, u)$ is nef and of $C^{\infty}$-type.
If we set $\overline{A} = \overline{P}_{\delta} + (0, u)$,
then $\overline{A} \leq \overline{D}$ and
\[
\avol(\overline{D}) - \epsilon < \avol(\overline{D} - (0, \delta)) \leq \avol(\overline{A}).
\]
\end{proof}

\begin{Remark}
\label{rem:Zariski:decomp:pseudo:effective}
It is expected that the converse of (2) in Proposition~\ref{prop:non:empty:Zariski:decomp} holds, that is,
if $\overline{D}$ is of $C^0$-type and $\overline{D}$ is pseudo-effective, then $(-\infty, \overline{D}] \cap \aNef(X)_{\RR} \not= \emptyset$.
\end{Remark}

\begin{Remark}
\label{rem:Zariski:decomp:pseudo:effective:bis}
We assume that $\overline{D}$ is of $C^0$-type, big and not nef.
Let $\overline{D} = \overline{P} + \overline{N}$ be the Zariski decomposition of $\overline{D}$ and
let $N = c_1 C_1 + \cdots + c_l C_l$ be the decomposition such that
$c_1, \ldots, c_l \in \RR_{>0}$ and $C_1, \ldots, C_l$ are $1$-dimensional closed integral subschemes on $X$.
Then $C_1, \ldots, C_l$ are not necessarily linearly independent in $\operatorname{Pic}(X) \otimes_{\ZZ} \QQ$
(cf. Remark~\ref{rem:C0:C:infty}).
\query{$\otimes \QQ$\quad
$\Longrightarrow$\quad $\otimes_{\ZZ} \QQ$
(17/October/2010)}
\end{Remark}

\subsection{Examples of Zariski decompositions on $\PP^1_{\ZZ}$}
\label{subsec:Zariski:decomp:P:1}
\setcounter{Theorem}{0}

Let $\PP^1_{\ZZ} = \Proj(\ZZ[x,y])$, $C_0 = \{ x = 0 \}$, $C_{\infty} = \{ y = 0 \}$ and $z = x/y$.
Let $\alpha$ and $\beta$ be positive real numbers.
We set
\[
D = C_0,\quad g = -\log \vert z \vert^2 + \log \max \{ \alpha^2 \vert z \vert^2, \beta^2\}\quad\text{and}\quad
\overline{D} = (D,  g).
\]
The purpose of this subsection is to show the following fact:

\begin{Proposition}
\label{prop:Zariski:decomp:P:1}
The Zariski decomposition of $\overline{D}$ exists if and only if
either $\alpha \geq 1$ or $\beta \geq 1$. Moreover, we have the following:
\begin{enumerate}
\renewcommand{\labelenumi}{(\arabic{enumi})}
\item
If $\alpha \geq 1$ and $\beta \geq 1$, then $\overline{D}$ is nef.

\item
If $\alpha \geq 1$ and $\beta < 1$, then
the positive part of $\overline{D}$ is given by
\[
(\theta C_0, -\theta \log \vert z \vert^2 + \log \max \{ \alpha^2 \vert z\vert^{2\theta}, 1 \}),
\]
where $\theta = \log\alpha/(\log\alpha - \log \beta)$.

\item
If $\alpha < 1$ and $\beta \geq 1$, then
the positive part of $\overline{D}$ is given by
\[
(C_0 - (1-\theta')C_{\infty}, -\log \vert z \vert^2 + \log \max \{ \vert z \vert^{2\theta'}, \beta^2\}),
\]
where $\theta' = \log\beta/(\log\beta - \log \alpha)$.
\end{enumerate}
\end{Proposition}

\begin{proof}
Let us begin with the following claim:

\begin{Claim}
\label{claim:example:Zariski:decomp:1}
For $a, b, \lambda  \in \RR_{>0}$, we set 
\[
L = \lambda C_0,\quad
h = -\lambda \log \vert z \vert^2 + \log \max \{ a^2 \vert z \vert^{2\lambda}, b^2 \}\quad\text{and}\quad
\overline{L} = (L, h).
\]
Then we have the following:
\begin{enumerate}
\renewcommand{\labelenumi}{(\alph{enumi})}
\item
$\overline{L}$ is an arithmetic $\RR$-divisor of $(C^0 \cap \Tpsh)$-type.
In additions, $\overline{L}$ is effective if and only if $a \geq 1$.

\item
$H^0(\PP^1_{\ZZ}, L) = \bigoplus_{i \in \ZZ, 0 \leq i \leq \lambda} \ZZ z^{-i}$.

\item
For $i \in \ZZ$ with $0 \leq i \leq \lambda$, 
${\displaystyle
\Vert z^{-i} \Vert_h = \frac{1}{a^{1 - i/\lambda} b^{i/\lambda}}}$.

\item
For $s = \sum_{0 \leq i \leq \lambda} c_i z^{-i} \in H^0(\PP^1_{\ZZ}, L)$,
\[
\Vert s \Vert_h \geq \sqrt{\sum_{0 \leq i \leq \lambda} \left( \frac{c_i}{a^{1 - i/\lambda} b^{i/\lambda}} \right)^2}.
\]

\item
$\aH(\PP^1_{\ZZ}, \overline{L}) = \{ 0 \}$ if $a < 1$ and $b < 1$.

\item
$\overline{L}$ is nef if and only if $a \geq 1$ and $b \geq 1$.

\item
$\overline{L}$ is adequate if $a^2 > 2^{\lambda}$ and $b^2 >2^{\lambda}$.
\end{enumerate}
\end{Claim}

\begin{proof}
(a) and (b) are obvious. (c) is a straightforward calculation. (e) follows from (d).
Let us see (d) , (f) and (g).

(d) Indeed,
\begin{align*}
\Vert s \Vert_{h} & \geq \sup_{ | \zeta | = (b/a)^{\frac{1}{\lambda}}} \left\{ | s  |_{h} (\zeta) \right\}  =
\frac{1}{a} \sup_{ | \zeta | = (b/a)^{\frac{1}{\lambda}}} \left\{ \left| \sum_{0 \leq i \leq \lambda} c_i \zeta^{- i} \right| \right\} \\
& \geq \frac{1}{a}  \sqrt{ \int_0^{1} \left|  \sum_{0 \leq i \leq \lambda} c_i\left((b/a)^{\frac{1}{\lambda}} \exp(2\pi\sqrt{-1}t)\right)^{-i}   \right|^2 dt } \\
& =  \frac{1}{a} \sqrt{\sum_{0 \leq i,j \leq \lambda} \int_0^{1} c_i c_j
(b/a)^{\frac{- i-j}{\lambda}} \exp(2\pi\sqrt{-1}(j-i)t) dt} \\
& = \sqrt{\sum_{0 \leq i \leq \lambda} \left( \frac{c_i }{a^{1 - i/\lambda} b^{i/\lambda}}\right)^2}.
\end{align*}

(f)
It is easy to see that
$\adeg (\rest{\overline{L}}{C_0}) = \log(b)$ and $\adeg (\rest{\overline{L}}{C_{\infty}}) = \log(a)$.
For $\gamma \in \overline{\QQ}$,
let $C_{\gamma}$ be the $1$-dimensional closed  integral subscheme of $\PP^1_{\ZZ}$ given by the Zariski closure 
of $\{ (\gamma : 1) \}$.
Then
\[
\adeg (\rest{\overline{L}}{C_\gamma}) \geq \sum_{\sigma \in C_{\gamma}(\CC)} \left(- \lambda \log \vert \sigma(\gamma) \vert + \log \max \{ a \vert \sigma(\gamma) \vert^{\lambda}, b \}
\right).
\]
Thus (f) follows.

(g)
We choose $\delta \in \RR_{>0}$ such that $a^2 \geq (2(1+\delta))^{\lambda}$ and $b^2 \geq (2(1+\delta))^{\lambda}$.
Then, as
\ifpxfont
\[
\lambda \log((1+\delta)\vert z \vert^2 + (1 + \delta)) \leq \lambda \log \max\{ 2(1+\delta)\vert z \vert^2, 2(1+\delta) \} 
\leq \log \max \{ a^2 \vert z \vert^{2\lambda}, b^2 \},
\]
\else
\begin{multline*}
\lambda \log((1+\delta)\vert z \vert^2 + (1 + \delta)) \leq \lambda \log \max\{ 2(1+\delta)\vert z \vert^2, 2(1+\delta) \} \\
\leq \log \max \{ a^2 \vert z \vert^{2\lambda}, b^2 \},
\end{multline*}
\fi
we have
\[
\lambda(C_0, -\log \vert z \vert^2 + \log((1+\delta)\vert z \vert^2 + (1 + \delta))) \leq \overline{L}.
\]
Note that
$(C_0, -\log \vert z \vert^2 + \log((1+\delta)\vert z \vert^2 + (1 + \delta)))$ is ample.
Thus (g) follows.
\end{proof}

Next we claim the following:

\begin{Claim}
\label{claim:example:Zariski:decomp:1:bis}
If $\alpha < 1$ and $\beta < 1$, then the Zariski decomposition of $\overline{D}$ does not exists.
\end{Claim}

\begin{proof}
For $t > 0$, we set
\[
\overline{D}_t = (C_0, -\log \vert z \vert^2 + \log \max \{ t^2 \alpha^2 \vert z \vert^2, t^2\beta^2\}).
\]
It is easy to see that
\[
a \overline{D}_{t_1} + b \overline{D}_{t_2} = (a+b)\overline{D}_{(t_1^at_2^b)^{\frac{1}{a+b}}}
\]
for $t_1, t_2 \in \RR_{> 0}$ and $a, b \in \RR_{>0}$.
Moreover, by (g) in Claim~\ref{claim:example:Zariski:decomp:1},
$\overline{D}_{t_0}$ is adequate if $t_0 \gg 1$.
We assume that the Zariski decomposition of $\overline{D}$ exists.
Let $\overline{P}$ be the positive part of $\overline{D}$.
We choose $\epsilon > 0$ such that $t_0^{\frac{\epsilon}{1+\epsilon}} \alpha < 1$ and $ t_0^{\frac{\epsilon}{1+\epsilon}} \beta < 1$.
$\overline{P} + \epsilon \overline{D}_{t_0}$ is adequate by Proposition~\ref{prop:ample:plus:nef:ample}.
Thus, by Proposition~\ref{prop:ample:RR:div},
\[
\avol\left(\overline{D}_{t_0^{\frac{\epsilon}{1+\epsilon}}}\right) =
\frac{\avol\left((1+ \epsilon)\overline{D}_{(t_0^{\epsilon})^{\frac{1}{1+\epsilon}}} \right)}{(1+\epsilon)^2} = 
\frac{\avol(\overline{D} + \epsilon \overline{D}_{t_0})}{(1+\epsilon)^2} \geq \frac{\avol(\overline{P} + \epsilon \overline{D}_{t_0})}{(1+\epsilon)^2}  > 0,
\]
which yields 
a contradiction by virtue of (e) in Claim~\ref{claim:example:Zariski:decomp:1}.
\end{proof}

By the above claim,
it is sufficient to see (1), (2) and (3).
(1) follows from (f) in Claim~\ref{claim:example:Zariski:decomp:1}.

\medskip
(2) In this case, $\overline{D}$ is effective. Thus
the Zariski decomposition of $\overline{D}$ exists.
First we assume that $\alpha > 1$, so that $0 < \theta < 1$ and $\alpha^{1-\theta}\beta^{\theta} = 1$.
Let us see the following claim:

\begin{Claim}
\label{claim:example:Zariski:decomp:2}
$\langle \aH(\PP^1_{\ZZ}, n\overline{D}) \rangle_{\ZZ} = \bigoplus_{i \in \ZZ, 0 \leq i \leq n\theta} \ZZ z^{-i}$.
\end{Claim}

\begin{proof}
By (c) in Claim~\ref{claim:example:Zariski:decomp:1}, $\Vert z^{-i} \Vert_{ng} = \beta^{\frac{n\theta - i}{1-\theta}}$.
Thus $z^{-i} \in \aH(\PP^1_{\ZZ}, n\overline{D})$ for $0 \leq i \leq n\theta$.
For $s = \sum_{i=0}^n a_i z^{-i} \in H^0(\PP^1_{\ZZ}, nD)$, by (d) in Claim~\ref{claim:example:Zariski:decomp:1},
\[
\Vert s \Vert_{ng} \geq \sqrt{\sum_{i=0}^n \left( | a_i | \beta^{\frac{n\theta - i}{1-\theta}}\right)^2}
\]
Thus, if $\Vert s \Vert_{ng} \leq 1$, then $a_i = 0$ for $i > n\theta$, which means that
$s \in \bigoplus_{0 \leq i \leq n\theta} \ZZ z^{-i}$.
\end{proof}

\begin{Claim}
\label{claim:example:Zariski:decomp:3}
$\overline{D}$ is big and
\[
\mu_C(\overline{D}) =
\begin{cases}
1 - \theta & \text{if $C = C_0$}, \\
0 & \text{if $C \not= C_0$}
\end{cases}
\]
for a $1$-dimensional closed  integral subscheme $C$ of $\PP^1_{\ZZ}$.
\end{Claim}

\begin{proof}
Note that $(z^{-i}) + nD = (n-i)C_0 + i C_{\infty}$. Thus
the second assertion follows from Claim~\ref{claim:example:Zariski:decomp:2}.
Let us see that $\overline{D}$ is big.
We set 
\[
S_n = \left.\left\{ \sum_{0 \leq i \leq n\theta/3} a_i z^{-i} \ \right|\  \vert a_i \vert \leq \beta^{\frac{-i}{1-\theta}} \right\}.
\]
It is easy to see that $S_n \subseteq \aH(\PP^1_{\ZZ}, n\overline{D})$
for $n \gg 1$.
Note that, for $M \in \RR_{\geq 0}$,
\[
\# \{ a \in \ZZ \mid \vert a \vert \leq M \} = 2 \lfloor M \rfloor + 1 \geq \lfloor M \rfloor + 1 \geq M.
\]
Therefore
\[
\#(S_n) \geq \prod_{0 \leq i \leq n\theta/3} \beta^{\frac{-i}{1-\theta}} = \beta^{\frac{-1}{1-\theta}\frac{\lfloor n\theta/3 \rfloor (\lfloor n\theta/3 \rfloor + 1)}{2}},
\]
which implies
\[
\ah(X, n\overline{D}) \geq \log \#(S_n) \geq \frac{-\log \beta}{1-\theta} \frac{\lfloor n\theta/3 \rfloor (\lfloor n\theta/3 \rfloor + 1)}{2}
\]
for $n \gg 1$, and hence $\avol(\overline{D}) > 0$.
\end{proof}

We set
\[
P' = \theta C_0,\quad p' = -\theta \log \vert z \vert^2 + \log \max \{ \alpha^2 \vert z\vert^{2\theta}, 1 \}\quad\text{and}\quad
\overline{P}' = (P', p').
\]
By Claim~\ref{claim:example:Zariski:decomp:3} and Claim~\ref{claim:thm:asym:orthogonal:1} in the proof of Theorem~\ref{thm:asym:orthogonal},
if $\overline{P} = (P, p)$ is the positive part of the Zariski decomposition, then
$P = \theta C_0$. Let us see that $\overline{P}' = \overline{P}$.
First of all, $\overline{P}' \leq \overline{D}$ and $\overline{P}'$ is nef by (f) in Claim~\ref{claim:example:Zariski:decomp:1}. 
Thus $\overline{P}' \leq \overline{P}$, and hence
there is a continuous function $u$ such that $u \geq 0$ and $\overline{P} = \overline{P}' + (0, u)$.
Note that
\[
0 \leq u \leq   -(1-\theta)\log \vert z \vert^2 + \log \max \{ \alpha^2 \vert z \vert^2, \beta^2 \} - \log \max \{ \alpha^2 \vert z \vert^{2\theta}, 1 \}.
\]
In particular, if $\vert z \vert \geq \beta^{\frac{1}{1-\theta}}$, then $u(z) = 0$.
As $p = -\theta \log \vert z \vert^2+ u$ on $\{ z \mid \vert z \vert < \beta^{\frac{1}{1-\theta}} \}$, $u$ is subharmonic on $\{ z \mid \vert z \vert < \beta^{\frac{1}{1-\theta}} \}$.
Thus, by the maximal principle, 
\[
u(z) \leq \sup_{\vert \zeta \vert = \beta^{\frac{1}{1-\theta}}} \{ u(\zeta) \} = 0,
\]
which implies that $u(z) = 0$ on $\{ z \mid \vert z \vert < \beta^{\frac{1}{1-\theta}} \}$.
Therefore $\overline{P}' = \overline{P}$.

\medskip
Finally let us consider the case where $\alpha = 1$. Let $\overline{P}$ be the positive part of $\overline{D}$.
For $t \in (1, 1/\beta)$, we set 
\[
\overline{D}_t = (C_0, -\log \vert z \vert^2 + \log \max \{ t^2 \vert z \vert^2, t^2\beta^2\})
\]
as in the proof of Claim~\ref{claim:example:Zariski:decomp:1:bis}.
Then $\overline{D} \leq \overline{D}_t$ and, by the previous observation, the positive part $\overline{P}_t$ of $\overline{D}_t$
is given by
\[
\overline{P}_t = (\theta_t C_0, -\theta_t \log \vert z \vert^2 + \log \max \{ t^2 \vert z\vert^{2\theta_t}, 1 \}),
\]
where $\theta_t = \log t/(-\log \beta)$. Therefore, $(0,0) \leq \overline{P} \leq \overline{P}_t$, and hence $\overline{P} = (0,0)$
as $t \to 1$.

\medskip
(3) If we set $\overline{D}'' = \overline{D} - \widehat{(z)}$,
then $\overline{D}'' = (C_{\infty}, -\log \vert w \vert^2 + \log \max \{ \beta^2 \vert w \vert^2, \alpha^2 \})$,
where $w = y/x$.
Thus, in the same way as (2), we can see that the positive part of $\overline{D}''$ is
\[
(\theta' C_{\infty}, -\theta' \log \vert w \vert^2 + \log \max \{ \beta^2 \vert w \vert^{2\theta'}, 1 \}),
\]
where $\theta' = \log\beta/(\log\beta - \log \alpha)$, so that the positive part of $\overline{D} = \overline{D}'' +  \widehat{(z)}$ is
\[
(C_0 - (1-\theta')C_{\infty}, -\log \vert z \vert^2 + \log \max \{ \vert z \vert^{2\theta'}, \beta^2\})
\]
by Proposition~\ref{prop:easy:prop:Zariski:decomp}.
\end{proof}

\begin{Remark}
\label{rem:C0:C:infty}
Let us choose $\alpha,\alpha',\beta,\beta' \in \RR_{>0}$ such that $\alpha \geq 1$, $\alpha' \geq 1$,
$\alpha\beta' < 1$ and $\alpha'\beta < 1$.
We set 
\[
M = C_0 + C_{\infty},\quad
\varphi = -\log \vert z \vert^2 + \log \max \{ \alpha^2\vert z \vert^2, \beta^2 \} + \log \max \{ {\alpha'}^2, {\beta'}^2 \vert z \vert^2 \}
\]
and $\overline{M} = (M, \varphi)$, that is,
\query{$h$\quad
$\Longrightarrow$\quad
$\varphi$
(09/July/2010)}
\[
\varphi = \begin{cases}
-\log \vert z \vert^2 + \log (\alpha' \beta)^2 & \text{if $\vert z \vert \leq \beta/\alpha$},\\
\log (\alpha\alpha')^2 & \text{if $\beta/\alpha \leq \vert z \vert \leq \alpha'/\beta'$},\\
\log \vert z \vert^2 + \log (\alpha\beta')^2 & \text{if $\vert z \vert \geq \alpha'/\beta'$}.
\end{cases}
\]
It is easy to see that $\overline{M}$ is an effective arithmetic divisor of $(C^0 \cap \Tpsh)$-type and that
\[
\adeg(\rest{\overline{M}}{C_0}) = \log (\alpha' \beta)\quad\text{and}\quad
\adeg(\rest{\overline{M}}{C_{\infty}}) = \log (\alpha\beta').
\]
If we set
\[
\vartheta = \frac{\log \alpha + \log \alpha'}{\log \alpha - \log \beta},\quad
\vartheta' = \frac{\log \alpha + \log \alpha'}{\log \alpha' - \log \beta'}
\]
and
\[
\psi = -\vartheta \log \vert z \vert^2 + \log \max \{ \alpha^2 \vert z \vert^{2\vartheta}, {\alpha'}^{-2} \}
+ \log \max \{ {\alpha'}^2 , {\alpha}^{-2} \vert z \vert^{2\vartheta'}\},
\]
that is,
\[
\psi = \begin{cases}
-\vartheta \log \vert z \vert^2 & \text{if $\vert z \vert \leq \beta/\alpha$},\\
\log (\alpha\alpha')^2 & \text{if $\beta/\alpha \leq \vert z \vert \leq \alpha'/\beta'$},\\
\vartheta' \log \vert z \vert^2 & \text{if $\vert z \vert \geq \alpha'/\beta'$},
\end{cases}
\]
then 
the positive part of $\overline{M}$ is 
\[
(\vartheta C_0 + \vartheta'C_{\infty}, \psi).
\]
This can be checked in the similar way as Proposition~\ref{prop:Zariski:decomp:P:1}.
For details, we leave it to the readers.
In the case where $\alpha = \alpha' = 1$,
the negative part of $\overline{M}$ is $\overline{M}$ itself,
which means that the support of the negative part contains $C_0$ and $C_{\infty}$ despite
$C_0 - C_{\infty} = (z)$.
\query{$\widehat{(z)}$\quad
$\Longrightarrow$\quad
$(z)$
(09/July/2010)}
This example also show that if the positive parts of 
$\overline{D}$ and
$\overline{D}'$ are $\overline{P}$ and $\overline{P}'$ respectively,
then the positive part of $\overline{D} + \overline{D}'$ is not
necessarily $\overline{P} + \overline{P}'$.
\end{Remark}

\begin{Remark}
\label{rem:Zariski:Faltings}
Let $\lambda$ be a positive real number.
We set
\[
\phi_\lambda = -\log \vert z \vert^2 + \log(\vert z \vert^2 + \lambda)\quad\text{and}\quad
\overline{M}_\lambda = (C_0, \phi_\lambda).
\]
We denote $\overline{M}_1$ by $\overline{L}$, that is,
$\overline{L} = (C_0, -\log \vert z \vert^2 + \log(\vert z \vert^2 + 1))$.
It is easy to see that $\overline{M}_\lambda$ is an arithmetic divisor of $(C^{\infty} \cap \Tpsh)$-type,
$\adeg(\overline{M}_\lambda^2) = (\log(\lambda) + 1)/2$
and that $\overline{M}_\lambda$ is nef for $\lambda \geq 1$. 
In particular, 
$\overline{M}_\lambda$ is big for $\lambda \geq 1$. 

From now on, we fix $\lambda$ with $0 < \lambda < 1$.
By using an inequality:
\[
\log(1 + \lambda x) \geq \lambda \log(1 + x)\quad(x \in \RR_{\geq 0}),
\]
we can see that $\lambda\overline{L} \leq \overline{M}_\lambda$, which means that $\overline{M}_\lambda$ is big.
On the other hand,
\[
\adeg(\rest{\overline{M}_\lambda}{C_0}) = \log(\lambda) < 0,
\]
so that
$\overline{M}_\lambda$ is not nef.
We set 
\[
\Phi_\lambda = dd^c(\log (\vert z \vert^2 + \lambda)) = \frac{\lambda}{2\pi\sqrt{-1}(\vert z \vert^2 + \lambda)^2} dz \wedge d\bar{z},
\]
which gives rise to an $F_{\infty}$-invariant volume form on $\PP^1(\CC)$ with $\int_{\PP^1(\CC)}\Phi_\lambda = 1$.
Moreover, we set
\[
\aDiv_{\Phi_\lambda}(\PP^1_{\ZZ})_{\RR} = \left\{ (A, g_A) \left| \begin{array}{l}
\text{(1)  $A$ is an $\RR$-divisor on $\PP^1_{\ZZ}$.} \\
\text{(2) $g_A$ is an $F_{\infty}$-invariant $A$-Green function of $C^{\infty}$-type} \\
\text{\phantom{(2) }on $\PP^1(\CC)$ such that $dd^c([g_A]) + \delta_A = (\deg(A))\Phi_\lambda$}.
\end{array}\right\}\right.,
\]
which is the Arakelov Chow group consisting of admissible metrics with respect to $\Phi_{\lambda}$ due to Faltings \cite{Fal}.
Let us see that the set
\[
\{ (A, g_A) \in \aDiv_{\Phi_\lambda}(\PP^1_{\ZZ})_{\RR} \mid
\text{$(A, g_A)$ is nef and $(0,0) \leq (A, g_A) \leq \overline{M}_\lambda$} \}
\]
have only one element $(0,0)$. 

Indeed, let $\overline{A} = (A, g_A)$ be an element of the above set.
Then there are constants $a, b$ such that $0 \leq a \leq 1$ and $\overline{A} = a\overline{M}_\lambda + (0, b)$.
Since $g_A \leq \phi_\lambda$, we have $b \leq (1-a)\phi_\lambda$.
Thus $b \leq 0$ because $\phi_\lambda(\infty) = 0$.
In addition,
\[
\adeg(\rest{\overline{A}}{C_0}) = a \log(\lambda) + b \geq 0.
\]
In particular, $b \geq 0$, so that $b = 0$, and hence $a \log(\lambda) \geq 0$.
Thus $a = 0$.

This example shows that the 
Arakelov Chow group consisting of admissible metrics 
is insufficient to get
the Zariski decomposition.

Finally note that $\lambda\overline{L}$ is not necessarily the positive part of $\overline{M}_\lambda$ because
$\avol(\overline{M}_\lambda) \geq (\log(\lambda) + 1)/2$ (cf. Theorem~\ref{thm:gen:Hodge:index}),
$\avol(\lambda\overline{L}) = \lambda^2/2$ and $(\log(\lambda) + 1)/2 > \lambda^2/2$ for $0 < 1 - \lambda \ll 1$.
\end{Remark}

\begin{Remark}
Let $n$ be a positive integer and $f \in \RR[T]$ such that
$\deg(f) = 2n$ and $f(t) > 0$ for all $t \in \RR_{\geq 0}$.
It seems to be not easy to find the positive part of 
\[
\left(nC_0, -n \log \vert z \vert^2 + \log f(\vert z \vert)\right)
\]
on $\PP^1_{\ZZ}$.
\end{Remark}

\end{document}